\documentclass[10pt,reqno]{amsart}

\usepackage[latin2]{inputenc}
\usepackage{amsmath}
\usepackage{graphicx}
\usepackage{amssymb}
\usepackage{color}
\usepackage{amsthm}
\usepackage{enumitem}
\usepackage{mathrsfs}
\usepackage[english]{babel}
\usepackage[linktocpage=true,colorlinks=true,linkcolor=Blue,citecolor=Green]{hyperref}

\newtheorem{theorem}{Theorem}
\newtheorem{proposition}[theorem]{Proposition}
\newtheorem{lemma}[theorem]{Lemma}
\newtheorem{corollary}[theorem]{Corollary}

\newtheorem*{theorem*}{Theorem}

\theoremstyle{definition}
\newtheorem{definition}{Definition}
\newtheorem{remark}[theorem]{Remark}

\def\Xint#1{\mathchoice
{\XXint\displaystyle\textstyle{#1}}%
{\XXint\textstyle\scriptstyle{#1}}%
{\XXint\scriptstyle\scriptscriptstyle{#1}}%
{\XXint\scriptscriptstyle\scriptscriptstyle{#1}}%
\!\int}
\def\XXint#1#2#3{{\setbox0=\hbox{$#1{#2#3}{\int}$ }
\vcenter{\hbox{$#2#3$ }}\kern-.6\wd0}}

\def\dashint{\Xint-}

\definecolor{Yellow}{rgb}{0.95,0.9,0.0} 
\definecolor{Red}{rgb}{0.8,0.1,0.1}
\definecolor{Green}{rgb}{0.1,0.65,0.2}
\definecolor{Blue}{rgb}{0.1,0.1,0.8}
\definecolor{Purple}{rgb}{0.7,0.1,0.7}
\definecolor{Grey}{rgb}{0.6,0.6,0.6}

\newcommand{\supp}{\operatorname{supp}}

\newcommand{\dist}{\operatorname{dist}}

\DeclareMathOperator*{\argmin}{arg\,min}

\newcommand{\Rd}[1][d]{{\mathbb{R}^{#1}}}
\allowdisplaybreaks[4]

\newcommand{\dH}{d\mathcal{H}^{d-1}}
\newcommand{\Sph}{\mathbb{S}^{d-1}}
\newcommand{\ddt}{\frac{d}{dt}}

\newcommand{\p}{\partial}
\renewcommand{\vec}[1]{{\operatorname{#1}}}
\newcommand{\n}{\vec{n}}

\newcommand {\jump}[1] {[\![ #1 ]\!]}

\newcommand{\wkto}{\rightharpoonup}
\newcommand{\wksto}{\overset{*}{\rightharpoonup}}
\def\Xint#1{\mathchoice
{\XXint\displaystyle\textstyle{#1}}%
{\XXint\textstyle\scriptstyle{#1}}%
{\XXint\scriptstyle\scriptscriptstyle{#1}}%
{\XXint\scriptscriptstyle\scriptscriptstyle{#1}}%
\!\int}
\def\XXint#1#2#3{{\setbox0=\hbox{$#1{#2#3}{\int}$ }
\vcenter{\hbox{$#2#3$ }}\kern-.6\wd0}}
\def\intbar{\Xint-}

\begin{document}

\title[Weak solutions of Mullins--Sekerka flow as gradient flows]
{Weak solutions of Mullins--Sekerka flow as a Hilbert space gradient flow}

\author{Sebastian Hensel}
\address{Hausdorff Center for Mathematics, Universit{\"a}t Bonn, Endenicher Allee 62, 53115 Bonn, Germany}
\email{sebastian.hensel@hcm.uni-bonn.de}

\author{Kerrek Stinson}
\address{Hausdorff Center for Mathematics, Universit{\"a}t Bonn, Endenicher Allee 62, 53115 Bonn, Germany}
\email{kerrek.stinson@hcm.uni-bonn.de}

\begin{abstract} 
We propose a novel weak solution
theory for the Mullins--Sekerka equation primarily motivated from a gradient flow perspective.
Previous existence results on weak solutions due to Luckhaus and Sturzenhecker
(Calc.\ Var.\ PDE~3, 1995) or R\"{o}ger ({SIAM} J. Math. Anal.~37, 2005)
left open the inclusion of both a sharp energy dissipation principle and
a weak formulation of the contact angle at the intersection of the interface
and the domain boundary. To incorporate these, we introduce a functional framework encoding a weak
solution concept for Mullins--Sekerka flow essentially relying only on
\textit{i)}~a single sharp energy dissipation inequality in the
spirit of De~Giorgi, and \textit{ii)}~a weak formulation 
for an arbitrary fixed contact angle through a distributional
representation of the first variation of the underlying capillary energy.
Both ingredients are intrinsic to the interface of the evolving phase indicator
and an explicit distributional PDE formulation with potentials can be derived from them.
Existence of weak solutions is established via subsequential
limit points of the naturally associated minimizing movements scheme. 
Smooth solutions are consistent with the classical Mullins--Sekerka flow, 
and even further, we expect our solution concept to be amenable, at least in principle,
to the recently developed relative entropy approach for curvature driven interface evolution.



\end{abstract}


\maketitle
\vspace*{-.7em}\tableofcontents

\section{Introduction}

\subsection{Context and motivation}
The purpose of this paper is to develop the gradient flow perspective for the Mullins--Sekerka equation
at the level of a weak solution theory. 
The Mullins--Sekerka equation is a curvature driven evolution equation for a mass preserved quantity,
see~\eqref{eq:MullinsSekerka1}--\eqref{eq:MullinsSekerka6} below.
The ground breaking results of the early 90s showed that when strong solutions exist, this equation is in fact the sharp interface limit of the 
Cahn--Hilliard equation, a fourth order diffuse interface model for phase separation in materials \cite{Alikakos1994} (see also, e.g., \cite{batesFife1993}, \cite{Pego1989FrontMI}). 
However, as for mean curvature flows, one of the critical challenges in studying such sharp interface 
models is the existence of solutions after topological change. As a result, 
many different weak solution concepts
have been developed, and in analogy with the development of weak solution theories for PDEs and the 
introduction of weak function spaces, a variety of weak notions of smooth surfaces have been 
applied for solution concepts. In the case that a surface arises as the common boundary of 
two sets (i.e., an interface), a powerful solution concept has been the $BV$~solution,
first developed for the Mullins--Sekerka flow in the seminal work 
of Luckhaus and Sturzenhecker~\cite{LucStu}. 

For $BV$~solutions in the sense of~\cite{LucStu}, the evolving phase
is represented by a time-dependent family of characteristic functions which are of bounded variation.
Furthermore, both the evolution equation for the phase and the Gibbs--Thomson law are
satisfied in a distributional form. The corresponding existence result for such solutions crucially
leverages the well-known fact that the Mullins--Sekerka flow can be formally
obtained as an $H^{-1}$-type gradient flow of the perimeter functional (see, e.g., \cite{Garcke2013}). 
Indeed, $BV$~solutions for the Mullins--Sekerka
flow are constructed in~\cite{LucStu} as subsequential limit points of the associated minimizing
movements scheme. However, due to the discontinuity of the first variation of the perimeter functional with respect to
weak-$*$ convergence in~$BV$, Luckhaus and Sturzenhecker~\cite{LucStu} relied on the
additional assumption of convergence of the perimeters in order to obtain a $BV$~solution
in their sense. Based on geometric measure theoretic results of Sch\"{a}tzle~\cite{Schaetzle2001} and a pointwise interpretation of the Gibbs--Thomson law in terms of a generalized curvature intrinsic to the interface~\cite{Roeger2004}, R\"{o}ger~\cite{Roeger2005} was 
later able to remove the energy convergence assumption (see also \cite{AbelsRoeger}).

However, the existence results of Luckhaus and Sturzenhecker~\cite{LucStu}
and R\"{o}ger~\cite{Roeger2005} still leave two fundamental questions unanswered. First, both weak
formulations of the Gibbs--Thomson law do not encompass a weak formulation
for the boundary condition of the interface where it intersects the domain boundary. For instance, if the energy is proportional to the surface area of the interface, one expects a constant
ninety degree contact angle condition at the intersection points, which quantitatively accounts for the fact that minimizing energy in the bulk, the surface will travel the shortest path to the boundary. 
Second, neither of the two works establishes a sharp energy dissipation
principle, which, because of the formal gradient flow
structure of the Mullins--Sekerka equation, is a natural ingredient
for a weak solution concept as we will further discuss below. 
A second motivation to prove a sharp energy dissipation inequality
stems from its crucial role in the recent progress concerning weak-strong
uniqueness principles for curvature driven interface evolution problems
(see, e.g., \cite{Fischer2020c}, \cite{Fischer2020a} or~\cite{Hensel2021l}).

Turning to approximations of the Mullins--Sekerka flow via the Cahn--Hilliard equation
Chen \cite{Chen} introduced an alternative weak solution concept, which does include an energy dissipation inequality. To prove existence, Chen developed powerful estimates (that have been used in numerous applications, e.g., \cite{abelsLengeler2014}, \cite{AbelsRoeger}, \cite{melchionnaRocca2017}) to control the sign of the discrepency measure, an object which captures the distance of a solution from an equipartition of energy. Critically these estimates do not rely on the maximum principle and are applicable to the fourth-order Cahn--Hilliard equation. However, in contrast to Ilmanen's proof for the convergence of the Allen--Cahn equation to 
mean curvature flow~\cite{ilmanen}, where the discrepancy vanishes in the limit, Chen is restricted to proving non-positivity in the limit. As a result, the proposed solution concept requires a varifold lifting of the energy for the dissipation inequality and a modified varifold for the Gibbs--Thomson relation. In the interior of the domain, the modified Gibbs--Thomson relation no longer implies the pointwise interpretation of the evolving surface's curvature in terms of the trace of the chemical potential and, on the boundary, cannot account for the contact angle. Further, Chen's solution concept does not use the optimal dissipation inequality to capture the full dynamics of the gradient flow. 

Looking to apply the framework of evolutionary Gamma-convergence 
developed by Sandier and Serfaty~\cite{Sandier2004} to the convergence of the Cahn--Hilliard equation, 
Le~\cite{Le2008} introduces a gradient flow solution concept for the Mullins--Sekerka equation, which principally relies on an optimal dissipation inequality. However, interpretation of the limiting interface as a solution in this sense requires that the surface is regular and does not intersect the domain boundary, i.e., there is no contact angle. 
As noted by Serfaty~\cite{Serfaty2011}, though the result of Le~\cite{Le2008} sheds light 
on the gradient flow structure of the Mullins--Sekerka flow in a smooth setting, 
it is of interest to develop a general framework for viewing solutions of the Mullins--Sekerka flow 
as curves of maximal slope even on the level of a weak solution theory. 
This is one of the primary contributions of the present work.  

Though still in the spirit
of the earlier works by Le \cite{Le2008}, Luckhaus and Sturzenhecker~\cite{LucStu},
and R\"{o}ger~\cite{Roeger2005}, the solution concept we introduce includes both
a weak formulation for the constant contact angle and a sharp energy
dissipation principle. The boundary condition for the interface is 
in fact not only implemented 
for a constant contact angle $\alpha = \smash{\frac{\pi}{2}}$ but even 
for general constant contact angles $\alpha \in (0,\pi)$.
For the formulation of the energy dissipation inequality, we exploit
a gradient flow perspective encoded
in terms of a De~Giorgi type inequality. Recall to this end that for smooth
gradient flows, the gradient flow equation $\dot{u} = - \nabla E[u]$
can equivalently be represented by the inequality 
$$E[u(T)] + \int_{0}^T \frac{1}{2}|\dot{u}(t)|^2 
+ \frac{1}{2}|\nabla E[u(t)]|^2 \,dt \leq E[u(0)]$$
(for a discussion of gradient flows and their solution concepts in further detail see Subsection \ref{subsec:SandierSerfaty}).
Representation of gradient flow dynamics through the above dissipation inequality allows one to generalize to the weak setting and is often amenable to typical variational machinery such as weak compactness and lower semi-continuity.

The main conceptual contribution of this work consists of the introduction of
a functional framework for which a weak solution
of the Mullins--Sekerka flow is essentially characterized through only \textit{i)} a single sharp
energy dissipation inequality,
and \textit{ii)} a weak formulation for the contact angle condition
in the form of a suitable distributional representation of the first variation of the energy.
We emphasize that both these ingredients are intrinsic to the trajectory of the evolving phase indicator.
Beyond proving existence of solutions via a minimizing movements scheme (Theorem \ref{theo:mainResult}), we show that our solution concept extends Le's \cite{Le2008} to the weak setting (Subsection \ref{subsec:relationToLe}), a more classical distributional PDE formulation with potentials can be derived from it (Lemma \ref{lem:PDEinterpretation}), smooth solutions are consistent with the classical Mullins-Sekerka equation (Lemma \ref{lem:propertiesVarifoldSolutions}), and that the underlying varifold for the energy is of bounded variation (Proposition \ref{prop:RoegerInterpretation}).

A natural question arising from the present work is whether solutions
of the Cahn--Hilliard equation converge subsequentially to weak solutions
of the Mullins--Sekerka flow in our sense, which would improve the seminal
result of Chen~\cite{Chen} that relies on a (much) weaker formulation
of the Mullins--Sekerka flow. 
An investigation of this question will be the subject of a future work.

\subsection{Mullins--Sekerka motion law: Strong PDE formulation}
\label{subsec:strongFormulation}
\begin{subequations}
Let $d \geq 2$ and let $\Omega \subset \Rd$ be a bounded domain with 
orientable and $C^2$-boundary~$\p\Omega$. Consider also a finite
time horizon $T_{*} \in (0,\infty)$ and let $\mathscr{A}=(\mathscr{A}(t))_{t \in [0,T_{*})}$
be a time-dependent family of smoothly evolving open subsets $\mathscr{A}(t) \subset \Omega$ with
$\p\mathscr{A}(t)=\overline{\p^*\mathscr{A}(t)}$
and $\mathcal{H}^{d-1}(\p\mathscr{A}(t)\setminus\p^*\mathscr{A}(t))=0$, $t \in [0,T_{*})$, where $\partial^*\mathscr{A}$ refers to the reduced boundary \cite{AmbrosioFuscoPallara}. 
Denoting for every $t \in [0,T_{*})$
by $V_{\partial\mathscr{A}(t)}$ and $H_{\partial\mathscr{A}(t)}$ the associated normal velocity
and mean curvature vector, respectively, the family~$\mathscr{A}$ is said to evolve
by \emph{Mullins--Sekerka flow} if for each $t \in (0,T_{*})$ there exists
a chemical potential $\overline{u}(\cdot,t)$ so that
\begin{align}
\label{eq:MullinsSekerka1}
\Delta \overline{u}(\cdot,t) &= 0
&& \text{in } \Omega \setminus \p\mathscr{A}(t),
\\ \label{eq:MullinsSekerka2}
V_{\partial\mathscr{A}(t)} &= - \big(n_{\partial\mathscr{A}(t)}\cdot
\jump{\nabla\overline{u}(\cdot,t)}\big)n_{\partial\mathscr{A}(t)}
&& \text{on } \p\mathscr{A}(t) \cap \Omega,
\\ \label{eq:MullinsSekerka3}
c_0H_{\partial\mathscr{A}(t)} &= \overline{u}(\cdot,t) n_{\partial\mathscr{A}(t)}
&& \text{on } \p\mathscr{A}(t) \cap \Omega,
\\ \label{eq:MullinsSekerka5}
(n_{\p\Omega} \cdot \nabla) \overline{u}(\cdot,t) &= 0
&&\text{on } \p\Omega \setminus \overline{\partial \mathscr{A}(t)\cap \Omega}.
\end{align}
\end{subequations}

Here, we denote by $c_0 \in (0,\infty)$ a fixed surface tension
constant, by $n_{\partial\mathscr{A}(t)}$ the unit normal vector field
along~$\p\mathscr{A}(t)$ pointing inside the phase~$\mathscr{A}(t)$,
and similarly $n_{\p\Omega}$ is the inner normal on the domain boundary~$\p\Omega$.
Furthermore, the jump~$\jump{\cdot}$ across the interface~$\p\mathscr{A}(t) \cap \Omega$
in normal direction is understood to be oriented such that the signs in the following
integration by parts formula are correct:
\begin{equation}
\label{eq:defJump}
\begin{aligned}
-\int_{\Omega\setminus\p\mathscr{A}(t)} \eta\nabla\cdot v \,dx
&= \int_{\Omega} \nabla\eta\cdot v \,dx + 
\int_{\p\mathscr{A}(t)\cap\Omega} \eta \big(n_{\partial\mathscr{A}(t)} \cdot \jump{v}\big) \,\dH
\\&~~~
+ \int_{\p\Omega} \eta \big(n_{\p\Omega}\cdot v\big) \,\dH
\end{aligned}
\end{equation}
for all sufficiently regular functions 
$v \colon\overline{\Omega}\to\Rd$ and $\eta \colon \overline{\Omega} \to \mathbb{R}$.

For sufficiently smooth evolutions, it is a straightforward 
exercise to verify that the Mullins--Sekerka flow conserves
the mass of the evolving phase as
\begin{align}
\label{eq:massPreservation}
\ddt \int_{\mathscr{A}(t)} 1 \,dx =
- \int_{\p\mathscr{A}(t) \cap \Omega} 
V_{\partial\mathscr{A}(t)} \cdot n_{\partial\mathscr{A}(t)} \,\dH
= 0.
\end{align}
To compute the change of interfacial surface area, we first need to
fix a boundary condition for the interface. In the present work,
we consider the setting of a fixed contact angle $\alpha \in (0,\pi)$
in the sense that for all $t \in [0,T_{*})$ it is required that
\begin{align}
\label{eq:MullinsSekerka6}
\tag{1e}
n_{\partial\Omega} \cdot n_{\partial\mathscr{A}(t)} = \cos\alpha \quad \text{ on } \partial \Omega \cap \overline{ \partial \mathscr{A}(t) \cap \Omega }.
\end{align}
Then, it is again straightforward to compute that 
\begin{equation}
\label{eq:energyIdentity}
\begin{aligned}
&\ddt \bigg(\int_{\p\mathscr{A}(t) \cap \Omega} c_0 \,\dH
+ \int_{\p\mathscr{A}(t) \cap \partial\Omega} c_0\cos\alpha \,\dH\bigg)
\\&
= - \int_{\p\mathscr{A}(t) \cap \Omega} 
V_{\partial\mathscr{A}(t)} \cdot c_0H_{\partial\mathscr{A}(t)} \,\dH
= - \int_{\Omega} |\nabla \overline{u}(\cdot,t)|^2 \,dx \leq 0.
\end{aligned}
\end{equation}
In view of the latter inequality, one may wonder whether the Mullins--Sekerka
flow can be equivalently represented as a gradient flow with respect to interfacial 
surface energy. That this is indeed possible is of course a classical observation (see \cite{Garcke2013} and references therein) and,
at least for smooth evolutions, may be realized in terms of a suitable $H^{-1}$-type
metric on a manifold of smooth surfaces. 

\subsection{Gradient flow perspective assuming smoothly evolving geometry}\label{subsec:SandierSerfaty}
To take advantage of the insight provided by (\ref{eq:energyIdentity}), we recall two methods for gradient flows.
In parallel, our approach is inspired by De~Giorgi's methods for curves of maximal slope in 
metric spaces and the approach for Gamma-convergence of evolutionary equations developed by 
Sandier and Serfaty in~\cite{Sandier2004}, which has been applied to the Cahn--Hilliard approximation
of the Mullins--Sekerka flow by Le~\cite{Le2008}. 

Looking to the school of thought inspired by De Giorgi (see \cite{Ambrosio2005} and references therein), in a generic 
metric space $(X,d)$ equipped with energy $E:X \to \mathbb{R}\cup \{\infty\}$, a curve $t \mapsto u(t) \in X$ is said to be a solution 
of the differential inclusion $-\ddt u \in \partial E[u]$ 
if it is a curve of maximal slope, that is, it satisfies the optimal dissipation relation
\begin{equation}
\label{eqn:deGiorgiDissipation}
E[u(T)] + \frac{1}{2}\int_{0}^T \Big|\ddt u\Big|^2 + |\partial E [u]|^2 \, dt \leq E[u(0)]
\end{equation} for almost all $T\in (0,T_*)$, where
$|\ddt u|$ is interpreted in the metric sense and 
\begin{equation}\nonumber
|\partial E[u]|   := \limsup_{v \to u} \frac{(E[v]-E[u])_{+}}{d(v,u)}.
\end{equation} 
One motivation for this solution concept is in the Banach 
setting where, for sufficiently nice energies $E$, the optimal 
dissipation (\ref{eqn:deGiorgiDissipation}) is equivalent to solving 
the differential inclusion \cite{Ambrosio2005}. 

The energy behind the gradient flow structure of the Mullins--Sekerka flow 
is the perimeter functional, for which we have the classical result of 
Modica~\cite{Modica87number2} (see also~\cite{ModicaMortola})
\begin{equation*}
E_\epsilon [u]  := \int_{\Omega} \frac{1}{\epsilon}W(u) 
+ \epsilon \|\nabla u\|^2 \, dx 
\wkto_{\Gamma} c_0{\rm Per}_{\Omega} (\chi) =:E[\chi],
\end{equation*}
thereby making the perspective of Sandier and Serfaty \cite{Sandier2004} relevant. 
Abstractly, given $\Gamma$-converging (see, e.g., \cite{BraidesLocalMinNotes}, \cite{DalMasoBook}) 
energies $E_ \epsilon \wkto_{\Gamma} E$, 
this approach gives conditions for when a curve $t \mapsto u(t) \in Y$, which is the limit 
of $t \mapsto u_\epsilon(t) \in X_\epsilon$ solving $-\ddt u_{\epsilon} 
\in \nabla_{X_\epsilon} E_\epsilon[u_\epsilon]$, is a solution the 
gradient flow $-\ddt u \in \nabla_Y E[u]$ associated with the limiting energy. Specifically, this 
requires the lower semi-continuity of the time derivative and the variations given by 
\begin{align*}
\int_{0}^{T} \Big\|\ddt u\Big\|^2_{Y} \,dt
&\leq \liminf_{\epsilon\downarrow 0} \int_{0}^{T} \Big\|\ddt u_\epsilon\Big\|^2_{X_\epsilon} \,dt,
\\
 \int_{0}^{T} \big\|\nabla_Y E[u]\big\|^2_{Y} \,dt
&\leq \liminf_{\epsilon\downarrow 0} \int_{0}^{T} 
\big\|\nabla_{X_\epsilon} E_\epsilon[u_\epsilon]\big\|^2_{X_\epsilon} \,dt,
\end{align*}
 which are precisely the relations needed to maintain an optimal dissipation inequality~(\ref{eqn:deGiorgiDissipation}) in the limit. 
We note that this idea was precisely developed in finite dimensions with $C^1$-functionals, 
and extending this approach to geometric evolution equations seems to require 
re-interpretation in general.
 
This process of formally applying the Sandier--Serfaty approach to the Cahn--Hilliard equation 
was carried out by Le in \cite{Le2008} (see also \cite{le2010} and \cite{roccaScala2017}). As the Cahn-Hilliard equation and Mullins-Sekerka flow are mass preserving, it is necessary to introduce the Sobolev space $H^{1}_{(0)} := H^1(\Omega)\cap \{u: \intbar u\, dx = 0\}$ with dual $H^{-1}_{(0)}$. Then, for a set $A\subset \Omega$ with 
$\Gamma := \partial A \cap \Omega$ a piecewise Lipschitz surface, Le recalls the 
space $H^{1/2}_{(0)}(\Gamma)$, the trace space of $H^1(\Omega\setminus \Gamma)$ with constants quotiented out, and introduces a norm with Hilbert structure given by 
\begin{equation}\label{eqn:LeTraceSpace}
\| f\|_{H^{1/2}_{(0)}(\Gamma)} = \sqrt{(f,f)_{H^{1/2}_{(0)}(\Gamma)}} := \|\nabla \tilde f\|_{L^2(\Omega)},
\end{equation}
where~$\tilde f$  satisfies the Dirichlet problem
 \begin{equation}\label{def:LeDirichletProb}
 \begin{aligned}
 -\Delta  \tilde f = 0 \text{ in }\Omega\setminus \Gamma, \quad	
  \tilde f = f \text{ on }\Gamma.
 \end{aligned}
\end{equation}  
Additionally, $H^{-1/2}_{(0)}(\Gamma)$ is the naturally associated dual space
with a Hilbert space structure induced by the corresponding Riesz isomorphism. 

With these concepts, Le shows that in the smooth setting the Mullins--Sekerka flow 
is the gradient flow of the perimeter functional on a formal Hilbert manifold 
with tangent space given by~$H^{-1/2}_{(0)}(\Gamma)$, which for a characteristic function $u(t)$ with interface~$\Gamma_t := \partial \{u(t) {=} 1\}$ can summarily be written as
\begin{equation}\label{eqn:leMS}
-\ddt u(t) \in \nabla_{H^{-1/2}_{(0)}(\Gamma_t)} E[u(t)].
\end{equation}  
Further, solutions~$u_\epsilon$ of the Cahn--Hilliard equation
 \begin{equation*}
 \begin{aligned}
 \partial_t u_\epsilon  &= \Delta v_\epsilon \quad\text{where}\quad 
 v_\epsilon = \delta E_\epsilon [u_\epsilon] = \frac{1}{\epsilon} f'(u_\epsilon) - \epsilon \Delta u_\epsilon\quad \text{ in } \Omega, \\
 (n_{\partial\Omega} \cdot \nabla )u_\epsilon &= 0 \quad \text{ and }\quad
 (n_{\partial\Omega} \cdot \nabla) v_\epsilon = 0 \quad \text{ on } \partial \Omega,
 \end{aligned}
 \end{equation*}
 are shown to converge to a trajectory $t \mapsto u(t) \in BV(\Omega ; \{0,1\})$, such 
that if the evolving surface $\Gamma_t $ 
is~$C^3$ in space-time, then $u$ is a solution of the Mullins-Sekerka flow (\ref{eqn:leMS}) 
in the sense that
  \begin{align*}
\int_{0}^{T} \Big\|\ddt u(t)\Big\|^2_{H^{-1/2}_{(0)}(\Gamma_t)} \,dt
&\leq \liminf_{\epsilon\downarrow 0} \int_{0}^{T} \Big\|\ddt u_\epsilon(t)\Big\|^2_{H^{-1}_{(0)}(\Omega)} \,dt,
\\
\int_{0}^{T} \big\|H_{\Gamma_t} \big\|^2_{H^{1/2}_{(0)}(\Gamma_t)} \,dt 
&= \int_{0}^{T} \big\|\nabla_{H^{-1/2}_{(0)}(\Gamma_t)} E[u(t)]\big\|^2_{H^{-1/2}_{(0)}(\Gamma_t)} \,dt\\
&\leq \liminf_{\epsilon\downarrow 0} \int_{0}^{T} \big\|\nabla_{H^{-1}_{(0)}(\Omega)} 
E_\epsilon[u_\epsilon(t)]\big\|^2_{H^{-1}_{(0)}(\Omega)} \,dt,
\end{align*}
where $H_\Gamma$ is the scalar mean curvature of a sufficiently regular
surface~$\Gamma$. 
As developed by Le, interpretation of the left-hand side of the above inequalities is only possible for regular $\Gamma.$ In the next section, we will introduce function spaces and a solution concept that allow us to extend these quantities to the weak setting.


\section{Main results and relation to previous works}
So as not to waylay the reader, we first introduce in Subsection~\ref{subsec:weakSol}
a variety of function spaces necessary for our weak solution concept 
and then state our main existence theorem. 
Further properties of the associated solution space,
an interpretation of our solution concept from
the viewpoint of classical PDE theory (i.e., in terms of associated chemical potentials),
as well as further properties of the time-evolving oriented varifolds associated
with solutions which are obtained as limit points of the natural minimizing
movements scheme are presented in Subsection~\ref{subsec:furtherProperties}.
In Subsection~\ref{subsec:functionSpaces}, we return to a discussion of the function spaces
introduced in Subsection~\ref{subsec:weakSol} to further 
illuminate the intuition behind their choice. 
We then proceed in Subsection~\ref{subsec:relationToLe} with a discussion relating 
our functional framework to the one introduced by Le~\cite{Le2008} for the smooth setting.
In Subsection~\ref{subsec:PDEperspective}, we finally take 
the opportunity to highlight the potential of our framework
in terms of the recent developments concerning weak-strong uniqueness for curvature
driven interface evolution.

\subsection{Weak formulation: Gradient flow structure and existence result}\label{subsec:weakSol}
At the level of a weak formulation, we will 
describe the evolving interface, arising as the boundary of a phase region, in terms of a time-evolving
family of characteristic functions of bounded variation. This strongly motivates us to formulate the
gradient flow structure over a manifold of $\{0,1\}$-valued
$BV$ functions in~$\Omega$.
To this end, let $d \geq 2$ and let $\Omega \subset \Rd$ be a bounded domain with 
orientable $C^2$ boundary~$\p\Omega$. Fixing the mass to be $m_0 \in (0,\mathcal{L}^d(\Omega))$, 
we define the ``manifold''
\begin{align}
\label{def:manifold}
\mathcal{M}_{m_0} := \Big\{\chi \in BV(\Omega;\{0,1\}) \colon
\int_{\Omega} \chi \,dx = m_0 \Big\}.
\end{align} 

For the definition of the associated energy functional~$E$ on $\mathcal{M}_{m_0}$, 
recall that we aim to include contact point dynamics with fixed contact 
angle in this work. Hence, in addition to an isotropic interfacial energy
contribution in the bulk, we also incorporate a capillary contribution.
Precisely, for a fixed set of three positive surface tension constants $(c_0,\gamma_+,\gamma_-)$
we consider an interfacial energy~$E[\chi]$, $\chi \in \mathcal{M}_{m_0}$, of the form
\begin{align*}
\int_{\Omega} c_0 \,d |\nabla\chi|
+ \int_{\p\Omega} \gamma_+\chi \,\dH
+ \int_{\p\Omega} \gamma_-(1{-}\chi) \,\dH,
\end{align*}
where by an abuse of notation we do not distinguish
between~$\chi$ and its trace along~$\p\Omega$.
Furthermore, the surface tension constants are assumed to satisfy
Young's relation $|\gamma_+ {-} \gamma_{-}| < c_0$ so that there
exists an angle $\alpha \in (0,\pi)$ such that
\begin{align}
\label{eq:YoungsLaw}
(\cos\alpha) c_0 = \gamma_+ - \gamma_-.
\end{align}
For convenience, we will employ the following convention:
switching if needed the roles of the sets indicated by~$\chi$ and~$1{-}\chi$,
we may assume that $\gamma_- < \gamma_+$ and hence $\alpha \in (0,\smash{\frac{\pi}{2}}]$.
In particular, by subtracting a constant, we may work with the following
equivalent formulation of the energy functional on $\mathcal{M}_{m_0}$:
\begin{align}
\label{def:energy}
E[\chi] := \int_{\Omega} c_0 \,d|\nabla\chi|
+ \int_{\p\Omega} (\cos\alpha) c_0 \chi \,\dH,
\quad \chi \in \mathcal{M}_{m_0}.
\end{align}
As usual in the context of weak formulations for curvature driven interface evolution problems,
it will actually be necessary to work with a suitable (oriented) varifold relaxation of~$E$.
We refer to Definition~\ref{def:VarifoldSolution} below for details in this direction.

In order to encode a weak solution of the Mullins--Sekerka equation as a Hilbert space gradient flow
with respect to the interfacial energy~$E$, it still remains to introduce the associated
Hilbert space structure. To this end, we first introduce a class of regular test functions, 
which give rise to infinitesimally volume preserving inner variations, denoted by 
\begin{equation}\label{def:variationVelocities}
\mathcal{S}_{\chi} : = \left\{B\in C^1(\overline{\Omega};\Rd)\colon 
\int_{\Omega} \chi \nabla \cdot B  \, dx = 0, \, B \cdot n_{\partial\Omega} = 0 
\text{ on } \partial \Omega\right\}.
\end{equation}
As in Subsection \ref{subsec:SandierSerfaty}, we recall the Sobolev space of functions with mass-average zero given by $H^1_{(0)} := \{u \in H^1(\Omega) \colon \dashint_{\Omega} u \,dx = 0\}$ with norm $\|u\|_{H^1_{(0)}} := \|\nabla u\|_{L^2(\Omega)}$
and dual $H^{-1}_{(0)} := (H^{1}_{(0)})^*$.
Based on the test function space~$\mathcal{S}_\chi$,
we can introduce the space $\mathcal{V}_{\chi}\subset H^{-1}_{(0)}$ 
as the closure of regular mass preserving normal velocities 
generated on the interface associated with~$\chi \in \mathcal{M}_{m_0}$:
\begin{equation}\label{def:surfaceVelocities}
\mathcal{V}_{\chi} : = \overline{\big\{B \cdot \nabla \chi \colon 
B\in \mathcal{S}_\chi \big\}}^{H^{-1}_{(0)}}
\subset H^{-1}_{(0)},
\end{equation}
where $B \cdot \nabla \chi$ acts on elements $u \in H^1_{(0)}$ 
in the distributional sense, i.e., recalling that $B \cdot n_{\partial\Omega} = 0$
along $\partial\Omega$ for $B \in \mathcal{S}_{\chi}$ we have
\begin{equation}
\label{eq:actionVel}
\langle B\cdot \nabla \chi, u \rangle_{H^{-1}_{(0)}, H^{1}_{(0)}} 
:= - \int_\Omega \chi \nabla \cdot (uB)\, dx.
\end{equation}
The space~$\mathcal{V}_\chi$ carries a Hilbert space structure directly
induced by the natural Hilbert space structure of~$H^{-1}_{(0)}$.
The latter in turn is induced by the inverse $\Delta_{N}^{-1}$ of the weak Neumann 
Laplacian~$\Delta_{N}\colon H^1_{(0)} \to H^{-1}_{(0)}$
(which for the Hilbert space $H^1_{(0)}$ is in fact nothing else but the associated
Riesz isomorphism) in the form of
\begin{align}\label{eqn:h-1HilbertStruct}
( F,\widetilde F )_{H^{-1}_{(0)}}
:= \int_{\Omega} \nabla \Delta_{N}^{-1}(F) 
\cdot \nabla \Delta_{N}^{-1}(\widetilde F) \,dx
\quad\text{for all } F,\widetilde F \in H^{-1}_{(0)},
\end{align}
so that we may in particular define
\begin{align}
\label{def:normVel}
\|F\|_{\mathcal{V}_\chi}^2 := \|F\|_{H^{-1}_{(0)}}^2 
= ( F, F )_{H^{-1}_{(0)}},
\quad F \in \mathcal{V}_\chi.
\end{align}
We remark that operator norm on $H^{-1}_{(0)}$ is recovered from the inner product 
in~(\ref{eqn:h-1HilbertStruct}).
For the Mullins--Sekerka flow, the space~$\mathcal{V}_\chi$ is the natural space 
associated with the action of the first variation (i.e., the gradient) of the interfacial energy
on~$\mathcal{S}_\chi$, see~\eqref{def:MSmetricSlope1} in Definition~\ref{def:VarifoldSolution} below.

In view of the Sandier--Serfaty perspective on Hilbert space gradient
flows, cf.\ Subsection~\ref{subsec:SandierSerfaty},
it would be desirable to capture the time derivative of a 
trajectory $t \mapsto \chi(\cdot,t) \in \mathcal{M}_{m_0}$ within the same bundle of Hilbert spaces.
However, given the a priori lack of regularity of weak solutions, 
it will be necessary to introduce a second space of velocities~$\mathcal{T}_\chi$
(containing the space~$\mathcal{V}_\chi$) which can be thought of as a maximal 
tangent space of the formal manifold; this is given by 
\begin{equation}\label{def:tangentVelocities} 
\mathcal{T}_\chi := \overline{\big\{\mu \in H^{-1}_{(0)} \cap \mathrm{M}(\Omega)\colon 
{\rm supp}\, \mu \subset \supp|\nabla\chi|\big\}}^{H^{-1}_{(0)}}
\subset H^{-1}_{(0)},
\end{equation}
where $\mathrm{M}(\Omega)$ denotes the space of Radon measures on $\Omega$.
Both spaces $\mathcal{V}_\chi$ and $\mathcal{T}_{\chi}$ are spaces of velocities, 
and from the PDE perspective, associated with these will be spaces 
for the (chemical) potential. We will discuss this and quantify the separation 
between $\mathcal{V}_{\chi}$ and $\mathcal{T}_\chi$ in Subsection~\ref{subsec:functionSpaces}. 
However, despite the necessity to work with two spaces, 
we emphasize that our gradient flow solution concept still only requires 
use of the above formal metric/manifold structure and the above energy functional.

\begin{definition}[Varifold solutions of Mullins--Sekerka flow as curves of maximal slope]
\label{def:VarifoldSolution}
Let $d \in \{2,3\}$, consider a finite time horizon $T_{*} \in (0,\infty)$,
and let $\Omega \subset \Rd$ be a bounded domain with orientable 
$C^2$ boundary~$\p\Omega$. For a locally compact and separable metric space~$X$,
we denote by~$\mathrm{M}(X)$ the space of finite Radon measures on~$X$.
Fix $\chi_0 \in \mathcal{M}_{m_0}$
and define the associated oriented varifold $$\mu_0 := \mu_0^{\Omega} + \mu_0^{\p\Omega}
\in \mathrm{M}(\overline{\Omega} {\times} \Sph)$$ 
by $$\mu_0^{\Omega} := c_0\,|\nabla\chi_0| \llcorner \Omega
\otimes (\smash{\delta_{\frac{\nabla\chi_0}{|\nabla\chi_0|}(x)}})_{x \in \Omega}$$ and
$$\mu_0^{\p\Omega} := (\cos\alpha)c_0\chi_0\,\mathcal{H}^{d-1} \llcorner \p\Omega
\otimes (\smash{\delta_{n_{\p\Omega}(x)}})_{x\in\p\Omega}.$$

A measurable map $\chi\colon\Omega{\times}(0,T_*)\to\{0,1\}$ together
with $\mu \in \mathrm{M}((0,T_{*}) {\times} \overline{\Omega} {\times} \Sph)$
is called a \emph{varifold solution for Mullins--Sekerka 
flow~\emph{\eqref{eq:MullinsSekerka1}--\eqref{eq:MullinsSekerka6}} 
with time horizon~$T_*$ and initial data~$(\chi_0,\mu_0)$} if:
\begin{subequations}
\begin{itemize}[leftmargin=0.7cm]
\item[i)] \emph{(Structure and compatibility)} It holds that
					$$\chi \in L^\infty(0,T_{*};\mathcal{M}_{m_0}) \cap C([0,T_*);H^{-1}_{(0)}(\Omega))$$
					with ${\rm Tr}|_{t=0}\chi = \chi_0$ in $H^{-1}_{(0)}$.
					Furthermore, $\mu = \mathcal{L}^1 \llcorner (0,T_{*})
					\otimes \{\mu_t\}_{t \in (0,T_{*})}$, and for each $t \in (0,T_{*})$
					the oriented varifold~$\mu_t \in \mathrm{M}(\overline{\Omega} {\times} \Sph)$
					decomposes as $\mu_t = \mu_t^{\Omega} + \mu_t^{\partial\Omega}$
					for two separate oriented varifolds given in their disintegrated form by
					\begin{equation}\label{eq:compSlice1}					
					\mu_t^{\Omega} =: |\mu_t^{\Omega}|_{\mathbb{S}^{d-1}}
					\otimes (\lambda_{x,t})_{x \in \overline{\Omega}}
					\in \mathrm{M}(\overline{\Omega} {\times} \Sph)
					\end{equation}					
					 and
					\begin{equation}\label{eq:compSlice2}
					\mu_t^{\p\Omega} =: |\mu_t^{\p\Omega}|_{\mathbb{S}^{d-1}}
					\otimes (\smash{\delta_{n_{\p\Omega}(x)}})_{x\in\p\Omega}
					\in \mathrm{M}(\p\Omega {\times} \Sph).
					\end{equation}
					Finally, we
					require that these oriented varifolds contain the interface associated 
					with the phase modelled by~$\chi$ in the sense of
					\begin{align}
					\label{eq:comp1}
					c_0\,|\nabla\chi(\cdot,t)| \llcorner \Omega 
					&\leq |\mu_t^{\Omega}|_{\mathbb{S}^{d-1}} \llcorner \Omega,
					\\ \label{eq:comp2}
					(\cos\alpha)c_0\chi(\cdot,t)\,\mathcal{H}^{d-1} \llcorner \p\Omega 
					&\leq |\mu_t^{\p\Omega}|_{\mathbb{S}^{d-1}}
					{+} |\mu_t^{\Omega}|_{\mathbb{S}^{d-1}} \llcorner \p\Omega
					\end{align}
					for almost every $t \in (0,T_{*})$.
\item[ii)] \emph{(Generalized mean curvature)} For almost every~$t \in (0,T_*)$
					 there exists a function~$H_\chi(\cdot,t)$ such that 
					 \begin{equation}
					 \label{eq:integrabilityWeakMeanCurvature}
					 H_\chi(\cdot,t) \in L^s\big(\Omega;d|\nabla\chi(\cdot,t)|\llcorner\Omega\big)
					 \end{equation}
					 where $s \in [2,4]$ if $d = 3$ and
					 $s \in [2,\infty)$ if $d=2$, and $H_{\chi(\cdot,t)}$
					 is the generalized mean curvature vector of~$\supp|\nabla\chi(\cdot,t)| \llcorner \Omega$
					 in the sense of R\"{o}ger~\cite[Definition~1.1]{Roeger2004}.
					 Moreover, the first variation~$\delta\mu_t$
					 of~$\mu_t$ in the direction of a volume preserving
					 inner variation~$B\in\mathcal{S}_{\chi(\cdot,t)}$ is given by
					 \begin{align}
					 \label{eq:repTangentialFirstVariation}
					 \delta \mu_t(B) = - \int_{\Omega} c_0 H_\chi(\cdot,t) 
					 \frac{\nabla\chi(\cdot,t)}{|\nabla\chi(\cdot,t)|} \cdot B \,d|\nabla\chi(\cdot,t)|.
					 \end{align}
\item[iii)] \emph{(Mullins--Sekerka motion law as a sharp energy dissipation inequality)}
					 For almost every $T \in (0,T_{*})$, it holds that
					 \begin{align}
					 ~~~~~~&E[\mu_T] {+} \frac{1}{2} \int_{0}^{T} 
					 \big\|(\partial_t\chi)(\cdot,t)\big\|_{\mathcal{T}_{\chi(\cdot,t)}}^2 
					 + \big|\partial E [\mu_t]\big|^2_{\mathcal{V}_{\chi(\cdot,t)}} \, dt   
					 \leq E[\mu_0], 
					 \label{eq:DeGiorgiInequality}
					 \end{align}
					 where we define by a slight abuse of notation, but still in the spirit of 
					 the usual metric slope \`{a} la De~Giorgi (cf.\ \eqref{def:MSmetricSlope2} 
					 and \eqref{eq:minMovMetricSlope} below),
					 \begin{equation}\label{def:MSmetricSlope1}
					 ~~~\frac{1}{2}\big|\partial E [\mu]\big|^2_{\mathcal{V}_{\chi}} := 
					 \sup_{B \in \mathcal{S}_{\chi}} 
					 \bigg\{ \delta\mu\big(B\big) - \frac{1}{2}  
					 \big\|B\cdot\nabla\chi\big\|^2_{\mathcal{V}_{\chi}} \,dt \bigg\},
					 \end{equation}
					 and where the energy functional
					 on the varifold level is given by the total mass measure associated
					 with the oriented varifold~$\mu_t$, i.e.,
					 \begin{equation}
					 \label{def:varifoldEnergy}
					 E[\mu_t] 
					 := |\mu_t|_{\mathbb{S}^{d-1}}(\overline{\Omega})
					 = |\mu_t^{\Omega}|_{\mathbb{S}^{d-1}}(\overline{\Omega}) 
					    + |\mu_t^{\p\Omega}|_{\mathbb{S}^{d-1}}(\p\Omega).
					 \end{equation}
\end{itemize}
\end{subequations}

\begin{subequations}
Finally, we call~$\chi$ a \emph{$BV$~solution for
evolution by Mullins--Sekerka flow~\emph{\eqref{eq:MullinsSekerka1}--\eqref{eq:MullinsSekerka6}} 
with initial data~$(\chi_0,\mu_0)$} if there exists 
$\mu = \mathcal{L}^1 \llcorner (0,T_{*}) \otimes \{\mu_t\}_{t \in (0,T_{*})}$
such that $(\chi,\mu)$ is a varifold solution in the above sense and
the varifold~$\mu$ is given by the canonical lift of~$\chi$, i.e., for almost every $t \in (0,T_*)$
it holds that
\begin{align}
\label{eq:BVsol1}
|\mu_t|_{\mathbb{S}^{d-1}} &= c_0\,|\nabla\chi(\cdot,t)| \llcorner \Omega
+ (\cos\alpha)c_0\chi(\cdot,t)\,\mathcal{H}^{d-1} \llcorner \p\Omega,
\\
\label{eq:BVsol2}
\lambda_{x,t} &= \delta_{\frac{\nabla\chi(\cdot,t)}{|\nabla\chi(\cdot,t)|}(x)} \quad\text{for } 
(|\nabla\chi(\cdot,t)| \llcorner \Omega)\text{-almost every } x \in \Omega.
\end{align}
\end{subequations}
\end{definition}

Before we state the main existence result of this work,
let us provide two brief comments on the above definition.
First, we note that in Lemma~\ref{lem:PDEinterpretation} we show that if $(\chi,\mu)$ 
is a varifold solution to Mullins--Sekerka flow in the sense of
Definition~\ref{def:VarifoldSolution}, then it is also a solution from a 
more typical PDE perspective.
Second, to justify the notation of (\ref{def:MSmetricSlope1}), 
we refer the reader to Lemma~\ref{lem:metricSlope} where it is shown that if 
in addition the relation~\eqref{eq:repTangentialFirstVariation} is
satisfied, it holds that
\begin{equation*}
\big|\partial E [\mu]\big|_{\mathcal{V}_\chi}  
= \sup_{\Psi}
\limsup_{s \to 0}
\frac{(E[\mu\circ\Psi_s^{-1}] - E[\mu])_{+}}{\|\chi\circ\Psi_s^{-1} - \chi\|_{H^{-1}_{(0)}}},
\end{equation*}
where the supremum runs over all one-parameter families of diffeomorphisms $s\mapsto\Psi_s\in 
C^1\text{-}\mathrm{Diffeo}(\overline{\Omega},\overline{\Omega})$
which are differentiable in an open neighborhood of the origin and
further satisfy $\Psi_0=\mathrm{Id}$, $\int_{\Omega} \chi\circ\Psi_s^{-1} \,dx = m_0$
and $\partial_s\Psi_s|_{s=0} = B\in\mathcal{S}_\chi$.
Note that the relation $\partial_s(\chi\circ\Psi_s^{-1})|_{s=0} {+} (B\cdot\nabla)\chi = 0$ 
enforced by the chain rule, $(\chi\circ\Psi_s^{-1})|_{s=0} = \chi$
as well as $\partial_s\Psi_s^{-1}|_{s=0}=-\partial_s\Psi_s|_{s=0}=-B$ motivates us to consider $\mathcal{V}_{\chi}$ as the tangent 
space for the formal manifold at~$\chi\in\mathcal{M}_{m_0}$.

\begin{theorem}[Existence of varifold solutions of Mullins--Sekerka flow]
\label{theo:mainResult}
Let $d \in \{2,3\}$, $T_* \in (0,\infty)$, and $\Omega\subset\Rd$ be
a bounded domain with orientable $C^2$ boundary $\p\Omega$.
Let $m_0 \in (0,\mathcal{L}^d(\Omega))$, $\chi_0 \in \mathcal{M}_{m_0}$,
$c_0 \in (0,\infty)$, $\alpha \in (0,\smash{\frac{\pi}{2}}]$, and
let $\mu_0 \in \mathrm{M}(\overline{\Omega}{\times}\mathbb{S}^{d-1})$
be the associated oriented varifold to this data, cf.\ Definition~\ref{def:VarifoldSolution}.

Then, there exists a varifold solution for Mullins--Sekerka 
flow~\emph{\eqref{eq:MullinsSekerka1}--\eqref{eq:MullinsSekerka6}} 
with initial data~$(\chi_0,\mu_0)$ in the sense of Definition~\ref{def:VarifoldSolution}.

In fact, each limit point of the minimizing movements scheme
associated with the Mullins--Sekerka flow~\emph{\eqref{eq:MullinsSekerka1}--\eqref{eq:MullinsSekerka6}},
cf.\ Subsection~\ref{subsec:keyIngredientsMinMov},
is a solution in the sense of Definition~\ref{def:VarifoldSolution}.
In case of convergence of the time-integrated energies (cf.\ \eqref{eq:energyConvergenveAssumpMinMov}),
the corresponding limit point of the minimizing movements scheme is even a $BV$~solution
in the sense of Definition~\ref{def:VarifoldSolution}. 
\end{theorem}	

The proof of Theorem~\ref{theo:mainResult} is the content of 
Subsections~\ref{subsec:keyIngredientsMinMov}--\ref{subsec:proofMainResult}.

\begin{remark}
\label{rem:existencePDEPerspective}
If instead of the conditions from items~\textit{ii)} and~\textit{iii)} 
from Definition~\ref{def:VarifoldSolution} one asks for the 
existence of two potentials $u\in L^2(0,T_*;H^1_{(0)})$ 
and $w\in L^2(0,T_*;H^1)$, respectively, which 
satisfy the conditions~\eqref{eq:evolEquationPhase}, \eqref{eq:GibbsThomson},
\eqref{eq:L2control} and~\eqref{eq:DeGiorgiInequality2}
from Lemma~\ref{lem:PDEinterpretation} below,
the results of Theorem~\ref{theo:mainResult} in fact hold without any restriction
on the ambient dimenion~$d$. We will prove this fact in 
the course of Subsection~\ref{subsec:proofsFurtherProperties}.
\end{remark}

\subsection{Further properties of varifold solutions}\label{subsec:furtherProperties}
The purpose of this subsection is to collect a variety of further results
complementing our main existence result, Theorem~\ref{theo:mainResult}.
Proofs of these are postponed until Subsection~\ref{subsec:proofsFurtherProperties}.

\begin{lemma}[Interpretation as a De~Giorgi metric slope]
\label{lem:metricSlope}
Let $\chi\in BV(\Omega;\{0,1\})$
and $\mu \in \mathrm{M}(\overline{\Omega}{\times}\mathbb{S}^{d-1})$. 
Suppose in addition that the tangential first variation of $\mu$ is 
given by a curvature $H_\chi \in L^1(\Omega; |\nabla \chi|)$ in the sense of 
equation~\eqref{eq:repTangentialFirstVariation}. Then, it holds that
\begin{equation}\label{def:MSmetricSlope2}
\big|\partial E [\mu]\big|_{\mathcal{V}_\chi}  
= \sup_{\Psi} 
\limsup_{s \to 0}
\frac{(E[\mu\circ\Psi_s^{-1}] - E[\mu])_{+}}{\|\chi\circ\Psi_s^{-1} - \chi\|_{H^{-1}_{(0)}}},
\end{equation}
where the supremum runs over all one-parameter families of diffeomorphisms $s\mapsto\Psi_s\in 
C^1\text{-}\mathrm{Diffeo}(\overline{\Omega},\overline{\Omega})$
which are differentiable in a neighborhood of the origin and
further satisfy $\Psi_0=\mathrm{Id}$, $\int_{\Omega} \chi\circ\Psi_s^{-1} \,dx = m_0$
and $\partial_s\Psi_s|_{s=0} = B\in\mathcal{S}_\chi$.
Without~\eqref{eq:repTangentialFirstVariation}, the right hand side
of~\eqref{def:MSmetricSlope2} provides at least an upper bound.
\end{lemma}

Next, we aim to interpret the
information provided by the sharp energy inequality~\eqref{eq:DeGiorgiInequality}
from a viewpoint which is more in the tradition of classical PDE~theory.
More precisely, we show that~\eqref{eq:DeGiorgiInequality}
together with the representation~\eqref{eq:repTangentialFirstVariation}
already encodes the evolution equation for the evolving phase
as well as the Gibbs-Thomson law--both in terms
of a suitable distributional formulation featuring an associated potential.
We emphasize, however, that without further regularity assumptions
on the evolving geometry these two potentials may \textit{a~priori} not agree.
This flexibility is in turn a key strength of the gradient flow perspective
to allow for less regular evolutions (i.e., a weak solution theory). 

\begin{lemma}[Interpretation from a PDE~perspective]
\label{lem:PDEinterpretation}
Let $(\chi,\mu)$ be a varifold solution for Mullins--Sekerka flow 
with initial data~$(\chi_0,\mu_0)$ in the sense of Definition~\ref{def:VarifoldSolution}.
For a given $\chi\in\mathcal{M}_{m_0}$, define for each of the two
velocity spaces~$\mathcal{V}_{\chi}$ and~$\mathcal{T}_{\chi}$
an associated space of potentials via $\mathcal{G}_{\chi}:=\Delta_N^{-1}(\mathcal{V}_{\chi})
\subset H^1_{(0)}$ and $\mathcal{H}_{\chi}:=\smash{\Delta_N^{-1}(\mathcal{T}_{\chi})}
\subset H^1_{(0)}$, respectively.
\begin{itemize}[leftmargin=0.7cm]
\item[i)] There exists a potential
					$u \in L^2(0,T_{*};\mathcal{H}_{\chi(\cdot,t)}) \subset L^2(0,T_{*};H^1_{(0)})$
					such that $\partial_t\chi = \Delta u$ in $\Omega{\times}(0,T_*)$, 
					$\chi(\cdot,0) = \chi_0$ in $\Omega$, and $(n_{\p\Omega}\cdot\nabla)u = 0$
					on $\p\Omega{\times}(0,T_*)$, in the precise sense of
				  \begin{equation}
					\label{eq:evolEquationPhase}
					\begin{aligned}
					&\int_{\Omega} \chi(\cdot,T)\zeta(\cdot,T) \,dx 
					- \int_{\Omega} \chi_{0}\zeta(\cdot,0) \,dx
					\\&
					= \int_{0}^{T} \int_{\Omega} \chi \p_t\zeta \,dxdt
					- \int_{0}^{T} \int_{\Omega} \nabla u \cdot \nabla\zeta \,dxdt
					\end{aligned}
					\end{equation} 
					for almost every $T \in (0,T_*)$ and all 
					$\zeta \in C^1(\overline{\Omega} {\times} [0,T_{*}))$.
\item[ii)] There exists a potential $w \in L^2(0,T_*;H^1(\Omega))$ such that
					 for $w_0 := w {-} \dashint_{\Omega} w \,dx$ one has
					 $w_0\in L^2(0,T_{*};\mathcal{G}_{\chi(\cdot,t)}) \subset L^2(0,T_{*};H^1_{(0)})$,
					 and further satisfies the following three properties: first,
					 the Gibbs--Thomson law
					 \begin{align}
					 \label{eq:GibbsThomson}
					 \int_{\overline{\Omega} {\times} \Sph}
					 (\mathrm{Id} {-} s\otimes s) : \nabla B(x) \,d\mu_t(x,s)
					 = \int_{\Omega} \chi(\cdot,t) \nabla \cdot (w(\cdot,t)B) \,dx
					 \end{align}
					 holds true for almost every $t \in (0,T_{*})$
					 and all $B \in C^1(\overline{\Omega};\Rd)$
					 such that $(B \cdot n_{\p\Omega})|_{\p\Omega} \equiv 0$;
					 second, it holds that
					 \begin{align}
					 \label{eq:normEquivalenceGibbsThomsonPotential}
					 \int_{\Omega} \frac{1}{2}|\nabla w(\cdot,t)|^2 \,dx 
					 = \frac{1}{2}\|w_0(\cdot,t)\|^2_{\mathcal{G}_{\chi(\cdot,t)}} 
					 = \frac{1}{2}\big|\partial E [\chi(\cdot,t),\mu_t]\big|^2_{\mathcal{V}_{\chi(\cdot,t)}}
					 \end{align}
					 for almost every~$t \in (0,T_{*})$; and third, 
					 there is $C=C(\Omega,d,c_0,m_0,\chi_0) > 0$ such that
					 \begin{align}
					 \label{eq:L2control}
					 \|w(\cdot,t)\|_{H^1(\Omega)} \leq C(1 + \|\nabla w(\cdot,t)\|_{L^2(\Omega)})
					 \end{align}
					 for almost every $t \in (0,T_{*})$. 
\item[iii)] The energy dissipation inequality holds true in the sense that
						\begin{align}
						\label{eq:DeGiorgiInequality2}
						E[\mu_T]
						+ \int_{0}^{T} \int_{\Omega} \frac{1}{2}|\nabla u|^2 + \frac{1}{2}|\nabla w|^2 \,dxdt
						\leq E[\mu_0]
						\end{align}
						for almost every $T \in (0,T_*)$.
\end{itemize}
\end{lemma}

Note that in view of Proposition~\ref{prop:RoegerInterpretation}, item~ii), 
and the trace estimate~\eqref{eq:firstVariationInterior} from below, if $(\chi, \mu)$ is a varifold solution that is a limit point of the minimizing movements scheme (see Section \ref{subsec:keyIngredientsMinMov}), we may in particular 
deduce that 
\begin{equation}\label{eqn:curvature=trace}
c_0 H_\chi(\cdot,t) = w(\cdot,t)
\end{equation} for almost every $t \in (0,T_*)$
up to sets of ($|\nabla\chi(\cdot,t)| \llcorner \Omega$)-measure zero. Via similar arguments, for any varifold solution, (\ref{eqn:curvature=trace}) holds up to a constant, conceptually consistent with (\ref{eqn:LeTraceSpace}) and Subsections \ref{subsec:functionSpaces} and \ref{subsec:relationToLe}.

Next, we show subsequential compactness of our solution concept 
and consistency with classical solutions. To formulate the latter,
we make use of the notion of a time-dependent family $\mathscr{A}=(\mathscr{A}(t))_{t \in [0,T_{*})}$
of smoothly evolving subsets $\mathscr{A}(t) \subset \Omega$, $t \in [0,T_*)$.
More precisely, each set $\mathscr{A}(t)$ is open and
consists of finitely many connected components (the number of which is constant in time).
Furthermore, the reduced boundary of $\mathscr{A}(t)$ in~$\Rd$ differs from its topological
boundary only by a finite number of contact sets on~$\partial\Omega$
(the number of which is again constant in time) represented by 
$\partial(\partial^*\mathscr{A}(t) \cap \Omega)
= \partial(\partial^*\mathscr{A}(t) \cap \partial\Omega) \subset \partial\Omega$.
The remaining parts of~$\partial\mathscr{A}(t)$, i.e., $\partial^*\mathscr{A}(t) \cap \Omega$
and $\partial^*\mathscr{A}(t) \cap \partial\Omega$, are smooth manifolds with boundary
(which for both is given by the contact points manifold).

\begin{lemma}[Properties of the space of varifold solutions]
\label{lem:propertiesVarifoldSolutions}
Let the assumptions and notation of Theorem~\ref{theo:mainResult}
be in place. 
\begin{itemize}[leftmargin=0.7cm]
\item[i)] \emph{(Consistency)} Let $(\chi,\mu)$ be a varifold solution
					for Mullins--Sekerka flow in the sense of Definition~\ref{def:VarifoldSolution}
					which is smooth, i.e., $\chi(x,t) = \chi_{\mathscr{A}}(x,t) := \chi_{\mathscr{A}(t)}(x)$
					for a smoothly evolving family $\mathscr{A}=(\mathscr{A}(t))_{t \in [0,T_{*})}$.
					Furthermore, assume that~\eqref{eq:repTangentialFirstVariation} also holds
					with $\delta\mu_t$ replaced on the left hand side by $\delta E[\chi(\cdot,t)]$
					(which for a $BV$~solution does not represent an additional constraint). 
					Then, $\mathscr{A}$ is a classical solution for Mullins--Sekerka flow in the sense 
					of~\emph{\eqref{eq:MullinsSekerka1}--\eqref{eq:MullinsSekerka6}}. 
					If one assumes in addition that $\smash{\frac{1}{c_0}}\mu_t^\Omega
					\in \mathrm{M}(\overline{\Omega}{\times}\mathbb{S}^{d-1})$ 
					is an integer rectifiable oriented varifold, it also holds that
					\begin{align}
					\label{eq:identityMassMeasures1}
					c_0|\nabla\chi(\cdot,t)| \llcorner \Omega &= |\mu_t^\Omega|_{\mathbb{S}^{d-1}} \llcorner \Omega,
					&& |\mu_t^\Omega|_{\mathbb{S}^{d-1}} \llcorner \partial\Omega = 0,
					\\
					\label{eq:identityMassMeasures2}
					(\cos\alpha)c_0\chi(\cdot,t)\,\mathcal{H}^{d-1}\llcorner\partial\Omega
					&= |\mu_t^{\partial\Omega}|_{\mathbb{S}^{d-1}}
					\end{align}
					for a.e.\ $t \in (0,T_*)$.
					
					Vice versa, any classical solution~$\mathscr{A}$ of Mullins--Sekerka 
					flow~\emph{\eqref{eq:MullinsSekerka1}--\eqref{eq:MullinsSekerka6}}
					gives rise to a (smooth) $BV$~solution~$\chi = \chi_{\mathscr{A}}$ 
					in the sense of Definition~\ref{def:VarifoldSolution}.
\item[ii)] \emph{(Subsequential compactness of the solution space)} Let ${(\chi_k,\mu_k)}_{k\in\mathbb{N}}$
					 be a sequence of $BV$ solutions with initial data~$(\chi_{k,0},\mu_{k,0})$
					 and time horizon $0 < T_* < \infty$ in the sense of Definition~\ref{def:VarifoldSolution}. 
					 Assume that the associated energies $t \mapsto E[(\mu_k)_t]$ 
					 are absolutely continuous functions for all $k \in \mathbb{N}$, that 
					 $\sup_{k \in \mathbb{N}} E[\mu_{k,0}] < \infty$, and that
					 the sequence $(|\nabla\chi_{k,0}|\llcorner\Omega)_{k \in \mathbb{N}}$ 
					 is tight. Then, one may find a subsequence $\{k_n\}_{n\in\mathbb{N}}$, data~$(\chi_0,\mu_0)$,
					 and a varifold solution $(\chi,\mu)$ with initial data~$(\chi_0,\mu_0)$
					 and time horizon~$T_*$ in the sense of Definition~\ref{def:VarifoldSolution} 
					 such that $\chi_{{k_n}} \to \chi$ in $L^1(\Omega{\times}(0,T_*))$ 
					 as well as $\mu_{k_n} \wksto \mu$ in 
					 $\mathrm{M}((0,T_*){\times}\overline{\Omega}{\times}\mathbb{S}^{d-1})$ as $n \to \infty$.
\end{itemize}
\end{lemma}

We remark that the above compactness is formulated in terms of BV solutions so that 
the generalized mean curvature (R\"oger's interpretation~\cite{Roeger2004}) is recovered 
in the limit, an argument which requires the use of geometric machinery developed 
by Sch\"atzle for varifolds with first variation given in terms of a Sobolev 
function~\cite{Schaetzle2001}. One can alternatively formulate compactness over the 
space $GMM(\chi_{k,0})$ of generalized minimizing movements, introduced by 
Ambrosio et al.\ \cite{Ambrosio2005}. The space $GMM(\chi_{k,0})$ is given by all limit points 
as $h\to 0$ of the minimizing movements scheme introduced in 
Subsection~\ref{subsec:keyIngredientsMinMov}. By Theorem~\ref{theo:mainResult}, every 
element of $GMM(\chi_{k,0})$ is a varifold solution of Mullins--Sekerka flow with 
initial value $\chi_{k,0}$. Though we do not prove this, a diagonalization argument shows 
that for a sequence of initial data as in Part~ii) of 
Lemma~\ref{lem:propertiesVarifoldSolutions}, $(\chi_k, \mu_k)$ belonging 
to $GMM(\chi_{k,0})$ are precompact and up to a subsequence converge 
to $(\chi,\mu)$ in $GMM(\chi_0)$. Note that this compactness result holds 
without the assumption of absolute continuity of the associated energies $E[(\mu_k)_t]$.

As indicated by the above remark, solutions arising from the minimizing movements 
scheme can satisfy additional properties. In the following proposition, we collect 
both structural and regularity properties for the time-evolving
varifold associated with a solution which is a limit point
of the minimizing movements scheme of Subsection~\ref{subsec:keyIngredientsMinMov}. 
%

\begin{proposition}[Further structure of the evolving varifold 
for limit points of minimizing movements approximation]
\label{prop:RoegerInterpretation} 
Let the assumptions and notation of Theorem~\ref{theo:mainResult}
be in place. Let $(\chi,\mu)$ be a varifold solution for Mullins--Sekerka 
flow~\emph{\eqref{eq:MullinsSekerka1}--\eqref{eq:MullinsSekerka6}} 
with initial data~$(\chi_0,\mu_0)$ in the sense of Definition~\ref{def:VarifoldSolution}.
Assume that $(\chi,\mu)$ is obtained as a limit point of the minimizing movements scheme
(cf.\ Subsection~\ref{subsec:keyIngredientsMinMov}) naturally
associated with Mullins--Sekerka flow~\emph{\eqref{eq:MullinsSekerka1}--\eqref{eq:MullinsSekerka6}}.

Then, the varifold~$\mu$ satisfies the following additional properties:
\begin{itemize}[leftmargin=0.7cm]
\item[i)] \emph{(Stronger compatibility conditions)} 
					Consider some $\eta \in C^\infty(\overline{\Omega};\Rd)$
					such that $n_{\partial\Omega}\cdot \eta = 0$ on~$\partial\Omega$, and consider 
					some $\xi \in C^\infty(\overline{\Omega};\Rd)$
					with $n_{\partial\Omega}\cdot\xi = \cos\alpha$ on~$\partial\Omega$. Then, it holds that
				  \begin{align}
					\label{eq:comp3}
					\int_{\overline{\Omega} {\times} \Sph} s \cdot \eta(x) \,d\mu_t^{\Omega}(x,s) 
					&= \int_{\Omega} c_0 \frac{\nabla\chi(\cdot,t)}{|\nabla\chi(\cdot,t)|} \cdot \eta(\cdot) 
					\,d|\nabla\chi(\cdot,t)|,
					\\ \label{eq:comp4}
					- \int_{\overline{\Omega} {\times} \Sph} s \cdot \xi(x) \,d\mu_t^{\Omega}(x,s) 
					&= - \int_{\Omega} c_0 \frac{\nabla\chi(\cdot,t)}{|\nabla\chi(\cdot,t)|} \cdot \xi(\cdot)
					\,d|\nabla\chi(\cdot,t)|,
					\\&~~~~ \nonumber
					+ |\mu_t^{\p\Omega}|_{\mathbb{S}^{d-1}}(\p\Omega) 
					- \int_{\p\Omega} (\cos\alpha)c_0\chi(\cdot,t) \,\dH
					\end{align}
					for almost every $t \in (0,T_{*})$.
\item[ii)] \emph{(Integrability of generalized mean curvature vector w.r.t.\ tangential 
					variations, cf.\ R\"{o}ger~\cite{Roeger2005} and Sch\"{a}tzle~\cite{Schaetzle2001})}
					For almost every $t \in (0,T_*)$, the generalized mean curvature~$H_\chi(\cdot,t)$ 
					from item~\textit{ii)} of Definition~\ref{def:VarifoldSolution}
					satisfies~\eqref{eq:repTangentialFirstVariation} not
					only for $B \in \mathcal{S}_{\chi(\cdot,t)}$ but also for all 
					$B \in C^1(\overline{\Omega};\Rd)$ with $B \cdot n_{\p\Omega}=0$ along~$\p\Omega$.
					Furthermore, for each $s \in [2,4]$ if $d = 3$ or else
					 $s \in [2,\infty)$, there exists $C=C(\Omega,d,s,c_0,m_0,\chi_0) > 0$ 
					(independent of~$t$) such that
					\begin{align}
					\label{eq:quantitativeIntCurvature}
					~~~~\bigg(\int_{\Omega} |H_\chi(\cdot,t)|^s 
					\,d|\nabla\chi(\cdot,t)| \bigg)^\frac{1}{s}
					\leq C\big(1 {+} \max\big\{1,\big|\partial 
					E [\mu_t]\big|_{\mathcal{V}_{\chi(\cdot,t)}}^d\big\}
					\big)^{1 {+} \frac{1}{s}}
					\end{align}
					for almost every $t \in (0,T_*)$.
\item[iii)] \emph{(Global first variation estimate on~$\overline{\Omega}$)}
					 For almost every $t \in (0,T_*)$, the oriented varifold $\mu_t$
					 is of bounded first variation on~$\overline{\Omega}$ such that
					 \begin{equation}
					 \label{eq:quantitativeGlobalFirstVarEstimate}
					 \begin{aligned}
					 &\sup\big\{|\delta \mu_t(B)| \colon 
					 B \in C^1(\overline{\Omega}),\, \|B\|_{L^\infty} \leq 1\big\}
					 \\&
					 \leq C\big(1 {+} \max\big\{1,\big|\partial 
					 E [\mu_t]\big|_{\mathcal{V}_{\chi(\cdot,t)}}^d\big\}\big)^{3/2}
					 \end{aligned}
					 \end{equation}
					 for some $C=C(\Omega,d,c_0,m_0,\chi_0) > 0$ independent of~$t$.
\end{itemize}
\end{proposition}

The proof of the last two items of the previous result 
is based upon the following two auxiliary results,
which we believe are worth mentioning on their own.

\begin{proposition}[First variation estimate up to the boundary for tangential variations]
\label{prop:firstVariationTangential}
Let $d \in \{2,3\}$, let $\Omega \subset \Rd$ be a bounded domain
with orientable $C^2$ boundary~$\p\Omega$, let $w \in H^1(\Omega)$,
let~$\chi\in BV(\Omega;\{0,1\})$,
and let $\mu = |\mu|_{\mathbb{S}^{d-1}} \otimes (\lambda_x)_{x\in\overline{\Omega}}
\in \mathrm{M}(\overline{\Omega}{\times}\mathbb{S}^{d-1})$
be an oriented varifold such that $c_0|\nabla\chi|\llcorner\Omega
\leq |\mu|_{\mathbb{S}^{d-1}}\llcorner\Omega$ in the sense of measures for some constant $c_0>0$.
Assume moreover that the Gibbs--Thomson law holds true in form of
\begin{align}
\label{eq:GibbsThomsonAux}
\int_{\overline{\Omega} {\times} \Sph} (\mathrm{Id} - s \otimes s) : \nabla B \,d\mu
= \int_{\Omega} \chi \nabla \cdot (w B) \,dx
\end{align}
for all tangential variations $B \in C^1(\overline{\Omega};\Rd)$, 
$(B \cdot n_{\partial\Omega})|_{\partial\Omega} \equiv 0$.

There exists $r = r(\p\Omega) \in (0,1)$ such that for all 
$x_0 \in \overline{\Omega}$ with $\dist(x_0,\p\Omega) < r$ and all 
exponents $s \in [2,4]$ if $d = 3$ or otherwise $s \in [2,\infty)$
there exists a constant $C = C(r,s,d) > 0$ such that
\begin{align}
\label{eq:firstVariationTangential}
\bigg(\int_{B_r(x_0) \cap \overline{\Omega}} |w|^s \,d|\mu|_{\mathbb{S}^{d-1}} \bigg)^\frac{1}{s}
\leq C\big(1 + |\mu|_{\mathbb{S}^{d-1}}(\overline{\Omega}) +
\|w\|_{H^{1}(\Omega)}^d\big)^{1 {+} \frac{1}{s}}.
\end{align}
In particular, the varifold~$\mu$ is of bounded variation with respect to tangential variations
(with generalized mean curvature vector~$H^\Omega$ trivially given by
$\rho^\Omega \smash{\frac{w}{c_0}} \smash{\frac{\nabla\chi}{|\nabla\chi|}}$ 
where $\rho^\Omega := \smash{\frac{c_0|\nabla\chi|\llcorner\Omega}
{|\mu|_{\mathbb{S}^{d-1}}\llcorner\Omega}}\in [0,1]$, cf.\ \eqref{eq:GibbsThomsonAux})
and the potential satisfies
\begin{align}
\label{eq:firstVariationInterior}
\bigg(\int_{\overline{\Omega}} |w|^s \,d|\mu|_{\mathbb{S}^{d-1}} \bigg)^\frac{1}{s}
\leq C\big(1 + |\mu|_{\mathbb{S}^{d-1}}(\overline{\Omega}) 
+ \max\big\{1,\|w\|_{H^{1}(\Omega)}^d\big\}\big)^{1 {+} \frac{1}{s}}.
\end{align}
\end{proposition}

By a recent work of De~Masi~\cite{DeMasi2021}, one may post-process
the previous result to the following statement.

\begin{corollary}[First variation estimate up to the boundary]
\label{cor:boundedFirstVariationGlobally}
In the setting of Proposition~\ref{prop:firstVariationTangential},
the varifold~$\mu$ is in fact of bounded variation on~$\overline{\Omega}$. 
More precisely, there exist $H^\Omega,\,H^{\partial\Omega}$ and $\sigma_\mu$
with the properties
\begin{align}
H^\Omega &= \rho^\Omega \frac{w}{c_0} \frac{\nabla\chi}{|\nabla\chi|}, \quad
\rho^\Omega := \frac{c_0|\nabla\chi|\llcorner\Omega}
{|\mu|_{\mathbb{S}^{d-1}}\llcorner\Omega} \in [0,1],
\\
H^{\partial\Omega} &\in L^\infty(\partial\Omega,d|\mu|_{\mathbb{S}^{d-1}}),
\quad H^{\partial\Omega}(x) \perp \mathrm{Tan}_x\partial\Omega \,
\text{ for } |\mu|_{\mathbb{S}^{d-1}}\llcorner\partial\Omega\text{-a.e.\ } x \in \overline{\Omega},
\\
\sigma_{\mu} &\in \mathrm{M}(\partial\Omega),
\end{align}
such that the first variation~$\delta\mu$ of~$\mu$ is represented by
\begin{align}
\label{eq:repFirstVariationGlobally}
\delta \mu(B) = - \int_{\overline{\Omega}} (H^\Omega {+} H^{\partial\Omega}) \cdot B \,d|\mu|_{\mathbb{S}^{d-1}}
+ \int_{\partial\Omega} B \cdot n_{\partial\Omega} \,d\sigma_{\mu}
\end{align}
for all $B \in C^1(\overline{\Omega};\Rd)$. Furthermore,
there exists $C = C(\Omega) > 0$ (depending only
on the second fundamental form of the domain boundary~$\partial\Omega$)
such that
\begin{align}
\label{eq:boundedFirstVariationGlobally}
\sup\big\{|\delta \mu(B)| \colon B \in C^1(\overline{\Omega}),\,\|B\|_{L^\infty} \leq 1 \big\}
&\leq C\big(|\mu|_{\mathbb{S}^{d-1}}(\overline{\Omega}) 
{+} \|w\|_{L^1(\overline{\Omega},d|\mu|_{\mathbb{S}^{d-1}})}\big),
\\
\label{eq:estimateBoundaryCurvature}
\|H^{\partial\Omega}\|_{L^\infty(\partial\Omega,d|\mu|_{\mathbb{S}^{d-1}})} &\leq C,
\\
\label{eq:estimateBoundaryMeasure}
\sigma_{\mu}(\partial\Omega) &\leq C|\mu|_{\mathbb{S}^{d-1}}(\overline{\Omega}) 
+ \|w\|_{L^1(\overline{\Omega},d|\mu|_{\mathbb{S}^{d-1}})}.
\end{align}
\end{corollary}

\subsection{A closer look at the functional framework}\label{subsec:functionSpaces}
In this subsection, we characterize the difference between the 
velocity spaces $\mathcal{V}_\chi$ and $\mathcal{T}_{\chi}$, defined
in~\eqref{def:surfaceVelocities} and~\eqref{def:tangentVelocities} 
respectively, by expressing the quotient space $\mathcal{T}_{\chi} / \mathcal{V}_\chi$ 
in terms of a distributional trace space and quasi-everywhere trace space 
(see (\ref{eqn:velocityTraceIsomorphism})). As an application, this result will 
show that if $|\nabla \chi|$ is given by the surface measure of a Lipschitz graph, 
then the quotient space collapses to a point and $\mathcal{V}_{\chi} = \mathcal{T}_\chi.$

Both spaces $\mathcal{V}_\chi$ and $\mathcal{T}_{\chi}$ are spaces of velocities, and associated with 
these will be spaces of potentials where one expects to find the chemical 
potential. For this, we recall that the inverse of the 
weak Neumann Laplacian $\Delta_N^{-1} \colon H^{-1}_{(0)} \to H^1_{(0)}$
is defined by $u_F := \Delta_N^{-1}(F)$, $F \in H^{-1}_{(0)}$, where~$u_F \in H^1_{(0)}$
is the unique weak solution of the Neumann problem
\begin{equation}\label{eqn:lapDef}
\begin{aligned}
\Delta u_F &= F \quad \text{in } \Omega, \\
(n_{\p\Omega} \cdot \nabla) u_F &= 0 \quad \text{on } \partial \Omega.
\end{aligned} 
\end{equation}
Recall also that $\Delta_N^{-1}\colon H^{-1}_{(0)} \to H^1_{(0)}$ 
defines an isometric isomorphism (with respect to the Hilbert
space structures on $H^1_{(0)}$ and $H^{-1}_{(0)}$ defined in
Subsection~\ref{subsec:weakSol}), and since $\Delta_N$ is nothing
else but the Riesz isomorphism for the Hilbert space~$H^1_{(0)}$, the relation  
\begin{equation}\label{eqn:invLapRelation}
(u_F , v)_{H^1_{(0)}} = \langle F, v\rangle_{H^{-1}_{(0)},H^1_{(0)}} 
\end{equation}
holds for all $v\in H^1_{(0)}.$ We then introduce a space of potentials
associated with~$\mathcal{V}_\chi$ given by 
\begin{equation}\label{def:GA}
\mathcal{G}_{\chi} := \Delta_N^{-1} (\mathcal{V}_{\chi}) \subset H^1_{(0)}.
\end{equation}
Likewise we can introduce the space of potentials associated to 
the ``maximal tangent space''~$\mathcal{T}_\chi$ given by 
\begin{equation}\label{def:HA} 
\mathcal{H}_\chi : = \Delta_N^{-1} (\mathcal{T}_{\chi}) \subset H^1_{(0)}.
\end{equation}

To understand the relation between the spaces $\mathcal{V}_\chi$ and $\mathcal{T}_\chi$, we will develop annihilator relations for 
$\mathcal{G}_\chi$ and $\mathcal{H}_\chi$ in $H^1_{(0)}$. Throughout the remainder of this subsection, we identify $H^1_{(0)}$ with $H^1(\Omega)/\mathbb{R}$, the Sobolev space quotiented by constants, which allows us to consider any $v\in H^1(\Omega)$ as an element of $H^1_{(0)}$.

By~\cite[Corollary 9.1.7]{adamsHedberg} of Adams and Hedberg, 
$$\mathcal{T}_{\chi} = \{F \in H^{-1}_{(0)} \colon 
\langle F, v\rangle_{H^{-1}_{(0)},H^1_{(0)}} = 0 \text{ for all }
v\in C^1_c(\Omega \setminus \supp|\nabla\chi|) \}.$$ 
Using~\eqref{eqn:invLapRelation} and~\eqref{def:HA}, this implies that the space
\begin{equation}\label{eqn:annihilator0}
H^1_{0,\supp|\nabla\chi|} : = 
\overline{\{v\in C^1_c(\Omega \setminus \supp|\nabla\chi|)\}}^{H^{1}_{(0)}},
\end{equation}
satisfies the annihilator relation with $(\cdot,\cdot)_{H^1_{(0)}}$ 
\begin{equation}\label{eqn:annihilator1}
H^{1}_{0,\supp|\nabla\chi|} = \mathcal{H}_\chi^{\perp}.
\end{equation}
Further by \cite[Corollary 9.1.4]{adamsHedberg}, for the trace 
operator~${\rm Tr}_{\supp|\nabla\chi|}$ defined ${\rm Cap}_2$ quasi-everywhere 
for~$H^1(\Omega)$-functions, (\ref{eqn:annihilator1}) is the same as 
\begin{equation}\label{eqn:HAannihilator}
\ker\left( {\rm Tr}_{\supp|\nabla\chi|}\right) = \mathcal{H}_\chi^{\perp},
\end{equation}
where $u\in H^1_{(0)}\cap \ker\left( {\rm Tr}_{\supp|\nabla\chi|}\right) $ if and only if $ {\rm Tr}_{\supp|\nabla\chi|} u \equiv c$ for some $c\in \mathbb{R}$.

Similarly, one may use the definition (\ref{eq:actionVel}) and the relation~\eqref{eqn:invLapRelation} to show that
\begin{equation}\label{eqn:GAannihilator}
\mathcal{G}_\chi^{\perp} = \left\{u \in H^1_{(0)}\colon 
\int_\Omega \chi\nabla u \cdot B \, dx  = -\int_\Omega \chi u \nabla \cdot B \, dx
\text{ for all } B \in \mathcal{S}_\chi\right\}.
\end{equation}
For $B\in C^1(\overline\Omega;\Rd)$ such that $\int_\Omega \chi \nabla \cdot B \, dx \neq 0,$ 
one can consider fixed $\xi\in C^1(\overline \Omega;\Rd)$ with $\xi \cdot n_{\partial \Omega} = 0$ 
on $\partial \Omega$ such that $\int_\Omega \chi \nabla \cdot \xi \, dx \neq 0$ and use the 
corrected function $\tilde B  : = B - 
\frac{\int_\Omega \chi \nabla \cdot B \, dx}{\int_\Omega \chi \nabla \cdot \xi \, dx} \xi$ 
in (\ref{eqn:GAannihilator}) to see that the above relation is equivalent to 
\begin{align}
\label{eqn:GAannihilator2}
\mathcal{G}_\chi^{\perp} = \bigg\{u \in H^1_{(0)}& \colon 
 \int_\Omega \chi\nabla (u+c) \cdot B \, dx  = -\int_\Omega \chi (u+c) \nabla \cdot B \, dx 
\\ \nonumber
& \text{ for some } c\in \mathbb{R} \text{ and all } B \in C^1(\bar \Omega;\Rd) \text{ with }B \cdot n_{\partial \Omega} = 0 \text{ on }\partial \Omega\bigg\}.
\end{align}
Thus, $\mathcal{G}_{\chi}^\perp$ is the space functions 
in~$H^1_{(0)}$ which have vanishing trace on~$\supp|\nabla\chi|$ 
in a distributional sense.

We now show that $\mathcal{G}_{\chi} \subset \mathcal{H}_\chi$, which is equivalent to $\mathcal{V}_{\chi} \subset \mathcal{T}_{\chi}$. First
note that~\eqref{eqn:annihilator1} implies
\begin{equation}\label{eqn:HchiHarmonicRelation}
\mathcal{H}_\chi  = \big\{ u\in H^1_{(0)} \colon 
\Delta u = 0 \text{ in } \Omega \setminus \supp|\nabla\chi| \big\}.
\end{equation}
As a technical tool, we remark that for fixed $v\in  C^1_c(\Omega \setminus \supp|\nabla\chi|),$ up to a representative, $v \chi \in C^1_c(\Omega).$ To see this, let $\chi = \chi_A$ for $A\subset \Omega$ and note for any $x\in {\rm supp}\, v$ we can find $r>0$ such that $|B(x,r)\cap A| = |B(x,r)|$ or $|B(x,r)\cap (\Omega \setminus A)| = |B(x,r)|$. We construct a finite cover of ${\rm supp}\, v$ given by $\mathcal{C}: = \cup_{i}B(x_i,r_i)$, and define the set $$\mathcal{C}' := \bigcup_{x_i:|B(x_i,r_i)\cap A| = |B(x_i,r_i)|} B(x_i,r_i).$$
We have that 
$$v \chi  = \begin{cases} v & \text{ in }\mathcal{C}',\\
0 & \text{ otherwise},
 \end{cases}$$
 which is smooth as the balls used in $\mathcal{C}'$ are disjoint from those balls such that $|B(x,r)\cap (\Omega \setminus A)| = |B(x,r)|$, completing the claim. 
Now, for $u \in \mathcal{G}_{\chi}$ given by $u = u_{B\cdot \nabla\chi}$ with $B \in \mathcal{S}_{\chi}$, 
by~\eqref{eqn:invLapRelation} we compute
\begin{equation*}
(u, v)_{H^1_{(0)}} = -\int_{\Omega} \nabla \cdot (vB)\chi \, dx 
=- \int_{\Omega} \nabla \cdot ((v\chi)B) \, dx = 0.
\end{equation*}
It follows that $\mathcal{G}_{\chi} \subset \mathcal{H}_{\chi}$ by (\ref{eqn:annihilator0}) and (\ref{eqn:annihilator1}).

Using the quotient space isomorphism $Y/X \simeq X^{\perp}$ for a closed subspace~$X$ 
of $Y$~\cite[Theorem III.10.2]{conwayFunctional} and the subset relation 
$\mathcal{G}_{\chi} \subset \mathcal{H}_{\chi}$, we have 
$\mathcal{H}_\chi \big/ \mathcal{G}_{\chi} \simeq 
\mathcal{G}_\chi^\perp \big/ \mathcal{H}_{\chi}^\perp$. Consequently, 
unifying the results of this subsection, the following characterization 
of the difference between the velocity spaces follows:
\begin{align}
\label{eqn:velocityTraceIsomorphism}
& \mathcal{T}_{\chi}\big/ \mathcal{V}_\chi 
\simeq \mathcal{H}_\chi \big/ \mathcal{G}_{\chi} \simeq \\ \nonumber
&  \left\{u \in H^1_{(0)}\colon 
\int_\Omega \chi\nabla u \cdot B \, dx  = -\int_\Omega \chi u \nabla \cdot B \, dx 
\text{ for all } B \in \mathcal{S}_\chi\right\}\big/\ker {\rm Tr}_{\supp|\nabla\chi|}.
\end{align}
In summary, the gap in the velocity spaces $\mathcal{V}_{\chi}$ 
and $\mathcal{T}_{\chi}$ is exclusively due to a loss in regularity 
of the interface and amounts to the gap between having the trace in a 
distributional sense (see (\ref{eqn:GAannihilator2})) versus a quasi-everywhere sense.

\subsection{On the relation to Le's functional framework}\label{subsec:relationToLe}
We now have sufficient machinery to discuss our solution concept in relation
to the framework developed by Le in \cite{Le2008}. Within Le's work, the 
critical dissipation inequality for $\Gamma_t := \supp|\nabla\chi(\cdot,t)|$, a $C^3$ space-time interface, to be a solution of the Mullins--Sekerka flow
is given by
\begin{equation*}
E[\chi(\cdot,T)] + \int_{0}^T \frac{1}{2}\| \partial_t \chi \|_{H^{-1/2}(\Gamma_t)}^2 
+ \frac{1}{2}\| H_{\Gamma_t}\|_{H^{1/2}(\Gamma_t)}^2 \, dt \leq E[\chi_0],
\end{equation*}
where $\|H_{\Gamma}\|_{H^{1/2}(\Gamma)} = \|\nabla \tilde f\|_{L^2(\Omega)}$ 
for $\tilde f$ satisfying (\ref{def:LeDirichletProb}) with $f = H_{\Gamma}$ (the curvature)  
and $H^{-1/2}(\Gamma)$ again defined by duality and normed by means of the Riesz representation 
theorem (see also Lemma~2.1 of Le~\cite{Le2008}), 
\begin{equation}\label{eqn:LeH-1/2}
H^{-1/2}(\Gamma) \simeq \Delta_N( \{u \in H^1_{(0)}\colon \Delta u = 0 
\text{ in } \Omega \setminus \Gamma \}).
\end{equation}
As this is simply the image under the weak Neumann Laplacian of functions~$u$ 
associated with the problem~\eqref{def:LeDirichletProb}, we can rewrite this as 
\begin{equation}\label{eqn:LeH-1/22}
H^{-1/2}(\Gamma) = \Delta_N(H^{1/2}(\Gamma)).
\end{equation}

Considering our solution concept now, let $(\chi,\mu)$ be a solution in the sense of 
Definition \ref{def:VarifoldSolution} such that~$\Gamma_t : = {\rm supp}|\nabla \chi(\cdot, t)|$ 
is a Lipschitz surface for a.e.\ $t$. By (\ref{eqn:HchiHarmonicRelation}) 
and (\ref{eqn:LeH-1/2}),  $\mathcal{T}_{\chi(\cdot, t)} = H^{-1/2}(\Gamma_t)$. 
Then as the classical trace space is well-defined, the 
isomorphism (\ref{eqn:velocityTraceIsomorphism}) collapses to the identity showing that 
\begin{equation}\label{eqn:structRelationTH}
 \mathcal{T}_{\chi(\cdot, t)} = \Delta_N (\mathcal{G}_{\chi(\cdot, t)}),
\end{equation} verifying the analogue of (\ref{eqn:LeH-1/22}) and implying that $\mathcal{G}_{\chi(\cdot, t)} = H^{1/2}(\Gamma_t).$ Further, this discussion and
(\ref{eqn:curvature=trace}) (letting $c_0 = 1$ for convenience) show that
\begin{align*}
\| \partial_t \chi \|_{\mathcal{T}_{\chi(\cdot,t)}} 
=  \| \partial_t \chi \|_{H^{-1/2}(\Gamma_t)}
\quad \text{ and } \quad\| w(\cdot,t)\|_{\mathcal{G}_{\chi(\cdot,t)}} 
= \| H_{\Gamma_t}\|_{H^{1/2}(\Gamma_t)}.
\end{align*}
Looking to (\ref{eq:DeGiorgiInequality}) and (\ref{eq:normEquivalenceGibbsThomsonPotential}), we see that our solution concept naturally subsumes Le's, preserves structural relations on the function spaces, and works without any 
regularity assumptions placed on $\Gamma.$ 

Though beyond the scope of our paper, a natural question following from 
the discussion of this subsection and the prior is when does the relation (\ref{eqn:structRelationTH}) or the inclusion $\partial_t \chi 
\in \mathcal{V}_{\chi}\subset \mathcal{T}_\chi$ hold. By (\ref{eqn:velocityTraceIsomorphism}), both will
follow if $zero$ distributional trace is equivalent to having $zero$ trace in 
the quasi-everywhere sense. Looking towards results on traces~(see, e.g., \cite{chenLiTorres}, \cite{meyers1978}, and \cite{meyersZiemer1977}), characterization of this condition will be a nontrivial 
result, and applying similar ideas to the Mullins--Sekerka flow may require a fine 
characterization of the singular set from Allard's regularity theory~\cite{Allard1972}.

\subsection{Motivation from the viewpoint of weak-strong uniqueness}
\label{subsec:PDEperspective}
Another major motivation for our weak solution concept,
especially for the inclusion of a sharp energy dissipation principle, is drawn from the 
recent progress on uniqueness properties of weak solutions for various
curvature driven interface evolution problems. More precisely,
it was established that for incompressible Navier--Stokes two-phase flow with 
surface tension~\cite{Fischer2020c} and for multiphase mean curvature flow~\cite{Fischer2020a}
(cf.\ also~\cite{Hensel2021l} or~\cite{Hensel2021}), weak solutions with 
sharp energy dissipation rate are unique within a class of sufficiently regular 
strong solutions (as long as the latter exist, i.e., until they undergo a topology change). 
Such weak-strong uniqueness principles are optimal in the sense that weak
solutions in geometric evolution may in general be non-unique after the first topology change.
Extensions to constant contact angle problems as considered in the present work
are possible as well, see~\cite{Hensel2021m} for Navier--Stokes two-phase flow
with surface tension or~\cite{Hensel2021d} for mean curvature flow.

The weak-strong uniqueness results of the previous works rely on a Gronwall stability estimate
for a novel notion of distance measure between a weak and a sufficiently regular strong
solution. The main point is that this distance measure is in particular able 
to penalize the difference in the location of the two associated interfaces in 
a sufficiently strong sense. 
Let us briefly outline how to construct such a distance measure
in the context of the present work (i.e., interface evolution
in a bounded container with constant contact angle~\eqref{eq:MullinsSekerka6}).
To this end, it is convenient to assume next to~\eqref{eq:comp1} and~\eqref{eq:comp2} 
the two additional compatibility conditions~\eqref{eq:comp3} and~\eqref{eq:comp4}
from Proposition~\ref{prop:RoegerInterpretation}. Under these additional assumptions, 
we claim that the following functional represents a natural candidate for the desired 
error functional:
\begin{align*}
E_{\mathrm{rel}}[\chi,\mu|\mathscr{A}](t) 
&:= |\mu_t|_{\mathbb{S}^{d-1}}(\overline{\Omega}) 
- \int_{\p^*A(t) \cap \Omega} c_0 \frac{\nabla\chi(\cdot,t)}{|\nabla\chi(\cdot,t)|} 
\cdot \xi(\cdot,t) \,d|\nabla\chi(\cdot,t)|
\\&~~~~
- \int_{\p\Omega} (\cos\alpha)c_0\chi(\cdot,t) \,\dH,
\end{align*} 
where $\xi(\cdot,t)\colon\overline{\Omega}\to\{|x| {\leq} 1\}$ denotes a suitable extension of the
unit normal vector field~$n_{\partial\mathscr{A}(t)}$ of~$\p\mathscr{A}(t) \cap \Omega$. Due 
to the compatibility conditions~\eqref{eq:comp1} and~\eqref{eq:comp2} 
as well as the length constraint~$|\xi| \leq 1$,
it is immediate that~$E_{\mathrm{rel}} \geq 0$. The natural
boundary condition for~$\xi(\cdot,t)$ turns out to be 
$(\xi(\cdot,t) \cdot n_{\partial\Omega})|_{\partial\Omega} \equiv \cos\alpha$.
Indeed, this shows by means of an integration by parts that
\begin{align*}
E_{\mathrm{rel}}[\chi,\mu|\mathscr{A}](t)
= |\mu_t|_{\mathbb{S}^{d-1}}(\overline{\Omega}) 
+ \int_{\Omega} c_0\chi(\cdot,t) \nabla \cdot \xi \,dx.
\end{align*}
The merit of the previous representation of~$E_{\mathrm{rel}}$ is that
it allows one to compute the time evolution of~$E_{\mathrm{rel}}$ relying
in a first step only on the De~Giorgi inequality~\eqref{eq:DeGiorgiInequality} 
and using $\nabla \cdot \xi$ as a test function in
the evolution equation~\eqref{eq:evolEquationPhase}. 
Furthermore, the compatibility condition~\eqref{eq:comp4} yields that
\begin{align*}
\int_{\overline{\Omega} {\times} \Sph} \frac{1}{2} |s - \xi|^2 \,d\mu_t^\Omega
\leq \int_{\overline{\Omega} {\times} \Sph} 1  - s \cdot \xi \,d\mu_t^\Omega 
= E_{\mathrm{rel}}[\chi,\mu|\mathscr{A}](t),
\end{align*}
which in turn implies a tilt-excess type control provided by~$E_{\mathrm{rel}}$
at the level of the varifold interface. Further coercivity properties
may be derived based on the compatibility conditions~\eqref{eq:comp1} and~\eqref{eq:comp2}
in form of the associated Radon--Nikod\`ym derivatives
$\rho^{\Omega}_t := \frac{c_0|\nabla\chi(\cdot,t)|\llcorner\Omega}
{|\mu_t^{\Omega}|_{\mathbb{S}^{d-1}}\llcorner\Omega} \in [0,1]$
and $\rho^{\p\Omega}_t := \frac{(\cos\alpha)c_0\chi(\cdot,t)\mathcal{H}^{d-1}
\llcorner\p\Omega}{|\mu_t^{\p\Omega}|_{\mathbb{S}^{d-1}}
+ |\mu_t^{\Omega}|_{\mathbb{S}^{d-1}} \llcorner \p\Omega} \in [0,1]$, respectively.
More precisely, one obtains the representation
\begin{align*}
&E_{\mathrm{rel}}[\chi,\mu|\mathscr{A}](t)
\\&
= \int_{\Omega} 1 - \rho_t^{\Omega} \,d|\mu_t^{\Omega}|_{\mathbb{S}^{d-1}}
+ \int_{\p\Omega} 1 - \rho_t^{\p\Omega} 
\,d\big(|\mu_t^{\Omega}|_{\mathbb{S}^{d-1}} {+} |\mu_t^{\p\Omega}|_{\mathbb{S}^{d-1}}\big)
\\&~~~
+ \int_{\Omega} c_0\Big(1 - \frac{\nabla\chi(\cdot,t)}{|\nabla\chi(\cdot,t)|} \cdot \xi(\cdot,t)\Big) 
\,d|\nabla\chi(\cdot,t)|.
\end{align*}
The last of these right hand side terms ensures tilt-excess type control
at the level of the $BV$~interface
\begin{align*}
c_0 \int_{\Omega} \frac{1}{2} 
\Big|\frac{\nabla\chi(\cdot,t)}{|\nabla\chi(\cdot,t)|} - \xi(\cdot,t) \Big|^2 
\,d|\nabla\chi(\cdot,t)|
\leq E_{\mathrm{rel}}[\chi,\mu|\mathscr{A}](t).
\end{align*}
The other three simply penalize the well-known mass defects 
(i.e., mass moving out from the bulk to the domain boundary, or
the creation of hidden boundaries within the bulk) originating from the 
lack of continuity of the perimeter functional under weak-$*$ convergence in~$BV$.

In summary, the requirements of Definition~\ref{def:VarifoldSolution}
(together with the two additional mild compatibility conditions~\eqref{eq:comp3} and~\eqref{eq:comp4})
allow one to define a functional which on one side penalizes, in various ways, the 
``interface error'' between a varifold and a classical solution, and which 
on the other side has a structure supporting at least in principle the idea
of proving a Gronwall-type stability estimate for it. One therefore may hope that
varifold solutions for Mullins--Sekerka flow in the sense of Definition~\ref{def:VarifoldSolution}
satisfy a weak-strong uniqueness principle together with a weak-strong
stability estimate based on the above error functional. In the  
simplest setting of $\alpha = \smash{\frac{\pi}{2}}$, a $BV$~solution~$\chi$, 
and assuming no boundary contact for the interface of the classical solution~$\mathscr{A}$, 
this is at the time of this writing work in progress~\cite{Fischer2022}.

For the present contribution, however, we content ourselves with the above existence result
(i.e., Theorem~\ref{theo:mainResult}) for varifold solutions to Mullins--Sekerka flow in the sense of 
Definition~\ref{def:VarifoldSolution} together with establishing further properties of these.

\section{Existence of varifold solutions to Mullins--Sekerka flow}

\subsection{Key players in minimizing movements}
\label{subsec:keyIngredientsMinMov}
To construct weak solutions for the Mullins--Sekerka 
flow~\eqref{eq:MullinsSekerka1}--\eqref{eq:MullinsSekerka6}
in the precise sense of Definition~\ref{def:VarifoldSolution},
it comes at no surprise that we will employ the gradient flow perspective in the form of a
minimizing movements scheme, which we pass to the limit. 
Given an initial condition $\chi_0 \in \mathcal{M}_{m_0}$ (see (\ref{def:manifold})), a fixed time step size $h \in (0,1)$, and $E$ as in (\ref{def:energy}), we let $\chi^h_0:=\chi_0$ 
and choose inductively for each $n\in\mathbb{N}$
\begin{align}
\label{eq:minMovStep}
\chi^h_n \in \argmin_{\widetilde \chi \in \mathcal{M}_{m_0}}
\bigg\{E[\widetilde \chi\,] + \frac{1}{2h}
\big\|\chi^h_{n-1} - \widetilde \chi\,\big\|^2_{H^{-1}_{(0)}}\bigg\},
\end{align}
an approximation via the backward-Euler scheme.
Note that this minimization problem is indeed solvable by the
direct method in the calculus of variations; see, for instance,
the result of Modica~\cite[Proposition 1.2]{Modica1987}
for the lower-semicontinuity of the capillary energy.

By a telescoping argument, it is immediate 
that the associated piecewise constant interpolation
\begin{align}
\label{eq:minMovPiecewiseInterpol}
\chi^h(t) := \chi^h_{n-1} \quad \text{for all } t \in [(n{-}1)h,nh),\,n\in\mathbb{N},
\end{align}
satisfies the energy dissipation estimate
\begin{align}
\label{eq:minMovSuboptimalEnergyDissip}
E[\chi^h(T)] + \int_{0}^{T} \frac{1}{2h^2} 
\big\|\chi^h(t{+}h) - \chi^h(t)\big\|^2_{H^{-1}_{(0)}} \,dt
\leq E[\chi_0] \quad \text{for all } T \in \mathbb{N}h.
\end{align}

Although the previous inequality is already enough for usual compactness arguments,
it is obviously not sufficient, however, to establish the expected sharp energy dissipation
inequality (cf.\ \eqref{eq:energyIdentity}) in the limit as $h \to 0$. 
It goes back to ideas of De~Giorgi how to capture the remaining
half of the dissipation energy at the level of the minimizing movements scheme,
versus, for example, recovering the dissipation from the regularity of a solution to the limit equation. The key
ingredient for this is a finer interpolation than the piecewise constant one,
which in the literature usually goes under the name of De~Giorgi (or variational) interpolation
and is defined as follows:
\begin{align}
\nonumber
&\bar{\chi}^h((n{-}1)h) := {\chi}^h((n{-}1)h) = \chi^h_{n-1}, \quad n \in \mathbb{N},
\\ \label{eq:minMovVariationalInterpol}
&\bar{\chi}^h(t) \in \argmin_{\widetilde \chi \in \mathcal{M}_{m_0}}
\bigg\{E[\widetilde \chi\,] {+} \frac{1}{2(t{-}(n{-}1)h)}
\big\|\bar{\chi}^h((n{-}1)h) - \widetilde \chi\,\big\|^2_{H^{-1}_{(0)}}\bigg\},
\, t \in ((n{-}1)h,nh).
\end{align}
The merit of this second interpolation consists of the following
improved (and now sharp) energy dissipation inequality 
\begin{align}
\label{eq:minMovSharpEnergyDissip}
E[\chi^h(T)] + \int_{0}^{T} \frac{1}{2h^2} \big\|\chi^h(t{+}h) - \chi^h(t)\big\|^2_{H^{-1}_{(0)}} \,dt
+ \int_{0}^{T} \frac{1}{2} \big|\p E[\bar{\chi}^h(t)]\big|_{\mathrm{d}}^2 \,dt \leq E[\chi_0],
\end{align}
with $T \in \mathbb{N}h$~\cite{Ambrosio2005}. 
The quantity~$|\p E[\chi]|_{\mathrm{d}}$ is usually referred to as the metric slope
of the energy~$E$ at a given point~$\chi \in \mathcal{M}_{m_0}$, and
in our context may more precisely be defined by
\begin{align}
\label{eq:minMovMetricSlope}
\big|\p E[\chi]\big|_{\mathrm{d}} := 
\limsup_{\widetilde \chi \in \mathcal{M}_{m_0}\colon \|\chi - \widetilde \chi\|_{H^{-1}_{(0)}} \to 0} 
\frac{\big(E[\chi] - E[\widetilde \chi]\big)_+}{\big\|\chi - \widetilde \chi\big\|_{H^{-1}_{(0)}}}.
\end{align}
We remind the reader that~\eqref{eq:minMovSharpEnergyDissip} is a general
result for abstract minimizing movement schemes requiring only to work on a metric space.
However, as it turns out, we will be able to preserve a formal manifold structure even in the limit.
This in turn is precisely the reason why the ``De~Giorgi metric slope'' appearing in 
our energy dissipation inequality~\eqref{eq:DeGiorgiInequality} is computed
only in terms of inner variations, see~\eqref{def:MSmetricSlope2}.

With these main ingredients and properties of the 
minimizing movements scheme in place, our main task now consists of passing to the limit $h \to 0$
and identifying the resulting (subsequential but unconditional) limit object 
as a varifold solution to Mullins--Sekerka flow~\eqref{eq:MullinsSekerka1}--\eqref{eq:MullinsSekerka6}
in our sense. Furthermore, to obtain a $BV$~solution, we will \textit{additionally assume},
following the tradition of Luckhaus and Sturzenhecker~\cite{LucStu}, that 
for a subsequential limit point~$\chi$ obtained from~\eqref{eqn:chiStrongConvergence}
below, it holds that
\begin{align}
\label{eq:energyConvergenveAssumpMinMov}
\int_0^{T_*} E[\chi^h(t)] \,dt \to \int_0^{T_*} E[\chi(t)] \,dt.
\end{align}

\subsection{Three technical auxiliary results}
For the Mullins--Sekerka equation, a mass preserving flow, it will be helpful 
to construct ``smooth" mass-preserving flows corresponding to infinitesimally 
mass-preserving velocities, i.e., velocities in the test function
class~$\mathcal{S}_\chi$ (see~(\ref{def:variationVelocities})).
Using these flows as competitors in~(\ref{eq:minMovVariationalInterpol}) 
and considering the associated Euler--Lagrange equation, it becomes apparent 
that an approximate Gibbs--Thomson relation holds for infinitesimally mass-preserving 
velocities. To extend this relation to arbitrary variations (tangential at the boundary) 
we must control the Lagrange multiplier arising from the mass constraint. Though the first 
lemma and the essence of the subsequent lemma is contained in the work of Abels and 
R\"oger~\cite{AbelsRoeger} or Chen~\cite{Chen1992}, we include the proofs for both completeness 
and to show that the result is unperturbed if the energy exists at the varifold level.

\begin{lemma}
\label{lem:existenceDiffeos}
Let $\chi \in \mathcal{M}_0$ and $B \in \mathcal{S}_\chi$. Then there exists $\eta > 0$ and a 
family of $C^1$~diffeomorphisms $\Psi_s\colon \overline{\Omega} \to \overline{\Omega}$ 
depending differentiably on $s\in (-\eta,\eta)$ such that $\Psi_0(x) = x$, 
$\partial_s \Psi_s(x) |_{s=0} = B(x),$ and 
\begin{align*}
\int_\Omega \chi\circ \Psi_s^{-1} = m_0 \quad \text{ for all } s \in (-\eta,\eta).
\end{align*}
\end{lemma}
\begin{proof}
Fix $\xi \in C^{\infty}(\overline \Omega)$ such 
that $(\xi \cdot n_{\partial \Omega})|_{\partial\Omega} \equiv 0$ 
and $\int_\Omega  \chi \nabla\cdot\xi \, dx \neq 0.$ Naturally associated to $B$ and 
$\xi$ are flow-maps $\beta_s$ and $\gamma_r$ solving $\partial_s \beta_s(x) = B(\beta_s(x))$ 
and $\partial_r \gamma_r(x) = \xi(\gamma_r(x))$, each with initial condition given by the identity 
map, i.e., $\beta_0(x) = x$. Define the function $f$, which is locally differentiable 
near the origin, by
\begin{align*}
f(s,r) := \int_\Omega \chi \circ(\beta_s \circ \gamma_r)^{-1} - m_0  = 
\int_\Omega \chi \left({\rm det} (\nabla (\beta_s \circ \gamma_r) )-1 \right) dx .
\end{align*}
As $f(0,0) = 0$ and $\partial_r f(0,0) = \int_\Omega  \chi \nabla \cdot \xi \, dx\neq 0$ 
by assumption, we may apply the implicit function theorem to find a differentiable function 
$r = r(s)$ with $r(0) = 0$ such that $f(s,r(s)) = 0$ for $s$ near $0.$
We can further compute that (see (\ref{eqn:detTaylor}))
\begin{align*}
\partial_s f(0,0) = \int_\Omega \chi \big(\nabla \cdot B
+ r'(0) \nabla \cdot \xi \big) \, dx.
\end{align*}
Rearranging, we find 
\begin{align*}
r'(0) = - \frac{\int_\Omega \chi \nabla \cdot B\, dx}
{\int_\Omega \chi \nabla\cdot\xi \, dx} = 0,
\end{align*} 
and thus the flow given by $\beta_s \circ \gamma_{r(s)}$ 
satisfies $\partial_s (\beta_s \circ \gamma_{r(s)})|_{s=0} = B + r'(0)\xi = B$, thereby 
providing the desired family of diffeomorphisms.
\end{proof}

\begin{lemma}
\label{lem:AbelsRoeger}
Let $\chi \in \mathcal{M}_0$, $w\in H^1_{(0)}$, and $\mu \in 
\mathrm{M}(\overline{\Omega}{\times}\mathbb{S}^{d-1})$ be an oriented varifold such that
\begin{equation}\label{eqn:varifoldVarSoboHyp}
\delta \mu (B) = \int_\Omega \chi  \nabla \cdot (wB)\, dx \quad \text{ for all }B \in \mathcal{S}_{\chi}.
\end{equation} 
Then there is $\lambda \in \mathbb{R}$ such that
\begin{equation*}
\delta \mu (B) = \int_\Omega \chi \nabla \cdot \big((w {+} \lambda)B\big) \, dx 
\quad \text{ for all }B \in C^1(\overline \Omega;\Rd) 
\text{ with } (B \cdot n_{\partial \Omega})|_{\partial\Omega} \equiv 0
\end{equation*} 
and there exists $C = C(\Omega,d,c_0,m_0)$
\begin{equation}
\label{eq:boundH1NormPotential}
\|w {+} \lambda\|_{H^1(\Omega)} \leq C(1+|\nabla \chi|(\Omega)) 
\left( |\mu|(\overline \Omega) + \|\nabla w\|_{L^2(\Omega)}\right).
\end{equation}
\end{lemma}
\begin{proof}
For $\xi \in C^{\infty}(\overline \Omega)$ such that 
$(\xi \cdot n_{\partial \Omega})|_{\partial\Omega} \equiv 0$ 
and $\int_\Omega  \chi \nabla\cdot\xi \, dx \neq 0,$ we have that 
$\tilde B : = B - \frac{\int_\Omega \chi \nabla \cdot B\, dx}
{\int_\Omega \chi \nabla \cdot \xi \, dx} \xi$ belongs to $\mathcal{S}_\chi$, 
and plugging this into (\ref{eqn:varifoldVarSoboHyp}) and rearranging, one finds
\begin{equation*}
\delta \mu (B) - \int_\Omega \chi \nabla \cdot (wB) \, dx 
= \lambda \int_{\Omega} \chi \nabla \cdot B \, dx,
\end{equation*}
where 
\begin{equation}\label{eqn:lambdaDef}
\lambda  = \frac{\delta \mu (\xi ) - \int_\Omega \chi \nabla \cdot (w\xi) \, dx}
{\int_\Omega \chi \nabla \cdot \xi \, dx}.
\end{equation}
To conclude the lemma, it suffices to make a careful selection of $\xi$ such that
\begin{equation}\label{eqn:lambdaDes}
|\lambda| \leq C(1+|\nabla \chi|(\Omega)) 
\left( |\mu|(\overline \Omega) + \|\nabla w\|_{L^2(\Omega)}\right).
\end{equation}
Let $\rho_\epsilon$ be a standard mollifier for $\epsilon>0$, and let 
$\chi_\epsilon : =\chi *\rho_\epsilon$ with $m_\epsilon : = \dashint_\Omega \chi_\epsilon \, dx.$ 
We solve the Poisson problem 
\begin{equation*}
\left\{\begin{aligned}
\Delta \phi_\epsilon  &= \chi_\epsilon - m_\epsilon &&\text{ in }\Omega, \\
(n_{\partial\Omega} \cdot \nabla)\phi_\epsilon &= 0 && \text{ on }\partial \Omega, \\
\dashint_{\Omega} \phi_\epsilon \, dx  &= 0.
\end{aligned}\right.
\end{equation*}
As $\|\chi_\epsilon - m_\epsilon\|_{C^1}\leq C(\Omega)/\epsilon$, we can 
apply Schauder estimates to find
\begin{equation}\label{eqn:lazySchauder}
 \|\phi_\epsilon\|_{C^2(\Omega)}\leq C(\Omega)/\epsilon.
 \end{equation}
Noting the $L^1$ estimate
\begin{align*}
\|\chi_\epsilon - \chi\|_{L^1(\Omega)}\leq C(\Omega)\epsilon \left(1+ |\nabla \chi|(\Omega)\right)
\end{align*}
and
 $m_\epsilon \leq m_0,$ we have
\begin{equation}\label{eqn:lbForDiv}
\begin{aligned}
\int_\Omega \chi \nabla \cdot \phi_\epsilon \, dx 
& =\int_{\Omega} \chi (\chi_\epsilon - m_\epsilon)\, dx 
= (1-m_\epsilon)m_0|\Omega| + \int_\Omega \chi (\chi_\epsilon-\chi)\, dx \\
& \geq (1-m_0)m_0|\Omega| - C(\Omega)\epsilon \left(1+ |\nabla \chi|(\Omega)\right) \geq C(m_0,\Omega)
\end{aligned}
\end{equation}
where we have now fixed $\epsilon= \frac{(1-m_0)m_0|\Omega|}{4C(\Omega)(1+ |\nabla \chi|(\Omega))} .$
Choosing $\xi = \nabla  \phi_\epsilon$, by (\ref{eqn:lazySchauder}), (\ref{eqn:lbForDiv}), 
the Poincar\'e inequality for $w$, and the first variation formula
\begin{align*}
\delta \mu (\xi) = \int_{\overline \Omega {\times} \mathbb{S}^{d-1}} 
\left(\mathrm{Id} - s \otimes s\right):\nabla \xi(x) \, d \mu(x,s),
\end{align*}
we conclude (\ref{eqn:lambdaDes}) from (\ref{eqn:lambdaDef}).
\end{proof}

We finally state and prove a result which is helpful for the 
derivation of approximate Gibbs--Thomson laws from the optimality condition~\eqref{eq:minMovVariationalInterpol}
of De~Giorgi interpolants and is also needed in the proof
of Lemma~\ref{lem:metricSlope}.

\begin{lemma}
\label{lem:firstVarDeGiorgi}
Let $\chi \in \mathcal{M}_0$ and $B \in \mathcal{S}_\chi$,
and let $(\Psi_s)_{s \in (-\eta,\eta)}$ be an associated
family of diffeomorphisms from Lemma~\ref{lem:existenceDiffeos}.
Then, for any $\phi \in \smash{H^{1}_{(0)}}$ it holds
\begin{align}
\label{eq:dualityBound}
\bigg|\int_{\Omega} \phi\frac{\chi\circ\Psi^{-1}_s - \chi}{s} \,dx
+ \langle B\cdot\nabla\chi, \phi \rangle_{H^{-1}_{(0)},H^1_{(0)}}\bigg|
\leq \|\phi\|_{H^{1}_{(0)}} o_{s \to 0}(1).
\end{align}
In particular, taking the supremum over $\phi\in H^1_{(0)}$ 
with $\|\phi\|_{H^1_{(0)}}\leq 1$ in (\ref{eq:dualityBound}) 
implies $\frac{\chi\circ\Psi^{-1}_s - \chi}{s}\to -B \cdot \nabla \chi$
strongly in $H^{-1}_{(0)}$ as $s \to 0$.
\end{lemma} 

\begin{proof}
To simplify the notation, we denote $\chi_s : = \chi\circ \Psi_s^{-1}$.
Heuristically, one expects~\eqref{eq:dualityBound} by virtue
of the formal relation $\partial_s\chi_s|_{s=0} = - (B\cdot\nabla)\chi$ 
mentioned after~(\ref{def:MSmetricSlope2}). A rigorous argument is given as follows.

Using the product rule, we first expand~\eqref{eq:actionVel} as
\begin{equation}\label{eqn:expand}
\begin{aligned}
& \int_{\Omega} \phi\frac{\chi_s - \chi}{s} \,dx
+ \langle B\cdot\nabla\chi, \phi \rangle_{H^{-1}_{(0)},H^1_{(0)}}
\\& 
=  \int_{\Omega} \phi\frac{\chi_s - \chi}{s} \,dx
- \int_{\Omega} \chi \big((B\cdot\nabla)\phi + \phi\nabla\cdot B\big) \,dx.
\end{aligned}
\end{equation}
Recalling $\chi_s = \chi \circ \Psi^{-1}_s$, using the change of variables formula for the map $x~\mapsto~\Psi_s (x)$, and adding zero entails
\begin{equation}\label{eqn:expandInsert}
\begin{aligned}
\int_{\Omega} \phi\frac{\chi_s - \chi}{s} \,dx
= \int_{\Omega} \chi \frac{\phi\circ\Psi_s - \phi}{s} \,dx
+ \int_{\Omega} \chi (\phi\circ\Psi_s) \frac{|\det\nabla\Psi_s| - 1}{s} \,dx.
\end{aligned}
\end{equation}
Inserting (\ref{eqn:expandInsert}) into (\ref{eqn:expand}), we have that 
\begin{equation}\label{eq:error1}
\int_{\Omega} \phi\frac{\chi_s - \chi}{s} \,dx
+ \langle B\cdot\nabla\chi, \phi \rangle_{H^{-1}_{(0)},H^1_{(0)}} = I + II,
\end{equation}
 where
\begin{align*}
I &:= \int_{\Omega} \chi \frac{\phi\circ\Psi_s - \phi}{s} \,dx
- \int_{\Omega} \chi (B\cdot\nabla)\phi \,dx,
\\ 
II &:= \int_{\Omega} \chi (\phi\circ\Psi_s) \frac{|\det\nabla\Psi_s| - 1}{s}
- \int_{\Omega} \chi \phi\nabla\cdot B \,dx.
\end{align*}

To estimate $II$, we first Taylor expand $\nabla\Psi_s(x) = \mathrm{Id} + s\nabla B(x) + F_s(x)$
where, by virtue of the regularity of~$s\mapsto\Psi_s$ and~$B$, 
the remainder satisfies the upper bound $\sup_{x \in \Omega} |F_s(x)| \leq so_{s\to 0}(1)$.
In particular, from the Leibniz formula we deduce
\begin{equation}\label{eqn:detTaylor}
\det\nabla\Psi_s(x) - 1 = s(\nabla\cdot B)(x) + f_s(x),
\end{equation} where the remainder satisfies
the same qualitative upper bound as~$F_s(x)$. Note that by restricting to sufficiently small $s$,
we may ensure that $\det\nabla\Psi_s=|\det\nabla\Psi_s|$. Hence, using (\ref{eqn:detTaylor}), then adding zero to reintroduce the determinant for a change of variables, and applying the continuity of translation (by a diffeomorphism) in $L^2(\Omega)$, we have
\begin{equation}\label{eqn:IIestimate}
\begin{aligned}
II& = \|\phi\|_{L^2(\Omega)}o_{s \to 0}(1) + 
\int_{\Omega} \chi \nabla\cdot B(\phi\circ\Psi_s - \phi)\,dx \\
& =  \|\phi\|_{L^2(\Omega)}o_{s \to 0}(1) + \int_{\Omega} \phi \Big((\chi \nabla\cdot B)\circ\Psi_s^{-1} - \chi \nabla\cdot B\Big)\,dx \\
&~~~ 
+ \int_{\Omega} \chi\nabla \cdot B\left( \phi\circ \Psi_s\right) (1 {-} |\det\nabla\Psi_\lambda(x)|)\,dx \\
&\leq \|\phi\|_{L^2(\Omega)}o_{s \to 0}(1).
\end{aligned}
\end{equation}

To estimate~$I$, we first make use of the fundamental theorem 
of calculus along the trajectories determined by $\Psi_s$, reintroduce 
the determinant by adding zero as in (\ref{eqn:IIestimate}), and 
apply $\frac{d}{d\lambda} \Psi_\lambda (x) = B(x)+o_{\lambda\to 0}(1)$ to see
\begin{equation*}
\begin{aligned}
\int_{\Omega} \chi \frac{\phi\circ\Psi_s - \phi}{s} \,dx
&= \dashint_{0}^{s} \int_{\Omega} \chi \nabla\phi\big(\Psi_\lambda(x)\big)
\cdot\frac{d}{d\lambda}\Psi_\lambda(x)\,dx d\lambda
\\&
= \dashint_{0}^{s} \int_{\Omega} \chi \nabla\phi\big(\Psi_\lambda(x)\big)
\cdot B(x) |\det\nabla\Psi_\lambda(x)|\,dx d\lambda 
\\&~~~
+  \|\nabla \phi\|_{L^2(\Omega)}o_{s\to 0}(1).
\end{aligned}
\end{equation*}
Hence, by undoing the change of variables in the first term
on the right hand side of the previous identity, 
an argument analogous to the one for~$II$ guarantees
\begin{equation}\label{eqn:Iestimate}
I \leq \|\nabla \phi\|_{L^2(\Omega)}o_{s \to 0}(1).
\end{equation}
Looking to (\ref{eq:error1}), we use the two respective estimates 
for~$I$ and~$II$ given in (\ref{eqn:Iestimate}) and~(\ref{eqn:IIestimate}). By an
application of Poincar\'e's inequality, we arrive at~\eqref{eq:dualityBound}.
\end{proof}

\subsection{Proof of Theorem~\ref{theo:mainResult}}\label{subsec:proofMainResult}
We proceed in several steps.

\textit{Step 1: Approximate solution and approximate energy inequality.} 
For time discretization parameter $h \in (0,1)$ and initial condition $\chi_0 \in BV(\overline\Omega;\{0,1\})$, 
we define the sequence $\{\chi_n^h\}_{n \in \mathbb{N}_0}$ as in (\ref{eq:minMovStep}), and recall 
the piecewise constant function $\chi^h$ in (\ref{eq:minMovPiecewiseInterpol}) and the De~Giorgi 
interpolant $\bar{\chi}^h$ in (\ref{eq:minMovVariationalInterpol}).
We further define the linear interpolant~$\hat{\chi}^h$ by
\begin{equation}\label{eq:minMovLinearInterpol}
\hat{\chi}^h(t)  = \frac{nh-t}{h}\chi_{n-1}^h 
+ \frac{t - (n{-}1)h}{h} \chi_h^n \quad \text{for all } t \in [(n{-}1)h,nh),\,n\in\mathbb{N}.
\end{equation}
To capture fine scale behavior of the energy in the limit, we will introduce 
measures~$\mu^h = \mathcal{L}^1 \llcorner (0,T_*) \otimes (\mu^h_t)_{t\in (0,T_*)} \in 
\mathrm{M}((0,T_*) {\times} \overline\Omega {\times} \mathbb{S}^{d-1})$ 
so that for each $t \in (0,T_*)$ the total mass of the mass measure 
$|\mu^h_t|_{\mathbb{S}^{d-1}} \in \mathrm{M}(\overline\Omega)$ associated
with the oriented varifold $\mu^h_t \in \mathrm{M}(\overline\Omega {\times} \mathbb{S}^{d-1})$
is naturally associated to the energy of the De~Giorgi interpolant at time~$t$. 
More precisely, we define varifolds associated to the varifold lift of~$\bar \chi^h$ 
in the interior and on the boundary by
\begin{equation}
\begin{aligned}
\label{eq:approxVarifolds}
\mu^{h,\Omega} &:= \mathcal{L}^1 \llcorner (0,T_*) \otimes (\mu^{h,\Omega}_t)_{t\in (0,T_*)},
\\
\mu^{h,\Omega}_t &:= c_0|\nabla \bar \chi^h(\cdot,t)|\llcorner{\Omega} 
\otimes (\delta_{\frac{\nabla\bar\chi^h(\cdot,t)}{|\nabla\bar\chi^h(\cdot,t)|}(x)})_{x \in \Omega},
\end{aligned}
\end{equation}
respectively
\begin{equation}
\begin{aligned}
\label{eq:approxVarifoldsBoundary} 
\mu^{h,\partial\Omega} &:= \mathcal{L}^1 \llcorner (0,T_*) 
\otimes (\mu^{h,\partial\Omega}_t)_{t\in (0,T_*)},
\\
\mu^{h,\partial \Omega}_t &:= (\cos\alpha)c_0\bar\chi^h(\cdot,t)\llcorner{\partial \Omega} 
\otimes (\delta_{n_{\partial\Omega}(x)})_{x \in \partial\Omega},
\end{aligned}
\end{equation}
where $n_{\partial \Omega}$ denotes the inner normal on~$\partial \Omega$
and where we again perform an abuse of notation and do not distinguish
between~$\bar \chi^h(\cdot,t)$ and its trace along~$\partial\Omega$.
We finally define the total approximate varifold by
\begin{align}
\label{def:approxTotalVarifold}
\mu^h : = \mu^{h,\Omega} + \mu^{h,\partial \Omega}.
\end{align}

The remainder of the first step is concerned with the proof of the following
approximate version of the energy dissipation inequality~\eqref{eq:DeGiorgiInequality}:
for $\tau$ and $T$ such that $0<\tau<T<T_*$, and $h \in (0,T-\tau)$, 
we claim that
\begin{equation}
\label{eqn:discDissipation2}
\begin{aligned}
& |\mu^h_T|_{\mathbb{S}^{d-1}}(\overline\Omega) 
+ \int_0^{\tau} \frac{1}{2}\|\partial_t \hat \chi^h\|^2_{H^{-1}_{(0)}} 
+ \frac{1}{2}\bigg\|\frac{\bar \chi^{h}(t) - \bar\chi^h(\lfloor t/h \rfloor h) }
{t - \lfloor t/h \rfloor h}\bigg\|^2_{H^1_{(0)}} \, dt 
\leq  E[\chi_0].
\end{aligned}
\end{equation}
As a first step towards~\eqref{eqn:discDissipation2}, we claim
for all $n \in \mathbb{N}$ (cf. \cite{Ambrosio2005} and \cite{Chambolle2021})
\begin{equation}
\label{eqn:stepDissipation}
\begin{aligned}
&E[\bar \chi^h(nh)]
+ \frac{h}{2}\left\|\frac{\chi_n - \chi_{n-1}}{h}\right\|_{H^{-1}_{(0)}}^2 
+ \frac{1}{2}\int_{(n{-}1)h}^{nh} \bigg\| \frac{\bar \chi^{h}(t) 
- \bar\chi^h((n{-}1)h) }{t- h(n-1)}\bigg\|_{H^{-1}_{(0)}}^2 \, dt
\\ &
\leq E[\bar \chi^h((n{-}1)h)]
\end{aligned}
\end{equation}
and
\begin{align}
\label{eq:upperBoundEnergyDeGiorgiInterpolant}
E[\bar\chi^h(t)] \leq E[\bar\chi^h((n{-}1)h)]
\quad\text{for all } n \in \mathbb{N} \text{ and all } t \in ((n{-}1)h,nh).
\end{align}
In particular, using the definition~(\ref{eq:minMovLinearInterpol}) and then telescoping 
over~$n$ in~(\ref{eqn:stepDissipation}) provides
for all $n \in \mathbb{N}$ the discretized dissipation inequality
\begin{equation}
\label{eqn:discDissipation}
\begin{aligned}
&E[\bar \chi^h (nh)] 
+ \frac{1}{2}\int_0^{nh} \frac{1}{2}\|\partial_t \hat \chi^h\|^2_{H^{-1}_{(0)}} 
+ \frac{1}{2}\bigg\|\frac{\bar \chi^{h}(t) - \bar\chi^h(\lfloor t/h \rfloor h) }
{t - \lfloor t/h \rfloor h}\bigg\|^2_{H^1_{(0)}} \, dt 
\leq E[\chi_0].
\end{aligned}
\end{equation}

The bound~\eqref{eq:upperBoundEnergyDeGiorgiInterpolant}
is a direct consequence of the minimality of the 
interpolant~(\ref{eq:minMovVariationalInterpol}) at~$t$.
To prove~\eqref{eqn:stepDissipation}, and thus also~\eqref{eqn:discDissipation}, 
we restrict our attention to the interval $(0,h)$ and temporarily drop the 
superscript $h.$ We define the function 
\begin{equation}
\label{eq:defDeGiorgiValue}
f(t) : = E[\bar \chi(t)] + \frac{1}{2t} \left\|\bar \chi (t) - \chi_0\right\|_{H^{-1}_{(0)}}^2,
\quad t \in (0,h),
\end{equation}
and prove $f$ is locally Lipschitz in $(0,h)$ with
\begin{equation}\label{eqn:fDerivative}
\frac{d}{dt} f(t) = -\frac{1}{2t^2}\|\bar \chi (t) - \chi_0\|_{H^{-1}_{(0)}}^2
\quad\text{for a.e.\ } t \in (0,h).
\end{equation}

To deduce (\ref{eqn:fDerivative}), we first show 
\begin{align}
\label{eq:monotonicityDistance}
&(0,h] \ni t \mapsto \|\bar \chi (t) - \chi_0\|_{H^{-1}_{(0)}}
&&\text{is non-decreasing},
\\& \label{eq:monotonicityDeGiorgiValue}
(0,h] \ni t \mapsto f(t) &&\text{is non-increasing}.
\end{align}
Indeed, for $0 < s < t \leq h$ we obtain from minimality 
of the interpolant~(\ref{eq:minMovVariationalInterpol}) at~$s$,
then adding zero, and then from minimality of the 
interpolant~(\ref{eq:minMovVariationalInterpol}) at~$t$ that
\begin{align*}
f(s) &\leq E[\bar\chi(t)] + \frac{1}{2s}\|\bar \chi (t) - \chi_0\|_{H^{-1}_{(0)}}^2
\\&
\leq E[\bar\chi(s)] + \frac{1}{2t}\|\bar \chi (s) - \chi_0\|_{H^{-1}_{(0)}}^2
+ \Big(\frac{1}{2s} - \frac{1}{2t}\Big)\|\bar \chi (t) - \chi_0\|_{H^{-1}_{(0)}}^2.
\end{align*}
Recalling definition~\eqref{eq:defDeGiorgiValue},
this is turn immediately implies~\eqref{eq:monotonicityDistance}.
For a proof of~\eqref{eq:monotonicityDeGiorgiValue},
we observe for $s,t \in (0,h]$ by minimality of the 
interpolant~(\ref{eq:minMovVariationalInterpol}) at $t$ that
\begin{equation*}
f(t)-f(s)\leq \frac{1}{2t}\|\bar \chi (s) - \chi_0\|_{H^{-1}_{(0)}}^2 
- \frac{1}{2s}\|\bar \chi (s) - \chi_0\|_{H^{-1}_{(0)}}^2.
\end{equation*} 
Rearranging one finds for $0< s < t \leq h$
\begin{equation}
\label{eq:LipschitzBound1}
\frac{f(t)-f(s)}{t-s}\leq -\frac{1}{2ts}\|\bar \chi (s) - \chi_0\|_{H^{-1}_{(0)}}^2
\leq 0,
\end{equation}
proving~\eqref{eq:monotonicityDeGiorgiValue}.

Likewise using minimality of the interpolant~(\ref{eq:minMovVariationalInterpol}) at $s$, 
one also concludes for $0 < s < t < h$ the lower bound
\begin{equation}
\label{eq:LipschitzBound2}
\frac{f(t)-f(s)}{t-s} \geq -\frac{1}{2ts}\|\bar \chi (t) - \chi_0\|_{H^{-1}_{(0)}}^2.
\end{equation}
As the discontinuity set of a monotone function is at most countable,
we infer~(\ref{eqn:fDerivative}) from~\eqref{eq:LipschitzBound1},
\eqref{eq:LipschitzBound2} and~\eqref{eq:monotonicityDistance}.
Integrating (\ref{eqn:fDerivative}) on $(s,t)$, using 
optimality of the interpolant~(\ref{eq:minMovVariationalInterpol}) at~$s$
in the form of $f(s) \leq E[\bar\chi(0)] = E[\chi_0]$,
and using monotonicity of~$f$ from~\eqref{eq:monotonicityDeGiorgiValue}, we have
\begin{equation*}
E[\bar \chi(h)] + \frac{1}{2h} \left\|\bar \chi (h) - \chi_0\right\|_{H^{-1}_{(0)}}^2 
+ \int_s^t \frac{1}{2\tau^2}\|\bar \chi (\tau) - \chi_0\|_{H^{-1}_{(0)}}^2 d\tau \leq E[\chi_0].
\end{equation*}
Sending $s \downarrow 0$ and $t \uparrow h$, we recover (\ref{eqn:stepDissipation}) 
and thus also~\eqref{eqn:discDissipation}.

It remains to post-process~\eqref{eqn:discDissipation} to~\eqref{eqn:discDissipation2}.
Note first that by definitions~\eqref{eq:approxVarifolds}--\eqref{def:approxTotalVarifold}
it holds $|\mu^h_t|_{\mathbb{S}^{d-1}}(\overline\Omega) = E[\bar\chi^h(t)]$
for all $t \in (0,T_*)$. 
We claim that for all $h \in (0,1)$
\begin{align}
\label{eq:monotonicityEnergyDeGiorgiInterpolant}
(0,T_*) \ni t \mapsto |\mu^h_t|_{\mathbb{S}^{d-1}}(\overline{\Omega})
\text{ is non-increasing}.
\end{align}
Indeed, for $(n{-}1)h < s < t \leq nh$ we simply get from the
minimality of the De~Giorgi interpolant~(\ref{eq:minMovVariationalInterpol}) 
at time~$t$ 
\begin{align*}
&E[\bar\chi^h(t)] + \frac{1}{2(t - (n{-}1)h)}\|\bar\chi^h(t)
- \chi^h((n{-}1)h)\|^2_{H^{-1}_{(0)}}
\\
&\leq E[\bar\chi^h(s)] + \frac{1}{2(t - (n{-}1)h)}\|\bar\chi^h(s)
- \chi^h((n{-}1)h)\|^2_{H^{-1}_{(0)}}
\end{align*}
so that~\eqref{eq:monotonicityEnergyDeGiorgiInterpolant}
follows from~\eqref{eq:monotonicityDistance} and~\eqref{eq:upperBoundEnergyDeGiorgiInterpolant}.
Restricting our attention to $h\in (0,T-\tau)$, there is $n_0\in \mathbb{N}$ such that $\tau< n_0 h < T ,$ and by (\ref{eq:monotonicityEnergyDeGiorgiInterpolant}) and positivity of the integrand, we may bound the left-hand side of (\ref{eqn:discDissipation2}) by (\ref{eqn:discDissipation}) with $n= n_0$ completing the proof of (\ref{eqn:discDissipation2}).

\textit{Step 2: Approximate Gibbs--Thomson law.}
Naturally associated to the De Giorgi interpolant~(\ref{eq:minMovVariationalInterpol}) 
is the potential $w^h \in L^2(0,T_*; H^1_{(0)} )$ satisfying
\begin{equation}
\label{def:w_h}
\begin{aligned}
\Delta w^h(\cdot,t) &=  \frac{\bar{\chi}^h(t) -\bar \chi^h (\lfloor t/h \rfloor) }{t - \lfloor t/h \rfloor} 
&& \text{ in }\Omega, 
\\
(n_{\partial\Omega} \cdot \nabla) w^h(\cdot,t) &= 0  
&& \text{ on }\partial\Omega,
\end{aligned}
\end{equation}
for $t \in (0,T_*).$ Note that this equivalently expresses~\eqref{eqn:discDissipation2}
in the form of
\begin{equation}
\label{eqn:discDissipation3}
\begin{aligned}
& |\mu^h_T|_{\mathbb{S}^{d-1}}(\overline\Omega) 
+ \int_0^{\tau} \frac{1}{2}\|\partial_t \hat \chi^h\|^2_{H^{-1}_{(0)}} 
+ \frac{1}{2}\|\nabla w^h\|^2_{L^2(\Omega)} \, dt 
\leq  E[\chi_0].
\end{aligned}
\end{equation}

By the minimizing property~(\ref{eq:minMovVariationalInterpol})
of the De~Giorgi interpolant, Allard's first variation formula~\cite{Allard1972},
and Lemma~\ref{lem:firstVarDeGiorgi}, it follows that the De~Giorgi interpolant 
furthermore satisfies the approximate Gibbs--Thomson relation 
\begin{align}
\label{eq:approxGibbsThomson}
&  \int_{\overline\Omega {\times} \mathbb{S}^{d-1}}  
\left(\mathrm{Id} - s \otimes s\right) :\nabla B(x) \, d\mu_t^h(x,s)
= \int_\Omega \bar\chi^h(t) \nabla \cdot (w^h(\cdot,t)B) \, dx
\end{align} 
for all $t \in (0,T_*)$ and all $B \in \mathcal{S}_{\bar\chi^h(t)}$.

Applying the result of Lemma~\ref{lem:AbelsRoeger}
to control the Lagrange multiplier arising from the mass constraint, 
and using the uniform bound on the energy from the dissipation 
relation~(\ref{eqn:discDissipation}) and the estimate~\eqref{eq:upperBoundEnergyDeGiorgiInterpolant}, 
we find there exist functions $\lambda^h  \in  L^2(0,T)$ and a constant 
$C = C(\Omega,d,c_0,m_0,\chi_0)>0$ such that for all $t$ in $(0,T_*)$
\begin{equation}\label{eqn:discGT}
\int_{\overline\Omega {\times} \mathbb{S}^{d-1}}  
\left(\mathrm{Id} - s \otimes s\right) :\nabla B(x) \, d\mu_t^h(x,s)
=  \int_\Omega \bar\chi^h(t) \nabla \cdot \big((w^h(\cdot,t) + \lambda^h(t)) B \big) \, dx
\end{equation}
for all $B \in C^1(\overline\Omega;\Rd)$ with $(B\cdot n_{\partial\Omega})|_{\partial\Omega}\equiv 0$,
and
\begin{equation}\label{bdd:w+lambda}
\|w^h(\cdot, t) + \lambda^h(t) \|_{H^1(\Omega)} 
\leq C (1 + \|\nabla w^h(\cdot, t)\|_{L^2(\Omega)}).
\end{equation}

\textit{Step 3: Compactness, part I: Limit varifold.}
Based on the uniform bound on the energy from the dissipation 
relation~(\ref{eqn:discDissipation}) and the estimate~\eqref{eq:upperBoundEnergyDeGiorgiInterpolant}, 
we have by weak-$*$ compactness of finite Radon measures, up to
selecting a subsequence $h \downarrow 0$,
\begin{align}
\label{eqn:varifoldLimit}
\mu^{h,\Omega} &\wksto \mu^\Omega
&&\text{in } \mathrm{M}((0,T_*){\times}\overline{\Omega}{\times}\mathbb{S}^{d-1}),
\\ \label{eqn:varifoldBoundaryLimit}
\mu^{h,\partial\Omega} &\wksto \mu^{\partial\Omega}
&&\text{in } \mathrm{M}((0,T_*){\times}\partial\Omega{\times}\mathbb{S}^{d-1}).  
\end{align}
Define $\mu := \mu^{\Omega} + \mu^{\partial\Omega}$.

To ensure the required structure of the limit measures~$\mu^{\Omega}$
and~$\mu^{\partial\Omega}$, one may argue as follows. Thanks to the monotonicity~\eqref{eq:monotonicityEnergyDeGiorgiInterpolant},
one may apply \cite[Lemma~2]{Hensel2021l} to obtain that the limit measure~$\mu$ 
can be sliced in time as
\begin{equation}
\label{eqn:musliced}
\mu =  \mathcal{L}^1\llcorner{(0,T_*)} \otimes (\mu_t)_{t \in (0,T_*)},
\,\mu_t \in {\rm M}(\overline{\Omega} {\times} \mathbb{S}^{d-1})
\text{ for all } t \in (0,T_*),
\end{equation} 
and that the standard lower semi-continuity for measures can be applied 
to almost every slice in time --- precisely given by
\begin{equation}
\label{eqn:varifoldSlicingLSC}
|\mu_t|_{\mathbb{S}^{d-1}}(\overline\Omega)
\leq \liminf_{h\to 0}|\mu^h_t|_{\mathbb{S}^{d-1}}(\overline\Omega)
< \infty 
\quad \text{for a.e.\ } t \in (0,T_*).
\end{equation} 

Next, we show that $\mu^{\partial \Omega}$ satisfies
\begin{align}
\label{eqn:muBoundarySliced}
\mu^{\partial\Omega} = \mathcal{L}^1\llcorner{(0,T_*)} \otimes (\mu^{\partial\Omega}_t)_{t \in (0,T_*)},
\,\mu^{\partial\Omega}_t \in {\rm M}(\partial\Omega {\times} \mathbb{S}^{d-1})
\text{ for all } t \in (0,T_*),
\end{align}
together with~(\ref{eq:compSlice2}). To see that~(\ref{eqn:muBoundarySliced}) 
holds, recall that $|\mu^{h,\partial \Omega}_t|_{\mathbb{S}^{d-1}}$ is simply given by
the measure $g_h(\cdot,t) \mathcal{H}^{d-1}\llcorner\Omega$ 
for the trace $g_h(\cdot,t) : = c_0 \cos(\alpha)\bar \chi^h(\cdot,t)$. 
Due to $\|g_h\|_{L^\infty(\partial\Omega \times (0,T))}\leq c_0$, up to a subsequence, 
$g_h$ weakly converges to some $g$ in 
$L^2((0,T_*) {\times} \partial\Omega;\mathcal{L}^1\otimes \mathcal{H}^{d-1})$. 
From this, \eqref{eqn:varifoldBoundaryLimit} and~\eqref{eqn:varifoldSlicingLSC}, it 
follows that~$\mu^{\partial\Omega}$ has the structure~\eqref{eqn:muBoundarySliced}
and~(\ref{eq:compSlice2}) with
$|\mu_t^{\partial\Omega}|_{\mathbb{S}^{d-1}} = g(\cdot,t)\mathcal{H}^{d-1}\llcorner\partial\Omega$. 

Finally, we have to argue that
\begin{align}
\label{eqn:muBulkSliced}
\mu^{\Omega} = \mathcal{L}^1\llcorner{(0,T_*)} \otimes (\mu^{\Omega}_t)_{t \in (0,T_*)},
\,\mu^{\Omega}_t \in {\rm M}(\Omega {\times} \mathbb{S}^{d-1})
\text{ for all } t \in (0,T_*),
\end{align}
together with~(\ref{eq:compSlice1}).
However, \eqref{eqn:muBulkSliced} and~(\ref{eq:compSlice1}) 
directly follow from~\eqref{eqn:musliced}, \eqref{eqn:muBoundarySliced},
\eqref{eqn:varifoldSlicingLSC}, and the definition $\mu = \mu^\Omega + \mu^{\partial\Omega}$.

\textit{Step 4: Compactness, part II: Limit potential.}
Compactness for $w^h$ and $\lambda^h$ follows immediately 
from bounds~(\ref{eqn:discDissipation3}) and~(\ref{bdd:w+lambda})
as well as the Poincar\'e inequality, 
showing there is $w\in L^2(0,T;H^1_{(0)})$ and $\lambda \in L^2(0,T)$ 
such that, up to a subsequence, 
$w^h \wkto w$ in $L^2(0,T;H^1_{(0)})$ and $\lambda^h \wkto \lambda$ in $L^2(0,T)$.
In particular, $w^h + \lambda^h \wkto w + \lambda$ in $L^2(0,T;H^1(\Omega))$.

\textit{Step 5: Compactness, part III: Limit phase indicator.}
To obtain compactness of $\hat \chi^h$ and $\bar \chi^h$, we will use 
the classical Aubin--Lions--Simon compactness theorem (see \cite{aubin1963}, \cite{lions1969}, and \cite{simonCompact}). 
By~(\ref{eqn:discDissipation2}), the fundamental theorem of calculus, and 
Jensen's inequality, it follows that
\begin{equation}\label{eqn:ALShyp}
\int_0^{T_*-\delta} \| \chi^h(t {+} \delta) - \chi^h(t)\|_{H^{-1}_{(0)}}^2\, dt
\to 0 \quad \text{ uniformly in } h \text{ as }\delta\to 0
\end{equation}
(see, e.g., \cite[Lemma 4.2.7]{stinson-phd}). Looking to the dissipation (\ref{eqn:discDissipation2}) 
and recalling the definition of~$w^h$ in~(\ref{def:w_h}), one sees that
\begin{equation}\label{eqn:barchichi}
\int_0^{T_*}\| \bar \chi^h(t) - \chi^h(t)\|_{H^{-1}_{(0)}}^2\, dt\leq C_1 h^2.
\end{equation}
We \textit{claim} (\ref{eqn:ALShyp}) is also satisfied for $\bar \chi^h$ 
for a given sequence $h \to 0$. Fix $\epsilon > 0$, and choose $h_1>0$ such 
that $C_1h_1^2<\epsilon$. Choose $\delta_1>0$ such that the left hand side 
of~(\ref{eqn:ALShyp}) is bounded by $\epsilon$ uniformly in $h$ for $0<\delta <\delta_1$. 
By continuity of translation, which follows from density of smooth functions, we can suppose 
\begin{equation*}
\int_0^{T_*-\delta} \| \bar \chi^h(t+\delta) - \bar \chi^h(t)\|_{H^{-1}_{(0)}}^2\, dt 
< \epsilon \quad \text{ for }h>h_1 \text{ and }0<\delta<\delta_1.
\end{equation*}
Then by the triangle inequality one can directly estimate that for all $h$ 
(in the sequence) and $0<\delta <\delta_1$, we have
\begin{equation*}
\int_0^{T_*-\delta} \| \bar \chi^h(t+\delta) -\bar  \chi^h(t)\|_{H^{-1}_{(0)}}^2\, dt <3\epsilon,
\end{equation*}
proving the claim.
  With (\ref{eqn:ALShyp}), we may apply the Aubin--Lions--Simon compactness 
	theorem to $\bar \chi^h$ and $\hat \chi^h$ in the embedded spaces 
	$BV(\Omega) \hookrightarrow \hookrightarrow L^{p}(\Omega) \hookrightarrow H^{-1}_{(0)}$ 
	for some $6/5<p<1^*$ to obtain $\chi \in L^2(0,T_*;BV(\Omega;\{0,1\}))\cap H^1(0,T_*;H^{-1}_{(0)})$ 
	such that
\begin{equation}\label{eqn:chiStrongConvergence}
\chi^h, \bar \chi^h,\hat \chi^h \to \chi \quad \text{ in }L^2(0,T_*;L^2(\Omega))
\end{equation} 
(where we have used the Lebesgue dominated convergence theorem to 
move up to $L^2$ convergence). To see that the target of each approximation 
is in fact correctly written as a single function $\chi$, both $\chi^h$ and 
$\bar \chi^h$ must converge to the same limit by~(\ref{eqn:barchichi}).
Further, by the fundamental theorem of calculus, we have
\begin{equation}
\begin{aligned}
\|\chi^h(t) - \hat \chi^h(t)\|_{H^{-1}_{(0)}}  &=
\|\chi^h(ih) - \hat \chi^h(t)\|_{H^{-1}_{(0)}} 
\\&
\leq \int_{ih}^t \|\partial_t \hat \chi^h\|\, dt 
\leq h^{1/2}\|\partial_t \hat \chi^h\|_{L^2(0,T;H^{-1}_{(0)})},
\end{aligned}
\end{equation}
for some $i\in \mathbb{N}_0,$ which shows that $\chi^h$ and $\hat \chi^h$ 
also converge to the same limit, thereby justifying (\ref{eqn:chiStrongConvergence}). 
Finally, note the dimension dependent embedding was introduced for technical convenience 
to ensure that $L^p(\Omega)\hookrightarrow \smash{H^{-1}_{(0)}}$ is well defined, but can be 
circumnavigated (see, e.g., \cite{kroemerLaux2021}). 

We finally note that the distributional formulation of the initial 
condition survives passing to the limit for $\hat \chi^h$ as 
\begin{equation}\label{eqn:distributionalInitialCondition}
\int_0^{T_*} \left(\langle\partial_t \chi , \zeta \rangle_{H^{-1}_{(0)},H^1_{(0)}}  
+ \int_\Omega \chi \partial_t \zeta \, dx\right) dt
= - \int_\Omega \chi_0 \zeta(x,0) \, dx
\end{equation}
for all $\zeta \in C^1_c(\overline\Omega\times [0,T_*))\cap H^1(0,T_*;H^1_{(0)}).$ 
As the trace of a function in $H^1(0,T;H^{-1}_{(0)})$ exists in $H^{-1}_{(0)}$ (see \cite{LM-v1}), (\ref{eqn:distributionalInitialCondition}) implies 
\begin{equation}\label{eqn:traceInitialCondition}
{\rm Tr}|_{t=0}\chi = \chi_0 \quad \text{ in } H^{-1}_{(0)}.
\end{equation}

\textit{Step 6: Compatibility conditions for limit varifold.}
Returning to the definition of $\mu^{h,\Omega}$,
for any $\phi \in C^1_c(\overline\Omega {\times} (0,T_*);\Rd)$
with $(\phi\cdot n_{\partial\Omega})|_{\partial\Omega{\times}(0,T_*)} \equiv 0$,
\begin{equation*}
\begin{aligned}
\int_0^{T_*} \int_{\overline\Omega {\times}\mathbb{S}^{d-1}} \phi(x)\cdot s 
\, d\mu^{h,\Omega}_t(x,s) dt 
&= c_0\int_0^{T_*}\int_{ \Omega} \phi\cdot d\nabla \bar \chi^h \, dt 
\\&
= - c_0\int_0^{T_*}\int_{ \Omega} \bar\chi^h\nabla\cdot\phi \,dx dt.
\end{aligned}
\end{equation*}
Using the convergence from (\ref{eqn:chiStrongConvergence})
as well as the conclusions from Step~3 of this proof, we can pass to the 
limit on the left- and right-hand side of the above equation, undo the divergence 
theorem, and localize in time to find (using also a straightforward approximation argument) 
that for a.e.\ $t \in (0,T)$, it holds
\begin{equation}\label{eqn:compatibilityIntegralForm}
\int_{\overline\Omega {\times}\mathbb{S}^{d-1}} \phi(x)\cdot s \, d\mu^\Omega_t(x,s)   
= \int_{ \Omega} \phi\cdot c_0d\nabla  \chi 
\end{equation}
for all $\phi \in C^1_c(\overline \Omega;\Rd)$ 
with $(\phi\cdot n_{\partial\Omega})|_{\partial\Omega} \equiv 0$
(the null set indeed does not 
depend on the choice of the test vector field~$\phi$ as $C^1_c(\overline \Omega;\Rd)$
normed by $f \mapsto \|f\|_\infty + \|\nabla f\|_\infty$ is separable). Taking the supremum 
in the above equation over $\phi \in C_c(U)$ 
for any $U \subset \Omega$ one has $c_0|\nabla \chi|(U)\leq |\mu_t^\Omega|(U)$, which 
by outer regularity implies (\ref{eq:comp1}). 

Let now $\phi \in C_c(\partial\Omega {\times} (0,T_*))$
and fix a $C^1(\Omega)$ extension~$\xi$ of the vector
field $(\cos\alpha) n_{\partial\Omega} \in C^1(\partial\Omega)$ (e.g., by 
multiplying the gradient of the signed distance function
for~$\Omega$ by a suitable cutoff localizing to a small enough
tubular neighborhood of~$\partial\Omega$). Recalling 
the definition of $\mu^{h,\partial\Omega}$, we obtain
\begin{equation*}
\begin{aligned}
&\int_0^{T_*} \int_{\overline\Omega {\times}\mathbb{S}^{d-1}} 
\phi(x)\xi(x)\cdot s \, d\mu^{h,\Omega}_t(x,s) dt 
+ \int_0^{T_*} \int_{\partial\Omega} \phi \, d|\mu^{h,\Omega}_t|_{\mathbb{S}^{d-1}} dt 
\\
&= c_0\int_0^{T_*}\int_{ \Omega} \phi\xi\cdot d\nabla \bar \chi^h \, dt 
+ c_0\int_{0}^{T_*} \int_{\partial\Omega} \bar\chi^h\phi\xi\cdot n_{\partial\Omega}
\,d\mathcal{H}^{d-1}dt
\\&
= - c_0\int_0^{T_*}\int_{ \Omega} \bar\chi^h\nabla\cdot(\phi\xi) \,dx dt,
\end{aligned}
\end{equation*}
which as before implies for a.e.\ $t \in (0,T_*)$
\begin{align*}
&\int_{\overline\Omega {\times}\mathbb{S}^{d-1}} 
\phi(x)\xi(x)\cdot s \, d\mu^{\Omega}_t(x,s) 
+ \int_{\partial\Omega} \phi \, d|\mu^{\Omega}_t|_{\mathbb{S}^{d-1}} 
\\&
= \int_{ \Omega} \phi\xi\cdot c_0d\nabla \bar \chi^h  
+ \int_{\partial\Omega} \phi (\cos\alpha)c_0\bar\chi^h\,d\mathcal{H}^{d-1}.
\end{align*}
Sending $\xi \to (\cos\alpha) n_{\partial\Omega} \chi_{\partial\Omega}$
and varying $\phi \in C^1_c(\partial\Omega)$ now implies (\ref{eq:comp2}).
Note that together with Step~3 of this proof, we thus
established item~\textit{i)} of Definition~\ref{def:VarifoldSolution}
with respect to the data~$(\chi,\mu)$.

\textit{Step 7: Gibbs--Thomson law in the limit and generalized mean curvature.} 
We can multiply the Gibbs--Thomson relation~(\ref{eqn:discGT}) 
by a smooth and compactly supported test function on $(0,T_*)$,
integrate in time, pass to the limit as $h \downarrow 0$ using the
compactness from Steps~3 to~5 of this proof, and then localize
in time to conclude that for a.e.\ $t \in (0,T_*)$, it holds
\begin{align}
&\int_{\overline\Omega{\times}\mathbb{S}^{d-1}} 
(\mathrm{Id} - s\otimes s):\nabla B(x) \, d\mu_t(x,s) 
= \int_\Omega \chi(\cdot,t) \nabla\cdot\big((w(\cdot,t) {+}\lambda(t))B\big) \, dx, 
\label{eqn:GTasSobolev}
\end{align}
for all $B \in C^1(\overline\Omega;\Rd)$ with $(B\cdot n_{\partial\Omega})|_{\partial\Omega}\equiv 0$
(the null set again does not depend on the choice of~$B$ due to the 
separability of the space $C^1(\overline\Omega;\Rd)$ normed by 
$f \mapsto \|f\|_\infty + \|\nabla f\|_\infty$).
Note that the left-hand side of (\ref{eqn:GTasSobolev}) 
is precisely $\delta \mu_t (B)$ by Allard's first variation formula~\cite{Allard1972}.
Finally, by Proposition~\ref{prop:firstVariationTangential}, 
the Gibbs--Thomson relation~(\ref{eqn:GTasSobolev}) can be expressed 
as in~(\ref{eq:repTangentialFirstVariation}) with the trace of $w+\lambda$ replacing 
$c_0 H_{\chi}$. Directly following the work of Section 4 in R\"oger~\cite{Roeger2005}, 
which applies for compactly supported variations in $\Omega$,  we conclude that the 
trace of $\frac{w+\lambda}{c_0}$ is given by the generalized mean curvature $H_{\chi}$, 
intrinsic to the surface $\supp |\nabla\chi| \llcorner \Omega,$ for almost every $t$ in $(0,T^*)$. 
Recalling the integrability guaranteed by Proposition \ref{prop:firstVariationTangential}, 
(\ref{eq:repTangentialFirstVariation}) and the curvature 
integrability~\eqref{eq:integrabilityWeakMeanCurvature} are satisfied.
In particular, we proved item~\textit{ii)} of Definition~\ref{def:VarifoldSolution}
with respect to the data~$(\chi,\mu)$.

\textit{Step 8: Preliminary optimal energy dissipation relation.}
By the compactness from Steps~3 to~5 of this proof
for the terms arising in (\ref{eqn:discDissipation3}), 
lower semi-continuity of norms, Fatou's inequality, inequality (\ref{eqn:varifoldSlicingLSC}), 
and first taking $h \downarrow 0$ and afterward $\tau \uparrow T$, we obtain
for a.e.\ $T \in (0,T_*)$
\begin{equation}\label{eqn:Dissipation}
|\mu_T|_{\mathbb{S}^{d-1}}(\overline\Omega) 
+ \int_0^{T} \frac{1}{2}\|\partial_t \chi\|^2_{H^{-1}_{(0)}} 
+ \frac{1}{2}\|\nabla w\|^2_{H^1_{(0)}} \, dt \leq E[\chi_0].
\end{equation}
Due to the previous two steps, it remains to upgrade~\eqref{eqn:Dissipation}
to~\eqref{eq:DeGiorgiInequality} to prove that~$(\chi,\mu)$
is a a varifold solution for Mullins--Sekerka 
flow~\eqref{eq:MullinsSekerka1}--\eqref{eq:MullinsSekerka6}
in the sense of Definition~\ref{def:VarifoldSolution}.

\textit{Step 9: Metric slope.} 
Let $t\in (0,T_*)$ be such that (\ref{eqn:GTasSobolev}) holds. 
For $B \in \mathcal{S}_{\chi(\cdot, t)}$, let $w_B \in H^{-1}_{(0)}$
solve the Neumann problem 
\begin{equation}
\begin{aligned}
\label{def:w_B}
\Delta w_B &=  B \cdot \nabla \chi(\cdot, t)
&&\text{in } \Omega,  
\\
(n_{\partial\Omega}\cdot\nabla) w_B &= 0
&&\text{on } \partial\Omega.
\end{aligned}
\end{equation}
Note by definition of the $H^{-1}_{(0)}$ norm, $\|\nabla w_B\|_{L^2(\Omega)} 
= \|B \cdot \nabla \chi(\cdot, t)\|_{H^{-1}_{(0)}}$.
From the Gibbs--Thomson relation~(\ref{eqn:GTasSobolev}), we have
\begin{equation}
\delta \mu_t (B) = \int_\Omega \chi(\cdot,t) \nabla\cdot(w(\cdot,t)B) \, dx 
= \int_{\Omega} \nabla w(\cdot,t) \cdot \nabla w_B \, dx. 
\end{equation}
Computing the norm of the projection of $w_B$ 
onto $\mathcal{G}_{\chi(\cdot,t)}$ (see (\ref{def:GA})) and recalling 
the inequality $\frac{1}{2}(a/b)^2 \geq a - \frac{1}{2}b^2$, we have
\begin{equation}\label{eqn:potentialVmetricSlope}
\begin{aligned}
\frac{1}{2}\|\nabla w(\cdot,t)\|_{L^2(\Omega)}^2 
&\geq \frac{1}{2}\left(\sup_{B \in \mathcal{S}_{\chi(\cdot,t)}} 
\frac{\int_\Omega \nabla w(\cdot,t) \cdot \nabla w_B}{ \|\nabla w_B\|_{L^2(\Omega)}}\right)^2 
\\&
\geq \sup_{B\in \mathcal{S}_\chi}\left\{\delta \mu_t (B) 
- \frac{1}{2}\|B \cdot \nabla \chi(\cdot, t)\|_{\mathcal{V}_{\chi(\cdot,t)}}^2\right\}
\end{aligned}
\end{equation}
for a.e.\ $t \in (0,T_*)$.

\textit{Step 10: Time derivative of phase indicator.} 
In this step, we show $\partial_t\chi(\cdot,t) \in \mathcal{T}_{\chi(\cdot,t)}$
(see~\eqref{def:tangentVelocities}) for a.e.\ $t \in (0,T_*)$. As for the 
metric slope term, cf.\ \eqref{eqn:potentialVmetricSlope}, we use potentials as a
convenient tool.

From the dissipation (\ref{eqn:Dissipation}), $\chi \in H^1(0,T_*;H^{-1}_{(0)})$ 
and there is $u \in H^1(0,T_*;H^{1}_{(0)})$ such that for almost every $t \in (0,T_*)$ 
the equation $\partial_t \chi(t) = \Delta_N u(t)$ holds. For any 
$\zeta \in C_c^\infty (\overline\Omega \times [0,T_*))\cap L^2(0,T;H^1_{(0)})$ 
via a short mollification argument, one can compute the derivative in 
time of $\langle \chi,\zeta \rangle_{H^{-1}_{(0)},H^1_{(0)}} = (\chi, \zeta)_{L^2(\Omega)}$ 
to find (recall~\eqref{eqn:traceInitialCondition})
\begin{equation}\label{eqn:distrEvolEquation}
\begin{aligned}
					&\int_{\Omega} \chi(\cdot,T)\zeta(\cdot,T) \,dx 
					- \int_{\Omega} \chi_{0}\zeta(\cdot,0) \,dx
					\\&
					= \int_{0}^{T} \int_{\Omega} \chi \p_t\zeta \,dxdt
					- \int_{0}^{T} \int_{\Omega} \nabla u \cdot \nabla\zeta \,dxdt
					\end{aligned}
\end{equation}
for almost every $T< T^*.$
To see that (\ref{eqn:distrEvolEquation}) holds for 
general $\zeta \in C_c^\infty (\overline\Omega \times [0,T_*))$ it suffices 
to check the equation for $c(t) = \intbar_\Omega \zeta(\cdot,t)\, dx$. In 
this case, the left hand side of (\ref{eqn:distrEvolEquation}) becomes $m_0(c(T) - c(0))$ 
and the right hand side becomes $m_0 \int_0^T \partial_t c(t) = m_0(c(T) - c(0)),$ 
verifying the assertion. Finally truncating 
a given test function $\zeta$ on the interval $(T,T_*)$, we have that for almost every $T< T^*$, 
equation (\ref{eqn:distrEvolEquation}) holds for all $\zeta \in C^\infty (\overline\Omega \times [0,T))$.

Note that for almost every $t \in (0,T_*)$ 
\begin{align}
\label{eq:LebesguePointPotential}
\lim_{\tau\downarrow 0} \dashint_{t-\tau}^{t+\tau}
\int_{\Omega} |\nabla u(x,t')-\nabla u(x,t)|^2 \,dx dt' = 0.
\end{align}
Furthermore, for a.e.\ $t\in(0,T_*)$ there is a set $A(t)\subset\Omega$
associated to $\chi$ in the sense that $\chi(\cdot,t) = \chi_{A(t)}$. 
As a consequence of~\eqref{eq:comp1}, \eqref{eq:integrabilityWeakMeanCurvature}, 
and~\eqref{eq:repTangentialFirstVariation} (see equation (2.13) of \cite{Schaetzle2001}), 
for almost every $t \in (0,T_*)$, there exists a measurable subset 
$\widetilde A(t) \subset \Omega$ representing a modification of~$A(t)$ in the sense
\begin{equation}\label{eq:ModificationPhase}
\begin{aligned}
\mathcal{L}^d\big(A(t) \Delta \widetilde A(t)\big) &= 0,
\quad \widetilde A(t) \text{ is open}, \\
\partial \widetilde A(t) \cap \Omega = \overline{\partial^*\widetilde A(t)} \cap \Omega
&\subset \supp |\mu_t^\Omega|_{\mathbb{S}^{d-1}} \llcorner \Omega,
\\
\supp |\nabla \chi(\cdot,t)| &= \partial \widetilde A(t).
\end{aligned}
\end{equation} 

We now claim that for almost every $t \in (0,T_*)$ it holds
\begin{align}
\label{eq:harmonicPotential}
\Delta u(\cdot,t) &= 0  &&\text{in } \Omega \setminus \partial \widetilde A(t)
\end{align}
in a distributional sense. In other words by (\ref{def:HA}), (\ref{eqn:HchiHarmonicRelation}), and (\ref{eq:ModificationPhase}), for almost every $t \in (0,T_*)$
\begin{align}
\label{eq:harmonicPotential2}
\partial_t\chi(\cdot,t) \in \mathcal{T}_{\chi(\cdot,t)}.
\end{align}
For a proof of~\eqref{eq:harmonicPotential}, fix $t \in (0,T_*)$ such 
that (\ref{eqn:distrEvolEquation}), \eqref{eq:LebesguePointPotential}, and \eqref{eq:ModificationPhase}
are satisfied. Fix $\zeta \in C^\infty_{c}(\Omega\setminus\overline{\widetilde A(t)};[0,\infty))$,
consider a sequence $s \downarrow t$ so that one may also apply~\eqref{eqn:distrEvolEquation}
for the choices $T=s$. Using~$\zeta$ as a constant-in-time test function in~\eqref{eqn:distrEvolEquation}
for $T=s$ and $T=t$, respectively, it follows from the nonnegativity
of~$\zeta$ with the first item of~\eqref{eq:ModificationPhase} that
\begin{align*}
0 \leq \frac{1}{s{-}t}\int_{\Omega} \chi_{A(s)} \zeta \,dx
= - \dashint_{t}^{s} \int_{\Omega} \nabla u \cdot \nabla \zeta \,dxdt'.
\end{align*}
Hence, for $s \downarrow t$ we deduce from the previous display
as well as~\eqref{eq:LebesguePointPotential} that
\begin{align*}
\int_{\Omega} \nabla u \cdot \nabla \zeta \,dx \leq 0
\quad\text{for all } \zeta \in C^\infty_{c}(\Omega\setminus\overline{\widetilde A(t)};[0,\infty)).
\end{align*}
Choosing instead a sequence $s \uparrow t$ so that one may apply~\eqref{eqn:distrEvolEquation}
for the choices $T=s$, one obtains similarly
\begin{align*}
\int_{\Omega} \nabla u \cdot \nabla \zeta \,dx \geq 0
\quad\text{for all } \zeta \in C^\infty_{c}(\Omega\setminus\overline{\widetilde A(t)};[0,\infty)).
\end{align*}
In other words, \eqref{eq:harmonicPotential} is satisfied 
throughout~$\Omega\setminus\overline{\widetilde A(t)}$
for all nonnegative test functions in~$H^1(\Omega\setminus\overline{\widetilde A(t)})$, 
hence also for all nonpositive test functions in~$H^1(\Omega\setminus\overline{\widetilde A(t)})$, 
and therefore for all smooth and compactly supported test functions 
in~$\Omega\setminus\overline{\widetilde A(t)}$. Adding and subtracting $\smash{\int_{\Omega} \zeta \,dx}$
to the left hand side of~\eqref{eqn:distrEvolEquation}, one may finally show along
the lines of the previous argument that \eqref{eq:harmonicPotential} is also satisfied 
throughout~$\widetilde A(t)$.

\textit{Step~11: Conclusion.}
We may now conclude the proof that $(\chi,\mu)$ is a varifold solution
for Mullins--Sekerka flow in the sense of Definition~\ref{def:VarifoldSolution},
for which it remains to verify item~\textit{iii)}. However, the desired
energy dissipation inequality \`{a} la De~Giorgi~\eqref{eq:DeGiorgiInequality}
now directly follows from~\eqref{eqn:Dissipation}, \eqref{eqn:potentialVmetricSlope}
and~\eqref{eq:harmonicPotential2}.

\textit{Step 12: $BV$~solutions.} 
Let~$\chi$ be a subsequential limit point
as obtained in~\eqref{eqn:chiStrongConvergence}. We now show that if the time-integrated energy convergence assumption~\eqref{eq:energyConvergenveAssumpMinMov} is satisfied then $\chi$ is a $BV$~solution. The main difficulty in proving this is showing that there exists a subsequence $h \downarrow 0$ such that
the De~Giorgi interpolants satisfy
\begin{align}
\label{eq:strictBVConvergence}
\bar\chi^h(\cdot,t) \to \chi(\cdot,t) \text{ strictly in } BV(\Omega;\{0,1\})
\text{ for a.e.\ } t \in (0,T_*)
\text{ as } h \downarrow 0.
\end{align}

Before proving~\eqref{eq:strictBVConvergence}, let us show how this
concludes the result. First, since~\eqref{eq:strictBVConvergence}
in particular means $|\nabla\bar\chi^h(\cdot,t)|(\Omega) \to |\nabla\chi(\cdot,t)|(\Omega)$
for a.e.\ $t \in (0,T_*)$ as $h \downarrow 0$, it follows from
Reshetnyak's continuity theorem, cf.\ \cite[Theorem~2.39]{AmbrosioFuscoPallara}, that
$$\mu^\Omega := \mathcal{L}^1 \llcorner (0,T_*) \otimes
\big(c_0\,d|\nabla\chi(\cdot,t)|\llcorner\Omega
\otimes (\delta_{\frac{\nabla\chi(\cdot,t)}
{|\nabla\chi(\cdot,t)|}(x)})_{x \in \Omega}\big)_{t \in (0,T_*)}$$
is the weak limit of $\mu^{h,\Omega}$, i.e., $\mu^{h,\Omega} \to \mu^{\Omega}$ weakly$^*$ in $\mathrm{M}((0,T_*)
{\times}\Omega{\times}\mathbb{S}^{d-1})$ as $h \to 0$.
Second, due to~\eqref{eq:strictBVConvergence} it follows from 
$BV$~trace theory, cf.\ \cite[Theorem~3.88]{AmbrosioFuscoPallara}, that
\begin{align}
\label{eq:traceConvergence}
\bar\chi^h(\cdot,t) \to \chi(\cdot,t) \text{ strongly in } L^1(\partial\Omega)
\text{ for a.e.\ } t \in (0,T_*)
\text{ as } h \downarrow 0.
\end{align}
Hence, defining $$\mu^{\partial\Omega} := \mathcal{L}^1 \llcorner (0,T_*) \otimes
\smash{\big(c_0(\cos\alpha)\chi(\cdot,t)\,\mathcal{H}^{d-1} \llcorner \partial\Omega
\otimes (\delta_{n_{\partial\Omega}(x)})_{x \in \partial\Omega}\big)_{t \in (0,T_*)}},$$
we have $\mu^{h,\partial\Omega} \to \mu^{\partial\Omega}$ weakly$^*$ in $\mathrm{M}((0,T_*)
{\times}\partial\Omega{\times}\mathbb{S}^{d-1})$ as $h \to 0$.
In summary, the relations~\eqref{eq:BVsol1} and~\eqref{eq:BVsol2} hold true as required.
Furthermore, defining $\mu := \mu^\Omega + \mu^{\partial\Omega}$, it follows from
the arguments of the previous steps that $(\chi,\mu)$ is a varifold solution
in the sense of Definition~\ref{def:VarifoldSolution}. In other words, $\chi$ is a 
$BV$~solution as claimed.

We now prove ~\eqref{eq:strictBVConvergence}.
To this end, we first show that~\eqref{eqn:chiStrongConvergence}
implies that there exists a subsequence $h \downarrow 0$
such that $E[\bar\chi^h(\cdot,t)] \to E[\chi(\cdot,t)]$
for a.e.\ $t \in (0,T_*)$. Indeed, since $E[\bar\chi^h(\cdot,t)]
\leq E[\chi^h(\cdot,t)]$ for all $t \in (0,T_*)$ by the optimality
constraint for a De~Giorgi interpolant (\ref{eq:minMovVariationalInterpol}), we may estimate
using the elementary relation $|a|=a+2a_-$ that
\begin{align*}
&\int_0^{T_*} \big|E[\bar\chi^h(\cdot,t)] - E[\chi(\cdot,t)]\big| \,dt
\\&
= \int_0^{T_*} \big( E[\bar\chi^h(\cdot,t)] - E[\chi(\cdot,t)] \big) \,dt
+ \int_0^{T_*} 2\big( E[\bar\chi^h(\cdot,t)] - E[\chi(\cdot,t)] \big)_- \,dt
\\&
\leq \int_0^{T_*} \big( E[\chi^h(\cdot,t)] - E[\chi(\cdot,t)] \big) \,dt
+ \int_0^{T_*} 2\big( E[\bar\chi^h(\cdot,t)] - E[\chi(\cdot,t)] \big)_- \,dt.
\end{align*}
The first right hand side term of the last inequality vanishes in
the limit ${h \downarrow 0}$ by assumption~\eqref{eq:energyConvergenveAssumpMinMov}.
For the second term on the right-hand side, we note that~\eqref{eqn:chiStrongConvergence}
entails that, up to a subsequence, we have 
$\bar\chi^h(\cdot,t) \to \chi(\cdot,t)$ strongly in $L^1(\Omega)$
for a.e.\ $t \in (0,T_*)$ as $h \downarrow 0$. Hence, the lower-semicontinuity
result of Modica~\cite[Proposition~1.2]{Modica1987} tells us that
$(E[\bar\chi^h(\cdot,t)] - E[\chi(\cdot,t)])_- \to 0$ pointwise a.e.\ in $(0,T_*)$
as $h \downarrow 0$, which in turn by Lebesgue's dominated convergence theorem
guarantees that the second term on the right-hand side of the last inequality vanishes
in the limit $h \downarrow 0$. In summary, for a suitable subsequence $h \downarrow 0$
\begin{align}
\label{eq:pointwiseConvergenceEnergies}
E[\bar\chi^h(\cdot,t)] &\to E[\chi(\cdot,t)]
&&\text{for a.e.\ } t \in (0,T_*),
\\ \label{eq:pointwiseL1Convergence}
\bar\chi^h(\cdot,t) &\to \chi(\cdot,t)
&&\text{strongly in } L^1(\Omega) \text{ for a.e.\ } t \in (0,T_*).
\end{align}

Now, due to~\eqref{eq:pointwiseL1Convergence} and the definition of strict convergence in~$BV(\Omega)$, \eqref{eq:strictBVConvergence} will follow if the total variations converge, i.e.,
\begin{align}
\label{eq:pointwiseConvergenceTV}
|\nabla\bar\chi^h(\cdot,t)|(\Omega) &\to |\nabla\chi(\cdot,t)|(\Omega)
&&\text{for a.e.\ } t \in (0,T_*) \text{ as } h \downarrow 0.
\end{align}
However, this is proven in a more general context 
(i.e., for the diffuse interface analogue
of the sharp interface energy~\eqref{def:energy})
in~\cite[Lemma~5]{Hensel2021d}. More precisely,
thanks to~\eqref{eq:pointwiseConvergenceEnergies} and~\eqref{eq:pointwiseL1Convergence},
one may simply apply the argument of~\cite[Proof of Lemma~5]{Hensel2021d}
with respect to the choices $\psi(u) := c_0 u$, $\sigma(u) := (\cos\alpha)c_0 u$
and $\tau(u) := (\sigma\circ\psi^{-1})(u) = (\cos\alpha) u$
to obtain~\eqref{eq:pointwiseConvergenceTV}, which in turn concludes the proof of Theorem~\ref{theo:mainResult}.
\qed

\subsection{Proofs for further properties of varifold solutions}\label{subsec:proofsFurtherProperties}
In this subsection, we present the proofs for the various
further results on varifold solutions to Mullins--Sekerka flow
as mentioned in Subsection~\ref{subsec:furtherProperties}.

\begin{proof}[Proof of Lemma~\ref{lem:metricSlope}]
The proof is naturally divided into two parts.

\textit{Step 1: Proof of ``$\leq$'' in~\eqref{def:MSmetricSlope2}
without assuming~\eqref{eq:repTangentialFirstVariation}.}
To simplify the notation, we denote $\chi_s : = \chi\circ \Psi_s^{-1}$
resp.\ $\mu_s : = \mu\circ \Psi_s^{-1}$ and 
abbreviate the right-hand side of~(\ref{def:MSmetricSlope2}) as 
\begin{align}
\label{eq:abbrevRHS}
A := \sup_{\substack{\partial_s \chi_s|_{s=0} = -B \cdot \nabla \chi 
\\ \chi_s \to \chi, \ B\in \mathcal{S}_{\chi}}} 
\limsup_{s\downarrow 0} \frac{(E[\mu_s] - E[\mu])_{+}}{\|\chi_s - \chi\|_{H^{-1}_{(0)}}}.
\end{align}

Fixing a flow $\chi_s$ such that $\chi_s \to \chi$ as $s \to 0$ with 
$\partial_s \chi_s|_{s=0} = -B\cdot \nabla \chi$ for some $B \in \mathcal{S}_{\chi}$, 
we claim that the upper bound~\eqref{def:MSmetricSlope2} follows from the assertions
\begin{align}
\label{eq:firstVarAux1}
\lim_{s \to 0} \frac{1}{s}\big(E [\mu_s] - E[\mu]\big) &= \delta\mu(B),
\\ \label{eq:firstVarAux2}
\lim_{s \to 0} \frac{1}{s}\|\chi_s - \chi\|_{H^{-1}_{(0)}} 
&= \|B\cdot\nabla\chi\|_{H^{-1}_{(0)}}.
\end{align}
Indeed, multiplying the definition in~\eqref{eq:abbrevRHS} 
by $1 = s/s$, using the inequality $\frac{1}{2}(a/b)^2 \geq a - \frac{1}{2}b^2$,
and recalling the notation~\eqref{def:normVel}, we find
\begin{equation}\label{eqn:AgeqMetric}
\begin{aligned}
\frac{1}{2}A^2 &\geq \lim_{s \to 0} \frac{E [\mu_s] - E[\mu]}{s} - \frac{1}{2} 
\lim_{s \to 0} \left( \frac{\|\chi_s - \chi\|_{H^{-1}_{(0)}}}{s} \right)^2 \\
 & =\delta \mu (B) - \frac{1}{2} \|B\cdot\nabla\chi\|_{\mathcal{V}_\chi}.
 \end{aligned}
\end{equation}
Recalling Lemma~\ref{lem:existenceDiffeos}, taking the supremum 
over $B \in \mathcal{S}_\chi$ thus yields ``$\leq$'' in~\eqref{def:MSmetricSlope2}.
It therefore remains to establish~\eqref{eq:firstVarAux1} and~\eqref{eq:firstVarAux2}.
However, the former is a classical and well-known result~\cite{Allard1972} 
whereas the latter is established in Lemma~\ref{lem:firstVarDeGiorgi}.

\textit{Step 2: Proof of ``$\geq$'' 
in~\eqref{def:MSmetricSlope2} assuming~\eqref{eq:repTangentialFirstVariation}.}
To show equality under the additional assumption of~\eqref{eq:repTangentialFirstVariation}, 
we may suppose that $|\partial E[\mu]|_{\mathcal{V}_\chi}<\infty$. 
First, we note that $B \cdot \nabla\chi \mapsto \delta \mu (B)$ is a well defined operator 
on $\{B \cdot \nabla \chi:B \in \mathcal{S}_{\chi}\}\subset \mathcal{V}_\chi $ (see definition (\ref{def:surfaceVelocities})). To see this, let $B\in \mathcal{S}_\chi$ be any function such that
 $B\cdot \nabla \chi = 0$ in $\mathcal{V}_\chi \subset H^{-1}_{(0)}$.
Recall that $\int_\Omega \chi \nabla\cdot B\, dx = 0$ by the definition of $\mathcal{S}_{\chi}$ in (\ref{def:variationVelocities}) to 
find that for all $\phi \in C^1(\overline{\Omega})$, one has 
$$\int_{\Omega} \phi B \cdot \frac{\nabla \chi}{|\nabla \chi|}\, d |\nabla \chi|
= \int_{\Omega} \Big(\phi - \dashint_{\Omega} \phi \,dx\Big) 
B \cdot \frac{\nabla \chi}{|\nabla \chi|}\, d |\nabla \chi| = 0.$$ 
The above equation implies that $B \cdot \frac{\nabla \chi}{|\nabla \chi|} = 0$ 
for $|\nabla \chi|-$almost every $x$ in $\Omega$, which by the representation of the first variation in terms of the curvature in (\ref{eq:repTangentialFirstVariation}) 
shows $\delta \mu(B) = 0.$ Linearity shows the operator is well-defined 
on $\{B \cdot \nabla \chi:B \in \mathcal{S}_{\chi}\}$.

With this in hand, the bound $|\partial E[\mu]|_{\mathcal{V}_\chi}<\infty$ 
implies that the mapping $B\cdot\nabla\chi\mapsto\delta\mu(B)$ can be extended to a bounded 
linear operator $L\colon\mathcal{V}_\chi \to \mathbb{R},$ which is identified with an 
element $L \in \mathcal{V}_\chi$ by the Riesz isomorphism theorem. 
Consequently, 
\begin{align}
\frac{1}{2}|\partial E[\mu]|_{\mathcal{V}_\chi}^2 
= \sup_{B \in \mathcal{S}_\chi} \left\{(L,B \cdot \nabla \chi)_{\mathcal{V}_{\chi}} 
- \frac{1}{2}\| B \cdot \nabla \chi\|_{\mathcal{V}_\chi}^2 \right\}
= \frac{1}{2}\|L\|_{\mathcal{V}_\chi}^2.
\end{align}
But recalling~\eqref{eq:firstVarAux1} and~\eqref{eq:firstVarAux2}, one also has that
\begin{align*}
A &= \sup_{\substack{\partial_s \chi_s|_{s=0} 
= -B \cdot \nabla \chi \\ \chi_s \to \chi, \ B\in \mathcal{S}_{\chi}}}  
\frac{\big(\lim_{s \to 0}\frac{1}{s}(E[\mu_s] - E[\mu])\big)_+}
{\lim_{s \to 0}\frac{1}{s}\|\chi_s - \chi\|_{H^{-1}_{(0)}}}
= \sup_{\ B\in \mathcal{S}_{\chi}}  
\frac{(\delta \mu (B))_{+}}{\| B \cdot \nabla \chi\|_{\mathcal{V}_{\chi}}} 
= \|L\|_{\mathcal{V}_{\chi}}.
\end{align*}
The previous two displays complete the proof that $A = |\partial E[\mu]|_{\mathcal{V}_\chi}.$
\end{proof}

\begin{proof}[Proof of Lemma~\ref{lem:PDEinterpretation}]
We proceed in three steps.

\textit{Proof of item~i):}
We first observe that by~\eqref{eq:DeGiorgiInequality} 
\begin{align*}
\int_{0}^{T_*} \frac{1}{2} \big\|\partial_t\chi\big\|_{H^{-1}_{(0)}(\Omega)}^2 \,dt = \int_{0}^{T_*} 
\frac{1}{2} \big\|\partial_t\chi\big\|_{\mathcal{T}_{\chi(\cdot,t)}}^2 \, dt
\leq E[\mu_0] < \infty,
\end{align*}
which in turn simply means that there exists
$u \in L^2(0,T_{*};\mathcal{H}_{\chi(\cdot,t)}) \subset L^2(0,T_{*};H^1_{(0)})$ such that
\begin{align}
\label{eq:propertiesPotentialMetricDerivative}
\partial_t\chi(\cdot,t) = \Delta_N u(\cdot,t) 
\quad\text{for a.e.\ } t \in (0,T_*).
\end{align}
In other words, recalling~\eqref{eqn:invLapRelation},
\begin{align}
\label{eq:evolEquationPhase2}
\int_{0}^{T_*} \int_{\Omega} \chi \p_t\zeta \,dxdt
= \int_{0}^{T_*} \int_{\Omega} \nabla u \cdot \nabla\zeta \,dxdt
\quad\text{for all } \zeta \in C^1_{cpt}(\overline{\Omega} {\times} (0,T_{*})).
\end{align}
By standard PDE~arguments and the requirement 
$\chi \in C([0,T_*);\smash{H^{-1}_{(0)}(\Omega)})$
such that ${\rm Tr}|_{t=0}\chi = \chi_0$ in $H^{-1}_{(0)}$,
one may post-process the previous display to~\eqref{eq:evolEquationPhase}.					
				
\textit{Proof of item~ii):} Thanks to~\eqref{eq:repTangentialFirstVariation},
we may apply the argument from the proof of Lemma~\ref{lem:metricSlope}
to infer that the map $B\cdot\nabla\chi(\cdot,t) \mapsto \delta\mu_t(B)$, $B \in \mathcal{S}_{\chi(\cdot,t)}$,
is well-defined and extends to a unique bounded and linear functional
$L_t\colon\mathcal{V}_{\chi(\cdot,t)} \to \mathbb{R}$. Recalling the definition
$\mathcal{G}_{\chi(\cdot,t)} = \Delta_N^{-1}(\mathcal{V}_{\chi(\cdot,t)}) \subset H^1_{(0)}$
and the fact that the weak Neumann Laplacian $\Delta_N\colon H^1_{(0)} \to H^{-1}_{(0)}$
is nothing else but the Riesz isomorphism between~$H^1_{(0)}$ and its dual~$H^{-1}_{(0)}$,
it follows from~\eqref{def:MSmetricSlope1} and~\eqref{eq:DeGiorgiInequality} that
there exists a potential $w_0 \in L^2(0,T_*;\mathcal{G}_{\chi(\cdot,t)})$	
such that
\begin{align*}
\delta\mu_t(B) = L_t\big(B\cdot\nabla\chi(\cdot,t)\big) 
= - \langle B\cdot\nabla\chi(\cdot,t), w_0(\cdot,t)\rangle_{H^{-1}_{(0)},H^1_{(0)}}
\end{align*}
for almost every $t \in (0,T_*)$ and all $B \in \mathcal{S}_{\chi(\cdot,t)}$,
as well as
\begin{align}
\label{eq:proofAux1}
\frac{1}{2}\|w_0(\cdot,t)\|^2_{\mathcal{G}_{\chi(\cdot,t)}} 
= \frac{1}{2}\big|\partial E [\mu_t]\big|^2_{\mathcal{V}_{\chi(\cdot,t)}}
\end{align}
for almost every $t \in (0,T_*)$. In particular, 
\begin{align}
\delta\mu_t(B) = \int_{\Omega} \chi(\cdot,t) \nabla\cdot\big(
B w_0(\cdot,t)\big) \,dx
\end{align}	
for almost every $t \in (0,T_*)$ and all $B \in \mathcal{S}_{\chi(\cdot,t)}$.
Due to Lemma~\ref{lem:AbelsRoeger}, there exists a measurable 
Lagrange multiplier $\lambda\colon (0,T_*)\to\mathbb{R}$
such that 
\begin{align}
\label{eq:proofAux2}
\delta\mu_t(B) = \int_{\Omega} \chi(\cdot,t) \nabla\cdot\big(
B \big(w_0(\cdot,t) {+} \lambda(t)\big)\big) \,dx
\end{align}
for almost every $t \in (0,T_*)$ and all $B \in C^1(\overline{\Omega};\Rd)$
with $(B \cdot n_{\p\Omega})|_{\p\Omega} \equiv 0$, and that there
exists a constant $C = C(\Omega,d,c_0,m_0) > 0$ such that
\begin{align}
\label{eq:proofAux3}
|\lambda(t)| \leq C \big(1 + |\nabla\chi(\cdot,t)|(\Omega)\big)
\big(|\mu_t|_{\mathbb{S}^{d-1}}(\overline{\Omega}) + \|\nabla w_0(\cdot,t)\|_{L^2(\Omega)}\big)
\end{align}
for almost every $t \in (0,T_*)$. Due to \eqref{eq:proofAux1}--\eqref{eq:proofAux3},
\eqref{eq:comp1} and~\eqref{eq:DeGiorgiInequality}, it follows
that the potential $w:=w_0+\lambda$ satisfies the desired
properties~\eqref{eq:GibbsThomson}--\eqref{eq:L2control}.

\textit{Proof of item~iii):} The De~Giorgi type energy dissipation inequality
in the form of~\eqref{eq:DeGiorgiInequality2} now
directly follows from~\eqref{eq:normEquivalenceGibbsThomsonPotential} 
and~\eqref{eq:propertiesPotentialMetricDerivative}.
\end{proof}

\begin{proof}[Proof of Remark~\ref{rem:existencePDEPerspective}]
A careful inspection of the proof of Theorem~\ref{theo:mainResult}
shows that the conclusions~\eqref{eqn:GTasSobolev} and~\eqref{eqn:Dissipation}
are indeed independent of an assumption on the value of the
ambient dimension $d \geq 2$. The same is true for the estimate~\eqref{eq:boundH1NormPotential} 
of Lemma~\ref{lem:AbelsRoeger}. The claim then immediately follows
from these observations and the first step of the proof of Lemma~\ref{lem:PDEinterpretation}.
\end{proof}

\begin{proof}[Proof of Lemma~\ref{lem:propertiesVarifoldSolutions}]
We start with a proof of the two consistency claims
and afterward give the proof of the compactness statement.

\textit{Step 1: Classical solutions are $BV$~solutions.}
Let~$\mathscr{A}$ be a classical solution for Mullins--Sekerka flow in the sense 
of~\emph{\eqref{eq:MullinsSekerka1}--\eqref{eq:MullinsSekerka6}}.
Define $\chi(x,t) := \chi_{\mathscr{A}(t)}(x)$ for all $(x,t) \in \Omega \times [0,T_*)$.
As one may simply define the associated 
varifold by means of $|\mu_t^\Omega|_{\mathbb{S}^{d-1}} := c_0|\nabla\chi(\cdot,t)|
\llcorner \Omega = c_0\mathcal{H}^{d-1}\llcorner(\partial^*\mathscr{A}(t) \cap \Omega)$,
$|\mu_t^{\partial\Omega}|_{\mathbb{S}^{d-1}} := c_0(\cos\alpha)
\chi(\cdot,t)\mathcal{H}^{d-1}\llcorner\Omega = c_0(\cos\alpha)
\mathcal{H}^{d-1}\llcorner(\partial^*\mathscr{A}(t) \cap \partial\Omega)$,
and \eqref{eq:BVsol2} with $\frac{\nabla\chi(\cdot,t)}{|\nabla\chi(\cdot,t)|}=n_{\partial\mathscr{A}(t)}$,
item~\textit{i)} of Definition~\ref{def:VarifoldSolution} is trivially satisfied.
Note that the varifold energy $|\mu_t|_{\mathbb{S}^{d-1}}(\overline{\Omega})$
simply equals the $BV$ energy functional $E[\chi(\cdot,t)]$ from~\eqref{def:energy}.

Due to the smoothness of the geometry and the validity of 
the contact angle condition~\eqref{eq:MullinsSekerka6} in the pointwise strong sense,
an application of the classical first variation formula together with
an integration by parts along each of the smooth manifolds
$\partial^*\mathscr{A}(t) \cap \Omega$
and $\partial^*\mathscr{A}(t) \cap \p\Omega$ ensures
(recall the notation from Subsection~\ref{subsec:strongFormulation})
\begin{align}
\nonumber
\delta E[\chi(\cdot,t)](B)
&= c_0 \int_{\partial^*\mathscr{A}(t) \cap \Omega} 
\nabla^\mathrm{tan}\cdot B \,d\mathcal{H}^{d-1}
+ c_0 (\cos\alpha) \int_{\partial^*\mathscr{A}(t) \cap \partial\Omega} 
\nabla^\mathrm{tan}\cdot B \,d\mathcal{H}^{d-1}
\\& \label{eq:classicalFirstVar}
= - c_0 \int_{\partial^*\mathscr{A}(t) \cap \Omega} 
H_{\partial\mathscr{A}(t)} \cdot B \,d\mathcal{H}^{d-1}
\end{align}
for all tangential variations $B \in C^1(\overline{\Omega};\Rd)$.
Hence, the identity \eqref{eq:repTangentialFirstVariation} holds with
$H_{\chi(\cdot,t)} = H_{\partial\mathscr{A}(t)} \cdot n_{\partial\mathscr{A}(t)}$
and the asserted integrability of~$H_{\chi(\cdot,t)}$ follows
from the boundary condition~\eqref{eq:MullinsSekerka3} and a standard 
trace estimate for the potential~$\bar u(\cdot,t)$. 

It remains to show~\eqref{eq:DeGiorgiInequality}. Starting point for this
is again the above first variation formula, now in the form of
\begin{align}
\label{eq:energyEvol1}
\frac{d}{dt} E[\chi(\cdot,t)] = - c_0 \int_{\partial^*\mathscr{A}(t) \cap \Omega}
H_{\partial\mathscr{A}(t)} \cdot V_{\partial\mathscr{A}(t)} \,d\mathcal{H}^{d-1}.
\end{align}
Plugging in~\eqref{eq:MullinsSekerka2} and~\eqref{eq:MullinsSekerka3},
integrating by parts in the form of~\eqref{eq:defJump},
and exploiting afterward~\eqref{eq:MullinsSekerka1} and~\eqref{eq:MullinsSekerka5},
we arrive at
\begin{align}
\label{eq:energyEvol2}
\frac{d}{dt} E[\chi(\cdot,t)] = - \int_{\Omega} |\nabla\bar{u}(\cdot,t)|^2 \,dx.
\end{align}
The desired inequality~\eqref{eq:DeGiorgiInequality} will follow once we prove
\begin{align}
\label{eq:identification1}
\|\partial_t\chi(\cdot,t)\|_{H^{-1}_{(0)}} &= \|\nabla\bar{u}(\cdot,t)\|_{L^2(\Omega)},
\\
\label{eq:identification2}
\big|\partial E [\mu_t]\big|_{\mathcal{V}_{\chi(\cdot,t)}}
&= \|\nabla\bar{u}(\cdot,t)\|_{L^2(\Omega)}.
\end{align}

Exploiting that the geometry underlying~$\chi$ is smoothly evolving,
i.e., $(\partial_t\chi)(\cdot,t) = -V_{\partial\mathscr{A}(t)}\cdot (\nabla\chi\llcorner\Omega)(\cdot,t)$,
the claim~\eqref{eq:identification1} is a consequence of the identity 
\begin{align*}
\int_{\Omega} \nabla\Delta^{-1}_N(\partial_t\chi)(\cdot,t) \cdot \nabla \phi \,dx
&= - \langle (\partial_t\chi)(\cdot,t), \phi \rangle_{H^{-1}_{(0)},H^1_{(0)}}
\\&
= - \int_{\partial^*\mathscr{A}(t) \cap \Omega} n_{\partial\mathscr{A}(t)} \cdot
\jump{\nabla \bar{u}(\cdot,t)} \phi \,d\mathcal{H}^{d-1}
\\&
= \int_{\Omega} \nabla\bar{u}(\cdot,t)\cdot\nabla\phi \,dx
\end{align*}
valid for all $\phi \in H^1_{(0)}$, where in the process we again made use
of~\eqref{eq:MullinsSekerka2}, \eqref{eq:MullinsSekerka1}, \eqref{eq:MullinsSekerka5},
and an integration by parts in the form of~\eqref{eq:defJump}.

For a proof of~\eqref{eq:identification2}, we first note that thanks to~\eqref{eq:classicalFirstVar}
and~\eqref{eq:MullinsSekerka3} it holds
\begin{align*}
\delta E[\chi(\cdot,t)](B) &= - \int_{\partial^*\mathscr{A}(t) \cap \Omega} 
\bar{u}(\cdot,t) n_{\partial\mathscr{A}(t)} \cdot B \,d\mathcal{H}^{d-1}
\\&
= \int_{\Omega} \nabla\bar{u}(\cdot,t)\cdot\nabla\Delta_N^{-1}(B\cdot\nabla\chi(\cdot,t)) \,dx
\end{align*}
for all $B \in \mathcal{S}_{\chi(\cdot,t)}$. Hence, in view of~\eqref{def:MSmetricSlope1}
it suffices to prove that for each fixed $t \in (0,T_*)$ there exists 
$\bar{B}(\cdot,t) \in \mathcal{S}_{\chi(\cdot,t)}$ such that
$\bar u(\cdot,t) = \Delta_N^{-1} (\bar{B}(\cdot,t)\cdot\nabla\chi(\cdot,t))$.

To construct such a $\bar B$, first note that from~\eqref{eq:MullinsSekerka1}, \eqref{eq:MullinsSekerka5} and~\eqref{eq:defJump}, we have $\Delta_N \bar{u}(\cdot,t) = (n_{\partial\mathscr{A}(t)} \cdot
\jump{\nabla \bar{u}(\cdot,t)}) \mathcal{H}^{d-1} \llcorner (\partial^*\mathscr{A}(t) \cap \Omega)$
in the sense of distributions.
Consequently, $\bar B(\cdot,t) \in C^1(\overline{ \Omega};\Rd)$ satisfying
\begin{equation}\label{eqn:Brequire}
\begin{aligned}
n_{\partial\mathscr{A}(t)} \cdot \bar{B}(\cdot,t) & =n_{\partial\mathscr{A}(t)} \cdot \jump{\nabla \bar{u}(\cdot,t)} \quad \text{ on } \partial^*\mathscr{A}(t) \cap \Omega,  \\
n_{\partial \Omega} \cdot \bar B(\cdot, t) &= 0 \quad ~~~~~~~~~~~~~~~~~~~~~\, \text{ on }  \partial \Omega,
\end{aligned}
\end{equation}
for which one has (using~\eqref{eq:MullinsSekerka1} and~\eqref{eq:massPreservation})
\begin{align*}
\int_{\Omega} \chi(\cdot,t) \nabla\cdot\bar{B}(\cdot,t) \,dx
=  - \int_{\partial\mathscr{A}(t) \cap \Omega} 
n_{\partial\mathscr{A}(t)} \cdot \jump{\nabla \bar{u}(\cdot,t)}  \,d\mathcal{H}^{d-1}
= 0,
\end{align*}
showing that $\bar B(\cdot,t) \in \mathcal{S}_{\chi(\cdot, t)},$ will satisfy the claim.

To this end, one first constructs a $C^1$ vector field $\mathcal{B}$ 
defined on $\overline{\partial\mathscr{A}(t) \cap \Omega}$ with the properties
that $n_{\partial\mathscr{A}(t)} \cdot \jump{\nabla \bar{u}(\cdot,t)}
= n_{\partial\mathscr{A}(t)} \cdot \mathcal{B}$ and $\mathcal{B}\cdot n_{\partial \Omega}=0$. Such $\mathcal{B}$ can be constructed by using a partition of unity on the manifold with boundary $\overline{\partial \mathscr{A}(t)\cap \Omega}$. Away from the boundary, set $\mathcal{B} := (n_{\partial\mathscr{A}(t)} \cdot \jump{\nabla \bar{u}(\cdot,t)})n_{\partial\mathscr{A}(t)}$, and near the boundary one can define $\mathcal{B}$ as an appropriate lifting of $$ \frac{(n_{\partial\mathscr{A}(t)} \cdot \jump{\nabla \bar{u}(\cdot,t)})}{\cos(\pi/2 - \alpha)}\tau_{\partial \mathscr{A}(t) \cap \partial \Omega},$$ where $\tau_{ \partial \mathscr{A}(t)\cap \partial \Omega}$ is the unit length vector field tangent to $\partial \Omega$ normal to the contact points manifold $\overline{\partial \mathscr{A}(t)\cap \Omega} \cap \partial \Omega $ and with $n_{\partial \mathscr{A}(t)}\cdot \tau_{\partial \mathscr{A}(t) \cap \partial \Omega } = \cos(\pi/2-\alpha).$ With the vector field~$\mathcal{B}$ in place,
any tangential $\bar{B}(\cdot,t ) \in C^1(\overline{\Omega};\Rd)$ extending $\mathcal{B}$ satisfies (\ref{eqn:Brequire}),
which in turn concludes the proof of the first step.

\textit{Step 2: Smooth varifold solutions are classical solutions.}
Let $(\chi,\mu)$ be a varifold solution with smooth geometry, i.e., satisfying Definition~\ref{def:VarifoldSolution}
and such that the indicator~$\chi$ can be represented as $\chi(x,t) = \chi_{\mathscr{A}(t)}(x)$,
where $\mathscr{A}=(\mathscr{A}(t))_{t \in [0,T_{*})}$
is a time-dependent family of smoothly evolving subsets $\mathscr{A}(t) \subset \Omega$,
$t \in [0,T_*)$, as in the previous step. It is convenient for what follows
to work with the two potentials~$u$ and~$w$ satisfying the conclusions
of Lemma~\ref{lem:PDEinterpretation}, i.e., 
\eqref{eq:evolEquationPhase}--\eqref{eqn:curvature=trace}.

Fix $t \in (0,T_*)$ such that~\eqref{eqn:curvature=trace} holds.
By regularity of $\supp|\nabla\chi(\cdot,t)| \llcorner \Omega
= \partial\mathscr{A}(t) \cap \Omega$ and the fact that
$H_{\chi(\cdot,t)}$ is the generalized mean curvature vector 
of~$\supp|\nabla\chi(\cdot,t)| \llcorner \Omega$
in the sense of R\"{o}ger~\cite[Definition~1.1]{Roeger2004},
it follows from R\"{o}ger's result~\cite[Proposition~3.1]{Roeger2004} that
\begin{align}
\label{eq:trueCurvature=RoegerCurvature}
H_{\partial\mathscr{A}(t)} &= H_{\chi(\cdot,t)}
&& \mathcal{H}^{d-1}\text{-a.e.\ on }
\partial\mathscr{A}(t) \cap \Omega.
\end{align}
Hence, we deduce from~\eqref{eqn:curvature=trace} that
\begin{align}
\label{eq:MullinsSekerka3Aux}
w(\cdot,t)n_{\partial\mathscr{A}(t)} &= c_0 H_{\partial\mathscr{A}(t)}
&& \mathcal{H}^{d-1}\text{-a.e.\ on } \partial\mathscr{A}(t) \cap \Omega.
\end{align}
Recalling $w\in \mathcal{G}_\chi \subset \mathcal{H}_\chi$ and (\ref{eqn:HchiHarmonicRelation}) 
shows that $w$ is harmonic in $\mathscr{A}$ and in the interior of $\Omega\setminus \mathscr{A}$. 
One may then apply standard elliptic regularity theory for the Dirichlet problem \cite{EvansPDE} to 
obtain a continuous representative for $w$ and further conclude that (\ref{eq:MullinsSekerka3Aux}) 
holds everywhere on $\partial\mathscr{A}(t) \cap \Omega$.

Next, we take care of the contact angle condition~\eqref{eq:MullinsSekerka6}.
To this end, we denote by $\tau_{\p\mathscr{A}(t) \cap \Omega}$
a vector field on the contact points manifold, $\p(\overline{\partial \mathscr{A}(t)\cap \Omega}) \subset \partial \Omega $, that is tangent 
to the interface $\p\mathscr{A}(t) \cap \Omega$, normal to the contact points
manifold, and which points away from $\p\mathscr{A}(t) \cap \Omega$. We further denote by
$\tau_{\p\mathscr{A}(t) \cap \partial\Omega}$ a vector field along the contact points manifold 
which now is tangent to~$\partial\Omega$, again normal to the contact points
manifold, and which this time points towards $\p\mathscr{A}(t) \cap \partial\Omega$.
Note that by these choices,
at each point of the contact points manifold the vector fields $\tau_{\p\mathscr{A}(t) \cap \Omega}$, 
$n_{\partial\mathscr{A}(t)}$, $\tau_{\p\mathscr{A}(t) \cap \partial\Omega}$ and $n_{\p\Omega}$ 
lie in the normal space of the contact points manifold, 
and that the orientations were precisely chosen such that
$\tau_{\p\mathscr{A}(t) \cap \Omega} \cdot \tau_{\p\mathscr{A}(t) \cap \partial\Omega}
= n_{\partial\mathscr{A}(t)} \cdot n_{\p\Omega}$. With these constructions in place,
we obtain from the classical first variation formula and an integration by parts
along  $\partial\mathscr{A}(t) \cap \Omega$ 
and $\partial\mathscr{A}(t) \cap \p\Omega$ that
\begin{align}
\nonumber
&\delta E[\chi(\cdot,t)](B)
\\& \nonumber
= c_0 \int_{\partial\mathscr{A}(t) \cap \Omega} 
\nabla^\mathrm{tan}\cdot B \,d\mathcal{H}^{d-1}
+ c_0 (\cos\alpha) \int_{\partial\mathscr{A}(t) \cap \partial\Omega} 
\nabla^\mathrm{tan}\cdot B \,d\mathcal{H}^{d-1}
\\& \nonumber
= - c_0 \int_{\partial\mathscr{A}(t) \cap \Omega} 
H_{\partial\mathscr{A}(t)} \cdot B \,d\mathcal{H}^{d-1}
\\&~~~ \nonumber
+ c_0 \int_{\p(\partial\mathscr{A}(t) \cap \Omega)} 
 (\tau_{\p\mathscr{A}(t) \cap \Omega} - 
(\cos\alpha)\tau_{\p\mathscr{A}(t) \cap \p\Omega}) \cdot B \,d\mathcal{H}^{d-2}
\\& \label{eq:classicalFirstVar2}
= - c_0 \int_{\partial\mathscr{A}(t) \cap \Omega} 
H_{\partial\mathscr{A}(t)} \cdot B \,d\mathcal{H}^{d-1}
\\&~~~ \nonumber
+ c_0 \int_{\p(\partial\mathscr{A}(t) \cap \Omega)} 
 (\tau_{\p\mathscr{A}(t) \cap \Omega}  \cdot \tau_{\p\mathscr{A}(t) \cap \p\Omega} - 
 \cos\alpha) (\tau_{\p\mathscr{A}(t) \cap \p\Omega} \cdot B) \,d\mathcal{H}^{d-2}
\end{align}
for all tangential variations $B \in C^1(\overline{\Omega})$.
Recall now that we assume that~\eqref{eq:repTangentialFirstVariation} even holds
with $\delta\mu_t$ replaced on the left hand side by $\delta E[\chi(\cdot,t)]$. 
In particular, $\delta\mu_t(B) = \delta E[\chi(\cdot,t)](B)$ for all
$B \in \mathcal{S}_{\chi(\cdot,t)}$ so that the argument from
the proof of Lemma~\ref{lem:PDEinterpretation}, item~\textit{$ii)$}, shows
\begin{align*}
\delta E[\chi(\cdot,t)](B) = 
- \int_{\partial^*\mathscr{A}(t) \cap \Omega} w(\cdot,t) 
n_{\partial\mathscr{A}(t)} \cdot B \,d\mathcal{H}^{d-1}
\end{align*}
for all tangential $B \in C^1(\overline\Omega;\Rd)$.
Because of~\eqref{eq:MullinsSekerka3Aux}, we thus infer from~\eqref{eq:classicalFirstVar2} that
the contact angle condition~\eqref{eq:MullinsSekerka6} indeed holds true.

Note that for each $t$ in $(0,T^*)$, the potential $u$ satisfies
\begin{align*}
\partial_t\chi(\cdot,t) = -V_{\partial\mathscr{A}(t)} 
\cdot (\nabla\chi\llcorner\Omega)(\cdot,t) = \Delta_N u.
\end{align*}
In particular, by the assumed regularity of $\supp|\nabla\chi(\cdot,t)| \llcorner \Omega
= \partial\mathscr{A}(t) \cap \Omega$, 
\begin{align}
\label{eq:MullinsSekerka1Aux}
\Delta u(\cdot,t) &= 0
&& \text{in } \Omega \setminus \p\mathscr{A}(t).
\end{align}
Furthermore, we claim that
\begin{equation}
\begin{aligned}
\label{eq:MullinsSekerka2Aux}
V_{\partial\mathscr{A}(t)} &= - \big(n_{\partial\mathscr{A}(t)}\cdot
\jump{\nabla u(\cdot,t)}\big)n_{\partial\mathscr{A}(t)}
&& \text{on } \p\mathscr{A}(t) \cap \Omega,
\\ 
(n_{\p\Omega} \cdot \nabla) u(\cdot,t) &= 0
&&\text{on } \p\Omega \setminus \overline{\partial \mathscr{A}(t)\cap \Omega}.
\end{aligned}
\end{equation}
To prove~\eqref{eq:MullinsSekerka2Aux}, 
we suppress for notational convenience the time variable and show
for any open set $\mathcal{O}\subset \Omega$ which does 
not contain the contact points manifold 
$\partial(\overline{\partial \mathscr{A}(t)\cap \Omega}) \subset \partial\Omega$ that
\begin{equation}\label{eqn:uAuxRegularity}
u \in H^3(\mathcal{O}\cap \mathscr{A}(t)) \cap 
H^3(\mathcal{O}\cap (\Omega\setminus\overline{\mathscr{A}(t)})).
\end{equation} 
With this, $\nabla u$ will have continuous representatives in 
$\overline{\mathscr{A}(t)}$ and $\overline{\Omega\setminus \mathscr{A}(t)}$ excluding contact points, 
from which (\ref{eq:MullinsSekerka2Aux}) will follow by applying the integration by parts 
formula~\eqref{eq:defJump}. Note that typical estimates apply for the Neumann problem 
if $\mathcal{O}$ does not intersect $\partial \mathscr{A}(t)$, and consequently to 
conclude (\ref{eqn:uAuxRegularity}), it suffices to prove regularity in the case of a 
flattened and translated interface $\partial \mathscr{A}(t)$ with $u$ truncated, that is, 
for $u$ satisfying
\begin{equation}
\label{eq:auxNeumannProblemPotential}
\begin{aligned}
\Delta u  &= -V \mathcal{H}^{d-1}\llcorner \{x_d = 0\} &&\text{ in } B(0,1), \\
u  &=0 &&\text{ on }\partial B(0,1),
\end{aligned}
\end{equation}
with $V$ smooth.
The above equation can be differentiated for all 
multi-indices $\beta \subset \mathbb{N}^{d-1}$
representing tangential directions, showing 
that $ \partial^\beta u \in H^1(B(0,1)) $. Rearranging (\ref{eq:MullinsSekerka1Aux})
to extract $\partial_d^2 u$ from the Laplacian, 
we have that $ u$ belongs to $H^2(B(0,1)\setminus\{x_d=0\}).$ To control the higher derivatives, 
note by the comment regarding multi-indices, we already have 
$\partial_i\partial_j\partial_d u \in L^2(\Omega)$ for all $i,j \neq d$. 
Furthermore, differentiating~\eqref{eq:auxNeumannProblemPotential}
with respect to the $i$-th direction, where $i \in \{1,\ldots,d{-}1\}$, 
and repeating the previous argument shows $\partial_d^2\partial_i u \in L^2(B(0,1)\setminus \{x_d = 0\})$.
Finally, differentiating (\ref{eq:MullinsSekerka1Aux}) with respect to the $d$-th direction
away from both~$\partial\mathscr{A}$ and~$\partial\Omega$ and then extracting
$\partial_d^3 u$ from $\Delta \partial_d u$, we get $\partial_d^3 u \in L^2(B(0,1)\setminus \{x_d = 0\})$,
finishing the proof of (\ref{eqn:uAuxRegularity}).


Finally, we show
\begin{align}
\label{eq:samePotentials}
\|(\nabla u - \nabla w)(\cdot,t)\|_{L^2(\Omega)} = 0.
\end{align}
Note that~\eqref{eq:samePotentials} is indeed sufficient to conclude
that the smooth $BV$~solution $\chi$ is a classical solution
because we already established~\eqref{eq:MullinsSekerka3Aux},
\eqref{eq:MullinsSekerka2Aux}, \eqref{eq:MullinsSekerka1Aux}
and~\eqref{eq:MullinsSekerka6}. For a proof of~\eqref{eq:samePotentials},
we simply exploit smoothness of the evolution
in combination with~\eqref{eq:energyEvol1},
\eqref{eq:MullinsSekerka3Aux}, \eqref{eq:MullinsSekerka1Aux},
\eqref{eq:MullinsSekerka2Aux}, \eqref{eq:MullinsSekerka6} and~\eqref{eq:defJump}
to obtain
\begin{align*}
\frac{d}{dt} E[\chi(\cdot,t)] = - \int_{\Omega} \nabla u(\cdot,t) \cdot \nabla w(\cdot,t) \,dx.
\end{align*}
Subtracting the previous identity (in integrated form) from~\eqref{eq:DeGiorgiInequality2}
and noting that $E[\chi(\cdot,T) \leq E[\mu_T]$ 
(due to the definitions~\eqref{def:energy} and~\eqref{def:varifoldEnergy} as well as 
the compatibility conditions~\eqref{eq:comp1} and~\eqref{eq:comp2}), we get
\begin{align}
\label{eq:identityEnergies}
0 \leq \int_{0}^{T} \int_{\Omega} \frac{1}{2}|\nabla u-\nabla w|^2 \,dxdt
\leq E[\chi(\cdot,T)] - E[\mu_T] \leq 0
\end{align}
for a.e.\ $T \in (0,T_*)$, which in turn proves~\eqref{eq:samePotentials}.

Note from~\eqref{eq:identityEnergies} that $E[\chi(\cdot,T)] = E[\mu_T]$
for a.e.\ $T \in (0,T_*)$. For general constants $a+b = c+d$, $a\leq c$, 
and $b\leq d$ implies $a=c$ and $b=d$, and we use this with coincidence 
of the varifold and BV energies to find that both (\ref{eq:comp1}) and (\ref{eq:comp2}) 
hold with equality. This implies that~\eqref{eq:identityMassMeasures1} holds. Assuming now that
$\smash{\frac{1}{c_0}}\mu_t^\Omega
\in \mathrm{M}(\overline{\Omega}{\times}\mathbb{S}^{d-1})$ 
is an integer rectifiable oriented varifold, we consider the Radon--Nikod\`ym 
derivative of both sides of (\ref{eq:comp2}) with respect 
to $\mathcal{H}^{d-1}\llcorner \partial \Omega$ to find
\begin{align*}
(\cos\alpha)c_0\chi(\cdot,t)
					= \frac{|\mu_t^{\p\Omega}|_{\mathbb{S}^{d-1}}}{\mathcal{H}^{d-1}\llcorner\partial \Omega}(\cdot,t)
					{+} c_0 m(\cdot,t) ,
\end{align*}
for some integer-valued function $m\colon\partial \Omega \to \mathbb{N}\cup \{0\}.$ 
Necessarily, $m \equiv 0$, concluding~\eqref{eq:identityMassMeasures2}.

\textit{Step 3: Compactness of solution space.}
It is again convenient to work explicitly with potentials. 
More precisely, for each $k \in \mathbb{N}$ we fix a potential~$w_k$
subject to item~\textit{ii)} of Lemma~\ref{lem:PDEinterpretation}
with respect to the varifold solution $(\chi_k,\mu_k)$. By virtue
of~\eqref{eq:DeGiorgiInequality} and~\eqref{eq:normEquivalenceGibbsThomsonPotential}, 
we have for all $k \in \mathbb{N}$				 
\begin{align*}
&E[(\mu_k)_T] {+} \frac{1}{2} \int_{0}^{T} 
\big\|(\partial_t\chi_k)(\cdot,t)\big\|_{H^{-1}_{(0)}}^2 
+ \|w_k(\cdot,t)\|^2_{L^2(\Omega)} \, dt   
\leq E[\mu_{k,0}].
\end{align*}

By assumption, we may select a subsequence $k\to\infty$
such that $\chi_{k,0} \wksto \chi_0$ in $BV(\Omega;\{0,1\})$
to some $\chi_0 \in BV(\Omega;\{0,1\})$. Since we also assumed
tightness of the sequence $(|\nabla\chi_{k,0}|\llcorner\Omega)_{k \in \mathbb{N}}$,
it follows that along the previous subsequence we also
have $|\nabla\chi_{k,0}|(\Omega) \to |\nabla\chi_0|(\Omega)$.
In other words, $\chi_{k,0}$ converges strictly in $BV(\Omega;\{0,1\})$ along 
$k\to\infty$ to $\chi_0$, which in turn implies convergence of
the associated traces in $L^1(\partial\Omega;d\mathcal{H}^{d-1})$.
In summary, we may deduce $E[\mu_{k,0}] = E[\chi_{k,0}]
\to E[\chi_0] = E[\mu_0]$ for the subsequence $k\to\infty$.

For the rest of the argument, a close inspection reveals
that one may simply follow the reasoning from \textit{Step~2}
to \textit{Step~11} of the proof of Theorem~\ref{theo:mainResult}
as these steps do not rely on the actual procedure generating
the sequence of (approximate) solutions but only on consequences
derived from the validity of the associated sharp energy dissipation inequalities. 
\end{proof}

\begin{proof}[Proof of Proposition~\ref{prop:RoegerInterpretation}]
We divide the proof into three steps.

\textit{Proof of item~i):} We start by recalling some notation
from the proof of Theorem~\ref{theo:mainResult}. For $h > 0$,
we denoted by~$\bar{\chi}^h$ the De~Giorgi interpolant~\eqref{eq:minMovVariationalInterpol}.
From the definition~\eqref{eq:approxVarifolds} of the approximate oriented 
varifold $\mu^{\Omega,h} \in \mathrm{M}((0,T_*){\times}\overline{\Omega}{\times}\mathbb{S}^{d-1})$ 
and the measure $|\mu^{\partial\Omega,h}|_{\mathbb{S}^{d-1}} \in \mathrm{M}((0,T_*){\times}\partial\Omega)$,
respectively, and an integration by parts, it then follows that
\begin{align*}
&\int_{(0,T_*) \times \overline{\Omega} \times \mathbb{S}^{d-1}} 
s \cdot \zeta(t)\eta(x) \,d\mu^{\Omega,h}(t,x,s)
\\&
= c_0\int_{0}^{T_*} \zeta(t) \int_{\Omega} 
\frac{\nabla\bar{\chi}^h(\cdot,t)}{|\nabla\bar{\chi}^h(\cdot,t)|}
\cdot \eta(\cdot) \,d|\nabla\bar{\chi}^h(\cdot,t)| dt
\\&
= - c_0\int_{0}^{T_*} \zeta(t) \int_{\Omega} \bar\chi^h(x,t) (\nabla\cdot\eta)(x) \,dx dt
\end{align*}
for all $\eta \in C^\infty(\overline{\Omega};\Rd)$
with $n_{\partial\Omega}\cdot \eta = 0$ on~$\partial\Omega$ and all $\zeta \in C^\infty_c(0,T_*)$, and also
\begin{align*}
&\int_{(0,T_*) \times \overline{\Omega} \times \mathbb{S}^{d-1}} 
s \cdot \zeta(t)\xi(x) \,d\mu^{\Omega,h}(t,x,s)
+ \int_{(0,T_*) \times \partial\Omega} \zeta(t)\,d|\mu^{\partial\Omega,h}|_{\mathbb{S}^{d-1}}(t,x)
\\&= c_0\int_{0}^{T_*} \zeta(t) \int_{\Omega} 
\frac{\nabla\bar{\chi}^h(\cdot,t)}{|\nabla\bar{\chi}^h(\cdot,t)|}
\cdot \xi(\cdot) \,d|\nabla\bar{\chi}^h(\cdot,t)| dt
\\&~~~
+ c_0 \int_{0}^{T_*} \zeta(t) \int_{\partial\Omega}
\bar\chi^h(\cdot,t)\cos\alpha \,d\mathcal{H}^{d-1} dt
\\&
= c_0\int_{0}^{T_*} \zeta(t) \int_{\Omega} 
\frac{\nabla\bar{\chi}^h(\cdot,t)}{|\nabla\bar{\chi}^h(\cdot,t)|}
\cdot \xi(\cdot) \,d|\nabla\bar{\chi}^h(\cdot,t)| dt
\\&~~~
+ c_0 \int_{0}^{T_*} \zeta(t) \int_{\partial\Omega} 
\bar\chi^h(\cdot,t)n_{\partial\Omega}\cdot\xi(\cdot) \,d\mathcal{H}^{d-1} dt
\\&
= - c_0\int_{0}^{T_*} \zeta(t) \int_{\Omega} \bar\chi^h(x,t) (\nabla\cdot\xi)(x) \,dx dt
\end{align*}
for all $\xi \in C^\infty(\overline{\Omega};\Rd))=$
with $n_{\partial\Omega}\cdot\xi= \cos\alpha$ on~$\partial\Omega$ and all $\zeta \in C^\infty_c(0,T_*)$.

It therefore follows from the convergences $\bar\chi^h \to \chi$
in $L^2(0,T_*;L^2(\Omega))$, $\mu^{\Omega,h} \wksto \mathcal{L}^1 \llcorner (0,T_*)
\otimes (\mu^{\Omega}_t)_{t \in (0,T_*)}$ in $\mathrm{M}((0,T_*){\times}\overline{\Omega}
{\times}\mathbb{S}^{d-1})$ and $|\mu^{\partial\Omega,h}|_{\mathbb{S}^{d-1}}
\wksto \mathcal{L}^1 \llcorner (0,T_*) \otimes 
(|\mu^{\partial\Omega}_t|_{\mathbb{S}^{d-1}})_{t \in (0,T_*)}$
in $\mathrm{M}((0,T_*){\times}\partial\Omega)$, respectively,
that first taking the limit in the previous two displays for a
suitable subsequence $h \downarrow 0$ and then undoing
the integration by parts in the respective right-hand sides gives
\begin{align*}
&\int_{0}^{T_*} \zeta(t) \int_{\overline{\Omega} \times \mathbb{S}^{d-1}} 
s \cdot \eta(x) \,d\mu^{\Omega}_t(x,s) dt
\\&
= c_0\int_{0}^{T_*} \zeta(t) \int_{\Omega} 
\frac{\nabla\chi(\cdot,t)}{|\nabla\chi(\cdot,t)|}
\cdot \eta(\cdot) \,d|\nabla\chi(\cdot,t)| dt
\end{align*}
for all $\eta \in C^\infty(\overline{\Omega};\Rd)$
with $n_{\partial\Omega}\cdot\eta = 0$ on~$\partial\Omega$ and all $\zeta \in C^\infty_c(0,T_*)$, and 
\begin{align*}
&\int_{0}^{T_*} \zeta(t) \int_{\overline{\Omega} \times \mathbb{S}^{d-1}} 
s \cdot \xi(x) \,d\mu_t^{\Omega}(x,s) dt
+ \int_{0}^{T_*} \zeta(t) \int_{\partial\Omega} 
1\,d|\mu^{\partial\Omega}_t|_{\mathbb{S}^{d-1}} dt
\\&= c_0\int_{0}^{T_*} \zeta(t) \int_{\Omega} 
\frac{\nabla\chi(\cdot,t)}{|\nabla\chi(\cdot,t)|}
\cdot \xi(\cdot) \,d|\nabla\chi(\cdot,t)| dt
\\&~~~
+ c_0 \int_{0}^{T_*} \zeta(t) \int_{\partial\Omega}
\chi(\cdot,t)\cos\alpha \,d\mathcal{H}^{d-1} dt
\end{align*}
for all $\xi \in C^\infty(\overline{\Omega};\Rd)$
with $n_{\partial\Omega}\cdot \xi= \cos\alpha$ on~$\partial\Omega$ and all $\zeta \in C^\infty_c(0,T_*)$.

The two compatibility conditions~\eqref{eq:comp3} and~\eqref{eq:comp4}
now immediately follow from the previous two equalities and a localization
argument in the time variable.

\textit{Proof of item~ii):} Let $w \in L^2(0,T_*;H^1(\Omega))$
be the potential from item~\textit{ii)} of Lemma~\ref{lem:PDEinterpretation}.
Thanks to Step~7 of the proof of Theorem~\ref{theo:mainResult},
the relation~\eqref{eq:repTangentialFirstVariation} indeed not
only holds for $B \in \mathcal{S}_{\chi(\cdot,t)}$ but also for all 
$B \in C^1(\overline{\Omega};\Rd)$ with $B \cdot n_{\p\Omega}=0$ along~$\p\Omega$.
Hence, due to~\eqref{eq:GibbsThomson} and the trace estimate~\eqref{eq:firstVariationInterior} 
from Proposition~\ref{prop:firstVariationTangential}
for the potential~$w$, it follows
\begin{equation}\label{eqn:curvature=trace2}
c_0 H_\chi(\cdot,t) = w(\cdot,t)
\end{equation} for almost every $t \in (0,T_*)$
up to sets of ($|\nabla\chi(\cdot,t)| \llcorner \Omega$)-measure zero.
The asserted estimate~\eqref{eq:quantitativeIntCurvature} follows in turn
from the trace estimate~\eqref{eq:firstVariationInterior},
the properties~\eqref{eq:normEquivalenceGibbsThomsonPotential} and~\eqref{eq:L2control}, the compatibility condition (\ref{eq:comp1}),
and the energy estimate~\eqref{eq:DeGiorgiInequality}.

\textit{Proof of item~iii):} Post-processing the first variation
estimate~\eqref{eq:boundedFirstVariationGlobally} from 
Corollary~\ref{cor:boundedFirstVariationGlobally} by means
of the trace estimate~\eqref{eq:firstVariationInterior} for $s=2$
and the energy estimate~\eqref{eq:DeGiorgiInequality}
yields~\eqref{eq:quantitativeGlobalFirstVarEstimate}.
\end{proof}

\begin{proof}[Proof of Proposition~\ref{prop:firstVariationTangential}]
We split the proof into three steps. 
In the first and second steps, we develop estimates for an approximation of the $(d-1)$-density of the varifold using ideas introduced by Gr\"uter and Jost~\cite[Proof of Theorem~3.1]{Grueter1986} (see also Kagaya and Tonegawa~\cite[Proof of Theorem 3.2]{Kagaya2017}), that were originally used to derive monotonicity formula for varifolds with integrable curvature near a domain boundary. In the third step, we combine this approach with Sch\"atzle's \cite{Schaetzle2001} work, which derived a monotonicity formula in the interior, to obtain a montonicity formula up to the boundary.

\textit{Step 1: Preliminaries.} Since $\p\Omega$ is compact and of class~$C^2$,
we may choose a localization scale~$r=r(\p\Omega) \in (0,1)$ such that~$\p\Omega$
admits a regular tubular neighborhood of width~$2r$. More precisely, the map
\begin{align*}
\Psi_{\p\Omega} \colon \p\Omega \times (-2r,2r) \to 
\{y \in \Rd \colon \dist(y,\p\Omega) < 2r\},
\quad (x,s) \mapsto x + s\n_{\p\Omega}(x)
\end{align*}
is a $C^1$-diffeomorphism such that $\|\nabla\Psi_{\p\Omega}\|_{L^\infty},
\|\nabla\Psi^{-1}_{\p\Omega}\|_{L^\infty} \leq C$. The inverse
splits in form of $\Psi^{-1}_{\p\Omega}=(P_{\p\Omega},s_{\p\Omega})$,
where $s_{\p\Omega}$ represents the signed distance to~$\p\Omega$
oriented with respect to~$\n_{\p\Omega}$ and~$P_{\p\Omega}$
represents the nearest point projection onto~$\p\Omega$:
$P_{\p\Omega}(x) = x - s_{\p\Omega}(x)\n_{\p\Omega}(P_{\p\Omega}(x))
= x - s_{\p\Omega}(x)(\nabla s_{\p\Omega})(x)$ for all~$x\in\Rd$
such that $\dist(x,\p\Omega)<2r$.

In order to extend the argument of Sch\"atzle~\cite[Proof of Lemma~2.1]{Schaetzle2001} up to the boundary of $\partial \Omega$, we employ the reflection technique of
Gr\"uter and Jost~\cite[Proof of Theorem~3.1]{Grueter1986}. To this end,
we introduce further notation. First, we denote by
\begin{align}
\label{def:reflectedPoint}
\tilde x := 2 P_{\p\Omega}(x) - x,
\quad x \in \Rd \text{ such that } \dist(x,\p\Omega) < 2r
\end{align}
the reflection of the point~$x$ across~$\p\Omega$ in normal direction.
Further, we define the ``reflected ball"
\begin{align}
\label{def:reflectedBall}
\tilde B_\rho(x_0) &:= \{x \in \Rd \colon \dist(x,\p\Omega) < 2r,\, |\tilde x - x_0| < \rho\},
\\&~~~~~~ \nonumber 
\rho \in (0,r), \,x_0 \in \overline{\Omega} \colon \dist(x_0,\p\Omega) < r.
\end{align}
Finally, we set
\begin{align}
\label{def:tangentialCorrection}
\iota_x(y) &:= \big(\mathrm{Id} {-} \n_{\p\Omega}(P_{\p\Omega}(x)) 
\otimes \n_{\p\Omega}(P_{\p\Omega}(x))\big)y
\\&~~~~~~~~~~~~~\nonumber
- \big(y\cdot\n_{\p\Omega}(P_{\p\Omega}(x))\big)\n_{\p\Omega}(P_{\p\Omega}(x)),
\\&~~~~~~ \nonumber
y \in \Rd, \, x \in \Rd \colon \dist(x,\p\Omega) < 2r,
\end{align}
which reflects a vector across the tangent space at $P_{\p\Omega}(x)$ on $\partial \Omega.$

\textit{Step 2: A preliminary monotonicity formula.}
Let $\rho \in (0,r)$ and $x_0 \in \overline{\Omega}$ such that $\dist(x_0,\p\Omega) < r$.
Let $\eta\colon [0,\infty) \to [0,1]$ be smooth and nonincreasing such that
$\eta \equiv 1$ on $[0,\frac{1}{2}]$, $\eta \equiv 0$ on $[1,\infty)$,
and $|\eta'| \leq 4$ on $[0,\infty)$.
Consider then the functional
\begin{align*}
I_{x_0}(\rho) := \int_{\overline{\Omega}} \eta\left(\frac{|x{-}x_0|}{\rho}\right) 
+ \tilde\eta_{x_0,\rho}(x) \,d|\mu|_{\mathbb{S}^{d-1}} \geq 0,
\end{align*}
where $\tilde\eta_{x_0,\rho}\colon\overline{\Omega} \to [0,1]$
represents the $C^1$-function
\begin{align}
\label{eq:noAbuseOfNotation}
\tilde\eta_{x_0,\rho}(x) := 
\begin{cases}
\eta\big(\frac{|\tilde x {-}x_0|}{\rho}\big) &\text{if } \dist(x,\p\Omega) < 2r,
\\
0 & \text{else.}
\end{cases}
\end{align}

A close inspection of the argument given by Kagaya and 
Tonegawa~\cite[Estimate~(4.11) and top of page~151]{Kagaya2017}
reveals that we have the bound
\begin{equation}
\label{eq:monotonicity1}
\begin{aligned}
\frac{d}{d\rho} \big(\rho^{-(d-1)} I_{x_0}(\rho)\big)
&\geq -C\big(\rho^{1-(d-1)}I'_{x_0}(\rho) + \rho^{-(d-1)}I_{x_0}(\rho)\big)
\\&~~~
- \rho^{-d}\int_{\overline{\Omega}} (\mathrm{Id} - s \otimes s) : \nabla \Phi_{x_0,\rho} \,d\mu,
\end{aligned}
\end{equation}
where the test vector field~$\Phi_{x_0,\rho}$ is given by
(recall the definition~\eqref{eq:noAbuseOfNotation} of~$\tilde\eta_{x_0,\rho}$)
\begin{align}
\label{eq:testVectorMonotonicity}
\Phi_{x_0,\rho}(x) := \eta\left(\frac{|x{-}x_0|}{\rho}\right) (x-x_0) 
+ \tilde\eta_{x_0,\rho}(x) \iota_x(\tilde x - x_0).
\end{align}
For any~$q \in [d,\infty)$, we may further post-process~\eqref{eq:monotonicity1} 
by an application of the chain rule to the effect of
\begin{align}
\nonumber
&\frac{d}{d\rho} \big(\rho^{-(d-1)} I_{x_0}(\rho)\big)^\frac{1}{q}
\\ \label{eq:monotonicity2}
&\geq -\frac{C}{q}\big(\rho^{-(d-1)} I_{x_0}(\rho)\big)^{\frac{1}{q}-1}
\big(\rho^{1-(d-1)}I'_{x_0}(\rho) + \rho^{-(d-1)}I_{x_0}(\rho)\big)
\\&~~~ \nonumber
- \frac{1}{q}\big(\rho^{-(d-1)} I_{x_0}(\rho)\big)^{\frac{1}{q}-1} \rho^{-d}
\int_{\overline{\Omega}} (\mathrm{Id} - s \otimes s) : \nabla \Phi_{x_0,\rho} \,d\mu.
\end{align}
Since we do not yet know that the generalized mean curvature vector 
field of the interface~$\supp|\nabla\chi|\llcorner\Omega$
is $q$-integrable, we can not simply proceed as in~\cite[Proof of Theorem~3.1]{Grueter1986}
or~\cite[Proof of Theorem 3.2]{Kagaya2017} (at least with respect to the second right hand side
term of the previous display).

To circumvent this technicality, define $$f(\rho) : =\max\{(\rho^{-(d-1)} I_{x_0}(\rho))^{1/q},1\}.$$ Noting that by choice of $q\geq d$ the first right-hand side term of (\ref{eq:monotonicity2}) is bounded from below by the derivative of the product $\rho^{1-\frac{d-1}{q}}(I_{x_0}(\rho))^{1/q}$, we integrate (\ref{eq:monotonicity2}) over an interval $(\sigma,\tau)$ where $(\rho^{-(d-1)} I_{x_0}(\rho))^{1/q} \geq 1$ to find 
\begin{equation}\label{eq:monotonicity3}
(1+\frac{C}{q}\tau)f(\tau) - (1+\frac{C}{q}\sigma)f(\sigma) \geq - \frac{1}{q} \int_\sigma^\tau \bigg| \rho^{-d} \int_{\overline{\Omega}} 
(\mathrm{Id} - s \otimes s) : \nabla \Phi_{x_0,\rho} \,d\mu \bigg| .
\end{equation}
The same bound trivially holds, over intervals $(\sigma, \tau)$ where $f(\rho) = 1$, and consequently telescoping, we have the montonicity formula (\ref{eq:monotonicity3}) for all $0<\sigma <\tau \ll 1.$

\textit{Step 3: Local trace estimate for the chemical potential.}
We first post-process the preliminary monotonicity formula~\eqref{eq:monotonicity3}
by estimating the associated second right-hand side term involving the first variation.
To this end, we recall for instance from~\cite[p.\ 147]{Kagaya2017} that
the test vector field~$\Phi_{x_0,\rho}$ from~\eqref{eq:testVectorMonotonicity} 
is tangential along~$\p\Omega$. In particular, it represents an admissible choice 
for testing~\eqref{eq:GibbsThomsonAux}:
\begin{align*}
\int_{\overline{\Omega}} (\mathrm{Id} - s \otimes s) : \nabla \Phi_{x_0,\rho} \,d\mu
= \int_{\Omega} \chi \big( w(\nabla \cdot \Phi_{x_0,\rho}) + \Phi_{x_0,\rho} \cdot \nabla w \big) \,dx.
\end{align*}

We distinguish between two cases. If $\rho < \dist(x_0,\p\Omega)$, then
$\tilde\eta_{x_0,\rho} \equiv 0$ and by plugging in~\eqref{eq:testVectorMonotonicity} 
as well as the bounds for~$\eta$
\begin{align*}
\bigg|\rho^{-d}\int_{\Omega} \chi \big( w(\nabla \cdot \Phi_{x_0,\rho}) 
+ \Phi_{x_0,\rho} \cdot \nabla w \big) \,dx\bigg|
\leq C\rho^{-d} \int_{\Omega \cap B_{\rho}(x_0)} \rho|\nabla w| + |w| \,dx. 
\end{align*}
If instead $\rho \geq \dist(x_0,\p\Omega)$, straightforward arguments show
\begin{align*}
|\tilde\eta_{x_0,\rho} \iota_x (\tilde{x} - x_0)| &\leq C\rho,
\\
|\nabla\tilde\eta_{x_0,\rho}| |\iota_{x}(\tilde x {-} x_0)| &\leq C,
\\
|\nabla \iota_x(\tilde x {-} x_0)| &\leq C,
\\
\supp \tilde\eta_{x_0,\rho} \subset \tilde B_{\rho}(x_0) &\subset B_{5\rho}(x_0),
\end{align*}
and therefore for $\rho \geq \dist(x_0,\p\Omega)$
\begin{align*}
\bigg|\rho^{-d}\int_{\Omega} \chi \big( w(\nabla \cdot \Phi_{x_0,\rho}) 
+ \Phi_{x_0,\rho} \cdot \nabla w \big) \,dx\bigg|
\leq C\rho^{-d} \int_{\Omega \cap B_{5\rho}(x_0)} \rho|\nabla w| + |w| \,dx. 
\end{align*}

To control the right-hand side of the display we argue for dimension $d=3$ and note that embeddings are stronger in dimension $d=2$ (see also~\cite[Proof of Lemma~2.1]{Schaetzle2001}).
Extending~$w$ in $H^{1}(\Omega)$ to a function
in $H^{1}(\{x\colon\dist(x,\p\Omega)<2r\}\cup \Omega)$, we let $2^* = 6$ be the dimension dependent Sobolev exponent and apply H\"older's inequality to find
\begin{equation}\nonumber
\begin{aligned}
\rho^{-d} \int_{\Omega \cap B_{5\rho}(x_0)} \rho|\nabla w| + |w| \,dx& \leq \rho^{1-d/2} \|\nabla w\|_{L^2(B_{5\rho}(x_0))} + \rho^{-d/2^*} \|w\|_{L^{2^{*}}(B_{5\rho}(x_0))} \\
&\leq C\rho^{-1/2}\|w\|_{H^1(\Omega)},
\end{aligned}
\end{equation}
as the exponents on $\rho$ coincide.

Thus for all $0<\rho<\frac{r}{5}$ due to the previous case study and estimate
above
\begin{align}
\label{eq:curvatureEstimate}
\bigg| \rho^{-d} \int_{\overline{\Omega}} (\mathrm{Id} - s \otimes s) : \nabla \Phi_{x_0,\rho} \,d\mu \bigg|
\leq C\|w\|_{H^{1}(\Omega)}\rho^{\beta-1}
\end{align}
for some $\beta \in (0,1)$ (accounting also for $d=2$). Inserting~\eqref{eq:curvatureEstimate} back into~\eqref{eq:monotonicity3} finally yields that the function
\begin{align*}
\rho &\mapsto (1+\frac{C}{q}\rho)\max\Big\{\big(\rho^{-(d-1)} I_{x_0}(\rho)\big)^\frac{1}{q}, 1\Big\}+ C\beta^{-1}\|w\|_{H^{1}(\Omega)}\rho^{\beta}
\end{align*}
is nondecreasing in $(0,\frac{r}{5})$. In particular, since $\eta\equiv 1$ on $[0,\frac{1}{2}]$, we obtain one-sided Alhfor's regularity for the varifold as
\begin{equation}\label{eqn:upperAhlfors}
\begin{aligned}
\sup_{x_0 \in \overline{\Omega}:\,\dist(x_0,\p\Omega)<r}
&\sup_{0<\rho<\frac{r}{5}} \Big(\frac{\rho}{2}\Big)^{-(d-1)}
|\mu|_{\mathbb{S}^{d-1}}\big(\overline{\Omega}\cap B_{\frac{\rho}{2}}(x_0)\big)
\\&
\leq C_{q,r,d}\big(1{+}\max\big\{|\mu|_{\mathbb{S}^{d-1}}(\overline{\Omega}),
\|w\|_{H^{1}(\Omega)}^q\big\}\big)
\end{aligned}
\end{equation}
for some $C_{q,r,d} \geq 1$ for all $q \geq d$. The estimate (\ref{eqn:upperAhlfors}) is sufficient to apply the trace theory (as in \cite{meyersZiemer1977}) for the BV function $|w|^s$, and the asserted local
estimate for the $L^s$-norm of the trace of the potential~$w$
on~$\supp|\mu|_{\mathbb{S}^{d-1}}$ now follows as 
in Sch\"atzle~\cite[Proof of Theorem~1.3]{Schaetzle2001}. 
\end{proof}

\begin{proof}[Proof of Corollary~\ref{cor:boundedFirstVariationGlobally}]
In view of the Gibbs--Thomson law~\eqref{eq:GibbsThomsonAux}
and by defining the Radon--Nikod\`ym derivative $\rho^{\Omega} := 
\smash{\frac{c_0|\nabla\chi|\llcorner\Omega}
{|\mu|_{\mathbb{S}^{d-1}}\llcorner\Omega}}\in [0,1]$, we recall the fact
that $H^\Omega:=\rho^{\Omega} \smash{\frac{w}{c_0}} 
\smash{\frac{\nabla\chi}{|\nabla\chi|}} \in L^1(\overline{\Omega},d|\mu|_{\mathbb{S}^{d-1}})$ 
represents the generalized mean curvature
vector of~$\mu$ with respect to tangential variations, i.e.,
\begin{align*}
\delta \mu(B) = \int_{\overline{\Omega} {\times} \Sph}
(\mathrm{Id} {-} s\otimes s) : \nabla B \,d\mu
= - \int_{\overline{\Omega}} H^\Omega \cdot B \,d|\mu|_{\mathbb{S}^{d-1}}
\end{align*}
for all $B \in C^1(\overline{\Omega})$, 
$(B \cdot n_{\partial\Omega})|_{\partial\Omega} \equiv 0$.
A recent result of De~Masi~\cite[Theorem~1.1]{DeMasi2021} therefore ensures 
that there exists $H^{\partial\Omega} \in L^\infty(\partial\Omega,d|\mu|_{\mathbb{S}^{d-1}})$
with the property that $H^{\partial\Omega}(x) \perp \mathrm{Tan}_x\partial\Omega$
for $|\mu|_{\mathbb{S}^{d-1}}\llcorner\partial\Omega\text{-a.e.\ } x \in \overline{\Omega}$,
and a bounded Radon measure $\sigma_{\mu} \in \mathrm{M}(\partial\Omega)$
such that
\begin{align*}
\delta \mu(B) = - \int_{\overline{\Omega}} (H^\Omega {+} H^{\partial\Omega}) \cdot B \,d|\mu|_{\mathbb{S}^{d-1}}
+ \int_{\partial\Omega} B \cdot n_{\partial\Omega} \,d\sigma_{\mu}
\end{align*}
for all ``normal variations'' $B \in C^1(\overline{\Omega})$ in the sense that
$B(x) \perp \mathrm{Tan}_x\partial\Omega$ for all $x \in \partial\Omega$.
There moreover exists a constant $C=C(\Omega)>0$ (depending only on the second fundamental
form of the domain boundary~$\partial\Omega$) such that
\begin{align*}
\|H^{\partial\Omega}\|_{L^\infty(\partial\Omega,d|\mu|_{\mathbb{S}^{d-1}})} &\leq C,
\\
\sigma_{\mu}(\partial\Omega) &\leq C|\mu|_{\mathbb{S}^{d-1}}(\overline{\Omega}) 
+ \|H^\Omega\|_{L^1(\overline{\Omega},d|\mu|_{\mathbb{S}^{d-1}})}.
\end{align*}
In particular, by a splitting argument into ``tangential''
and ``normal components'' of a general variation $B \in C^1(\overline{\Omega})$,
we deduce that the varifold~$\mu$ is of bounded first variation in~$\overline{\Omega}$ 
with representation~\eqref{eq:repFirstVariationGlobally}.
The asserted bounds~\eqref{eq:boundedFirstVariationGlobally}--\eqref{eq:estimateBoundaryMeasure}
are finally consequences of the two bounds from the previous
display, the representation of the first variation from~\eqref{eq:repFirstVariationGlobally},
and the definition of~$H^\Omega$.
\end{proof}

\section*{Acknowledgements}
The authors thank Helmut Abels for pointing out his paper \cite{AbelsRoeger} to them and Olli Saari for insightful discussion. 
This project has received funding from the Deutsche Forschungsgemeinschaft 
(DFG, German Research Foundation) under Germany's Excellence 
Strategy -- EXC-2047/1 -- 390685813.

\bibliographystyle{abbrv}
\bibliography{mullins_sekerka_existence}

\end{document}